  \documentclass[12pt]{amsart}
\address{Max-Planck-Institut f\"ur Mathematik in den Naturwissenschaften. Inselstraße 22, 04103 Leipzig, Germany.}

  \email{vasirog[at]gmail.com}

\usepackage[utf8]{inputenc}
\usepackage[T2A]{fontenc}
\usepackage{enumerate, amssymb, amsmath, amsthm}
\usepackage{amscd}
\usepackage{graphicx}

\usepackage[colorlinks=true,citecolor=blue, urlcolor=blue, linkcolor=blue, pagebackref]{hyperref}
\renewcommand*{\backref}[1]{}
\renewcommand*{\backrefalt}[4]{%
  \ifcase #1
    Not cited.%
  \or
    Cited on page~#2.%
  \else
    Cited on pages #2.%
  \fi%
}
\usepackage[a4paper,hmargin=2.5cm,vmargin=2.5cm]{geometry}
\usepackage{marginnote}
\usepackage[english]{babel}
\graphicspath{ {images/} }
\usepackage{wrapfig}
\usepackage[matrix,arrow,curve]{xy}\usepackage{mathabx}
\usepackage[OT2, T1]{fontenc}
\usepackage{appendix}
\usepackage{chngcntr}
\usepackage{etoolbox}
\usepackage{lipsum}

\DeclareSymbolFont{cyrletters}{OT2}{wncyr}{m}{n}
\DeclareMathSymbol{\Sha}{\mathalpha}{cyrletters}{"58}

\usepackage[a4paper,hmargin=2.5cm,vmargin=2.5cm]{geometry}
\usepackage[english]{babel}
\graphicspath{ {images/} }
\usepackage{wrapfig}
\usepackage[matrix,arrow,curve]{xy}\usepackage{mathabx}

\begin{document}
\renewcommand{\refname}{Bibliography}
\newtheorem{prop}{Proposition}[section]
\newtheorem{thrm}[prop]{Theorem}
\newtheorem{lemma}[prop]{Lemma}
\newtheorem{cor}[prop]{Corollary}
\newtheorem{mainthm}{Theorem}
\newtheorem{maincor}[mainthm]{Corollary}
\theoremstyle{definition}
\newtheorem{df}{Definition}
\newtheorem{ex}{Example}
\newtheorem{rmk}{Remark}
\newtheorem{conj}{Conjecture}
\newtheorem{cl}{Claim}
\newtheorem{q}{Question}
\newtheorem{constr}{Construction}
\renewcommand{\proofname}{\textnormal{\textbf{Proof:  }}}
\renewcommand{\refname}{Bibliography}
\renewcommand{\themainthm}{\Alph{mainthm}}
\renewcommand{\themaincor}{\Alph{maincor}}

\renewcommand{\phi}{\varphi}
\renewcommand{\epsilon}{\varepsilon}

\renewcommand{\C}{\mathbb C}
\newcommand{\Z}{\mathbb Z}
\newcommand{\Q}{\mathbb Q}
\newcommand{\R}{\mathbb R}
\newcommand{\N}{\mathbb N}
\newcommand{\Fp}{\mathbb{F}_p}
\newcommand{\Fq}{\mathbb{F}_q}

\renewcommand{\O}{\mathcal O}
\newcommand{\g}{\mathfrak g}
\newcommand{\h}{\mathfrak h}
\newcommand{\E}{\mathcal E}
\newcommand{\F}{\mathcal F}
\newcommand{\m}{\mathfrak{m}}

\renewcommand{\i}{\sqrt{-1}}
\renewcommand{\o}{\otimes}
\newcommand{\di}{\partial}
\newcommand{\acts}{\lefttorightarrow}
\newcommand{\dibar}{\overline{\partial}}
\newcommand{\im}{\operatorname{im}}
\renewcommand{\ker}{\operatorname{ker}}
\newcommand{\Hom}{\operatorname{Hom}}
\newcommand{\tr}{\operatorname{tr}}
\newcommand{\codim}{\operatorname{codim}}
\newcommand{\rk}{\operatorname{rk}}
\newcommand{\nilp}{\operatorname{nilp}}
\newcommand{\hdot}{{\:\raisebox{3pt}{\text{\circle*{1.5}}}}}
\newcommand{\Supp}{\operatorname{Supp}}
\newcommand{\Alb}{\operatorname{Alb}}
\newcommand{\alb}{\operatorname{alb}}
\newcommand{\Hilb}{\operatorname{Hilb}}
\newcommand{\Sh}{\operatorname{Sh}}
\newcommand{\sh}{\operatorname{sh}}
\newcommand{\CP}{\mathbb{C}\mathbf{P}}
\newcommand{\Isom}{\operatorname{Isom}}
\newcommand{\Sym}{\operatorname{Sym}}
\newcommand{\Stab}{\operatorname{Stab}}
\newcommand{\Aut}{\operatorname{Aut}}
\newcommand{\sslash}{\mathbin{/\mkern-6mu/}}
\newcommand{\Pic}{\operatorname{Pic}}
\newcommand{\V}{\mathbb{V}}
\newcommand{\alg}{\operatorname{alg}}
\newcommand{\Ext}{\operatorname{Ext}}
\newcommand{\MHS}{\operatorname{MHS}}
\newcommand{\HS}{\operatorname{HS}}
\newcommand{\Gr}{\operatorname{Gr}}
\newcommand{\End}{\operatorname{End}}
\newcommand{\odef}{\operatorname{def}}

\newcommand{\GL}{\operatorname{GL}}
\newcommand{\SL}{\operatorname{SL}}
\newcommand{\SU}{\operatorname{SU}}
\renewcommand{\U}{\operatorname{U}}
\newcommand{\SO}{\operatorname{SO}}
\newcommand{\Ogr}{\operatorname{O}}
\newcommand{\Sp}{\operatorname{Sp}}

\newcommand{\gl}{\mathfrak{gl}}
\renewcommand{\sl}{\mathfrak{sl}}
\newcommand{\su}{\mathfrak{su}}
\renewcommand{\u}{\mathfrak{u}}
\newcommand{\so}{\mathfrak{so}}
\renewcommand{\sp}{\mathfrak{sp}}
\renewcommand{\g}{\mathfrak{g}}
\renewcommand{\h}{\mathfrak{h}}
\newcommand{\z}{\mathfrak{z}}
\newcommand{\ad}{\operatorname{ad}}
\newcommand{\cd}{\operatorname{cd}}
\newcommand{\an}{\operatorname{an}}
\newcommand{\orb}{\operatorname{orb}}
\newcommand{\Trop}{\operatorname{Trop}}
\newcommand{\Hdg}{\operatorname{Hdg}}
\newcommand{\T}{\mathbb{T}}
\newcommand{\Uu}{\mathcal{U}}
\newcommand{\Zz}{\mathcal{Z}}
\newcommand{\D}{\mathcal{D}}
\newcommand{\M}{\mathcal{M}}
\newcommand{\K}{\mathbb{K}}
\renewcommand{\V}{\mathbb{V}}
\newcommand{\G}{\mathcal{G}}
\newcommand{\Ppic}{\mathbf{Pic}}
\title{o-minimal geometry of higher Albanese manifolds}

\author{Vasily Rogov}

\begin{abstract}
Let X be a normal quasi-projective variety over $\C$. We study its higher Albanese manifolds, introduced by Hain and Zucker, from the point of view of o-minimal geometry. We show that for each $s$ the higher Albanese manifold $\Alb^s(X)$ can be functorially endowed with a structure of an $\R_{\alg}$-definable complex manifold in such a way that the natural projections $\Alb^s(X) \to \Alb^{s-1}(X)$ are $\R_{\alg}$-definable and the higher Albanese maps $\alb^s \colon X^{\an} \to \Alb^s(X)$ are $\R_{\an, \exp}$-definable.

Suppose that for some $s \ge 3$ the definable manifold $\Alb^s(X)$ is definably biholomorphic to a quasi-projective variety. We show that in this case the higher Albanese tower stabilises at the second step, that is, the maps $\Alb^r (X) \to \Alb^{r-1}(X)$ are isomorphisms for $r\ge 3$. It follows that if $\alb^s \colon X^{\an} \to \Alb^s(X)$ is dominant for some $s \ge 3$, then the higher Albanese tower stabilises at the second step and the pro-unipotent completion of $\pi_1(X)$ is at most 2-step nilpotent. This confirms a special case of a conjecture by Campana on nilpotent fundamental groups of algebraic varieties.

As another application, we construct explicit models for nilpotent Shafarevich reductions.
\end{abstract}

\dedicatory{In memory of Tobias Kreutz}

\maketitle

\tableofcontents

\section{Introduction}

\subsection{Higher Albanese manifolds}

Let $X$ be a normal quasi-projective variety over $\C$. Classically, one associates with it a semiabelian variety, known as \emph{Albanese manifold} $\Alb(X)$, and an algebraic map (the \emph{Albanese map}) $\alb \colon X \to \Alb(X)$. Topologically, it is related to $X$ by the canonical isomorphism
\[
\pi_1(\Alb(X)) \simeq \pi_1(X)^{\operatorname{ab}}/\operatorname{torsion}
\]
and the complex structure on $\Alb(X)$ is determined by the mixed Hodge structure on $H_1(X, \Z)$. Therefore, the Albanese manifold can be viewed as the geometric incarnation of the degree $1$ part of the Hodge theory of $X$ (or, more vaguely, as the shadow of $X$ in the world of $1$-motives). The theory of higher Albanese manifolds is developed in \cite{HZ} and allows one to extend  this construction from $H_1(X, \Z)$ to the nilpotent quotients of $\pi_1(X)$.

For a finitely generated group, $\Gamma$ we denote by $\Gamma_s$ its lower central series and by $\Gamma^s:=\Gamma/\Gamma_s$ its universal nilpotent quotients. For every natural number $s$, there exists a connected unipotent $\Q$-algebraic group $\G^s_{\Q}$ and a representation $\mu^s \colon \Gamma \to \G^s_{\Q}$  such that every Zariski dense $s$-step unipotent representation of $\Gamma$ over $\Q$ factorises through $\mu^s$ (these are the nilpotent quotients of the Malcev completion of $\Gamma$). The image $\G_{\Z}^{s}:=\mu^s(\Gamma)$ is a discrete Zariski dense subgroup of $\G^s_{\Q}(\Q)$ and is isomorphic to $\Gamma^s$ modulo torsion.

Suppose now that $\Gamma$ is the fundamental group of a normal complex quasi-projective variety $X$. In this case, the Lie algebra $\g^s=\operatorname{Lie}(\G^s_{\Q}(\Gamma))$ carries a functorial mixed Hodge structure $(W_{\bullet}\g^s, F^{\bullet}\g^s)$ consistent with the Lie bracket. In particular, $[F^p\g^s, F^q\g^s] \subseteq F^{p+q}\g^s$ and $F^0\G^s:=\exp(F^0\g^s)$ is a closed subgroup of $\G^s_{\Q}(\C)$. The $s$-th Albanese manifold of $X$ is defined as
 \[
 \Alb^s(X):=\G^s_{\Z} \backslash \G^s_{\Q}(\C)/F^0\G^s.
 \]
 This is always a smooth complex manifold; its fundamental group is isomorphic to $\G^s_{\Z}$, and  its universal cover is biholomorphic to $\C^d$. A morphism of normal varieties $f \colon X \to Y$ induces a holomorphic map $\Alb^s(f) \colon \Alb^s(X) \to \Alb^s(Y)$. Higher Albanese manifolds are related  to each other by  holomorphic projections $p^s \colon \Alb^s(X) \to \Alb^{s-1}(X)$, and for $s=1$ the construction recovers the classical Albanese manifold.

 Hain and Zucker constructed (\cite{HZ}) holomorphic maps $\alb^s \colon X^{\an} \to \Alb^s(X)$ that generalise the classical Albanese map and, together with the projections $p^s$, form a commutative diagram
\begin{equation}\label{diagram}
\xymatrix{
& \vdots \ar[d]^{p^{s+1}}\\
&\Alb^s(X) \ar[d]^{p^s} \\
& \vdots \ar[d]^{p^3}\\
& \Alb^2(X) \ar[d]^{p^2}\\
X^{\an} \ar[r]_{\alb} \ar[ru]^{\alb^2} \ar[ruuu]^{\alb^s}& \Alb(X).
}
\end{equation}

The induced homomorphisms $\alb^s_{*} \colon \pi_1(X) \to \pi_1(\Alb^s(X))=\G^s_{\Z}$ coincide with the canonical maps $\mu^s \colon \pi_1(X) \to \G^s_{\Z}=(\pi_1(X)))^s/(\operatorname{torsion})$. Higher Albanese maps share the following universal property: every period map of an $s$-step unipotent admissible polarisable variation of mixed $\Z$-Hodge structures on $X$ factorises through $\alb^s$ (\cite[Section 5]{HZ}).

While the classical Albanese map ($s=1$) is a morphism of algebraic varieties, the Hain-Zucker construction is a priori of transcendental nature. In this paper, we address the following question: How far is the diagram (\ref{diagram}) from being a diagram of algebraic varieties?

The answer turns out to be ambiguous. On the one hand, we show that the diagram (\ref{diagram}) fits perfectly  within the framework of o-minimal geometry, implying that the behaviour of the presenting maps is \emph{tame} in a certain precise sense.  On the other hand, it turns out that (\ref{diagram}) can almost never be realised as the analytification of a diagram of complex algebraic varieties. 

\subsection{Main results}

First, we show that higher Albanese manifolds and higher Albanese maps are definable in an o-minimal structure.  We refer the reader to Section \ref{o-minimal sec} for a reminder on o-minimal structures and definable complex analytic geometry.

\begin{mainthm}\label{definabilisation main}[Theorem \ref{definabilisation}]
Let $X$ be a complex normal quasi-projective variety. For every $s \ge 1$ the higher Albanese manifold $\Alb^s(X)$ can be endowed with a structure of an $\R_{\alg}$-definable complex manifold in such a way that
\begin{itemize}
\item[(i)] the projections $p_s \colon \Alb^s(X) \to \Alb^{s-1}(X)$ are definable;
\item[(ii)] for each $s$ there exists a definable commutative connected complex Lie group $C^s$ such that $p_s \colon \Alb^s(X) \to \Alb^{s-1}(X)$ is a definable holomorphic principal $C^s$-bundle and the action $C^s \times \Alb^s(X) \to \Alb^s(X)$ is definable. Each $C^s$ is abstractly isomorphic (as a complex Lie group) to the Jacobian of a mixed Hodge structure.
\item[(iii)] the higher Albanese maps  $\alb^s \colon X^{\an} \to \Alb^s(X)$ are $\R_{\an, \exp}$-definable;
\item[(iv)] if $s=1$, the resulting definable structure on $\Alb^1(X)=\Alb(X)$ is the same as the one determined by the standard algebraic structure on $\Alb(X)$;
\item[(v)] if $f \colon X \to Y$ is a morphism of normal  varieties, $\Alb(f) \colon \Alb^s(X) \to \Alb^s(Y)$ is definable.
\end{itemize}
Moreover,
\begin{itemize}
\item[(vi)] the reduced image $\alb^s(X)^{\operatorname{red}}$ is the definable analytification of a quasi-projective variety and $\alb^s \colon X \to \alb^s(X)^{\operatorname{red}}$ is the analytification of an algebraic morphism.
\end{itemize}
\end{mainthm}

\begin{rmk}
In some special cases, for example when $X=\C\setminus\{0,1\}$, the higher Albanese maps $\alb^s$ are known to be closely related to polylogarithms and their analogues \cite{HMcP, Del87, Us}. Perhaps one can get new results on the transcendence of special values of polylogarithms using Theorem \ref{definabilisation main} and the Pila-Wilkie Transcendence Theorem \cite{PW}.
\end{rmk}

The definability of higher Albanese manifolds and higher Albanese maps follows from their relation to the period maps of certain admissible variations of mixed Hodge structures. 

Although the close relation of higher Albanese manifolds to mixed period domains is well known \cite{HZ, KNU16, Hast}, the explicit description of higher Albanese manifolds in the spirit of Pink-Klingler  formalism of mixed Hodge varieties \cite{Kling}, as well as its relation to o-minimal geometry, seems to be absent from the literature.

In order to clarify this relation, we introduce the notion of a \emph{nil-Jacobian} that generalises higher Albanese manifolds and interpolates between mixed Hodge varieties and Jacobians of mixed Hodge structures. A nil-Jacobian is a double coset of the form $\Gamma_W \backslash \mathbf{W}(\C) /F^0\mathbf{W}$ associated to the following data: 
\begin{itemize}
\item a finite-dimensional unipotent Lie algebra in the category of $\Q$-mixed Hodge structures $(\mathfrak{w}, W_{\bullet}\mathfrak{w}, F^{\bullet}\mathfrak{w}_{\C})$ with weights concentrated in negative degrees (i.e. $W_{-1}\mathfrak{w}=\mathfrak{w}$); to it one associates the underlying unipotent $\Q$-algebraic group $\mathbf{W}$ and a closed connected subgroup $F^0\mathbf{W} =\exp(F^0\mathfrak{w}_{\C}) \subseteq \mathbf{W}(\C)$;
\item a discrete Zariski dense subgroup $\Gamma_W \subset \mathbf{W}(\Q)$.
\end{itemize}
We show that every nil-Jacobian can be realised as a definable closed subset of a mixed Hodge variety (we refer to Section \ref{hodge basic sec} for a reminder on mixed Hodge varieties and to subsection \ref{o-minimal hodge} for definable structures on mixed Hodge varieties). Moreover, the inherited structure of a definable complex manifold depends only on the nil-Jacobian, but not on the embedding to a mixed Hodge variety. Since higher Albanese varieties are nil-Jacobians, this allows us to endow them with structures of definable complex manifolds. This is the key step in the proof of Theorem \ref{definabilisation main}.

As another application of the theory of nil-Jacobians, we construct \emph{partial higher Albanese manifolds} $\Alb_{\rho}^s(X)$. They can be thought of as higher analogues of \emph{partial Albanese maps} $\alb_{\theta} \colon X \to \Alb_{\theta}$ associated with a character $\theta \in H^1(X, \C)$. Using them, we are able to give explicit descriptions of nilpotent Shafarevich reductions, see subsection \ref{shafarevich subsection}.

Our second main result is the following.

\begin{mainthm}\label{no algebraic main}[Theorem \ref{no algebraic}]
Let $X$ be a normal quasi-projective variety and $s \ge 3$ a natural number. Suppose that one of the following holds:
\begin{itemize}
\item[(i)] $\alb^s \colon X \to \Alb^s(X)$ is dominant;
\item[(ii)] $\Alb^s(X)$ is definably biholomorphic to the definable analytification of a quasi-projective variety.
\end{itemize}
Then the map $p^r \colon \Alb^r(X) \to \Alb^{r-1}(X)$ is a principal $(\C^{\times})^k$-bundle if $r=2$ and is an isomorphism for $r>2$. In particular, the tower of higher Albanese manifolds stabilises at the second step and the Malcev completion of $\pi_1(X)$ is $2$-step nilpotent.
\end{mainthm}

Hain and Zucker showed a homotopy version of item \textit{(ii)} of Theorem \ref{no algebraic main} under stronger assumptions: $s \ge 4$ and $\pi_1(X)$  not rationally nilpotent (\cite[Theorem 5.43]{HZ}; see also Remark 5.45 \textit{loc.cit}). Their proof is based on rational homotopy theory. Our approach uses different (and, in some sense, more elementary) methods, namely results on the topology and algebraic geometry of principal bundles with abelian structure groups.

Theorem \ref{no algebraic main} is closely related to the long-standing problem of understanding nilpotent groups that arise as fundamental groups of smooth complex algebraic varieties. There are strong restrictions on the structure of such groups coming from Hodge theory \cite{CT, Camp}. Examples of smooth projective varieties with $2$-step nilpotent non-abelian fundamental group were constructed by Sommese and Van de Ven, and later by Campana \cite{SVdV, Camp}. So far, no essentially different new examples have been found, which motivates the following conjecture:

\begin{conj}[F. Campana]\label{campana conjecture}
Let $X$ be a complex normal quasi-projective variety. Suppose that $\pi_1(X)$ is virtually nilpotent. Then it is  virtually  at most $2$-step nilpotent.
\end{conj}

Theorem \ref{no algebraic main} implies Conjecture \ref{campana conjecture} for varieties with dominant higher Albanese maps. It can also be seen as a generalisation of a theorem of Aguilar Aguilar and Campana (\cite{AC}) that says that if $X$ is a normal quasi-projective variety whose (classical) Albanese map $\alb \colon X \to \Alb(X)$ is surjective and \emph{proper}, then Malcev completion of $\pi_1(X)$ is abelian. Notice that we do not assume any properness of $\alb^s$ in Theorem \ref{no algebraic main}.

Let us explain the heuristic behind the proof of Theorem \ref{no algebraic main}.

While the projective examples of varieties with nilpotent non-abelian fundamental groups in \cite{SVdV} and \cite{Camp} are rather involved, it is relatively easy to construct a quasi-projective variety with such a property.

Let $X_1$ be an abelian variety and $L$ a holomorphic line bundle on it with $c_1(L) \neq 0$. Let $X_2:=\operatorname{Tot}(L) \setminus L_0$, where $L_0$ is the zero section. Then $\pi_1(X_2)$ is the central extension
\[1 \to \pi_1(\C^{\times})=\Z \to \pi_1(X_2) \to \pi_1(X_1) \to 1,
\] and $c_1(L) \neq 0$ guarantees  that this extension is non-trivial (see \cite[Example 11.26]{CDY} and Lemma \ref{topology} below). Thus, $\pi_1(X_2)$ is non-abelian and $2$-step nilpotent.

But can we construct examples of quasi-projective varieties with nilpotent but not $2$-step nilpotent fundamental group? The first thing that comes to mind is to upgrade the example above by considering a principal bundle $p \colon X_3 \to X_2$ with fibres having abelian fundamental group. If one wants $\pi_1(X_3)$ to be not $2$-step nilpotent, one should require at least that such a bundle $p \colon X_3 \to X_2$ is topologically non-trivial. After a short analysis of cases, one is essentially reduced to one of the  two following situations:
\begin{itemize}
\item[(i)] $p$ is a principal $A$-bundle, where $A$ is an abelian variety; 
\item[(ii)] $p$ is a principal $(\C^{\times})^k$-bundle. 
\end{itemize}
The case \textit{(i)} is then ruled out by a theorem of Blanchard (Theorem \ref{blanchard}), that suggests that either $p$ is topologically trivial, or $X_3$ is not a K\"ahler (or even a Fujiki class $\mathcal{C}$) manifold.

Suppose we are in the case \textit{(ii)} . Then, using the fact that $X_2 \to X_1$ is an algebraic fibration with affine fibres one can show that the map $\Pic(X_1) \to \Pic(X_2)$ is surjective\footnote{Here we are talking about \emph{algebraic} Picard groups; for analytic  bundles this is no longer true.}, and, more generally, every algebraic principal $(\C^{\times})^k$-bundle on $X_2$ is a pull-back of a bundle on $X_1$. Therefore, the corresponding class of the extension of $\pi_1(X_2)$ by $\pi_1((\C^{\times})^k)=\Z^k$ pulls back from $\pi_1(X_1)$ and $\pi_1(X_3)$ is again only $2$-step nilpotent (cf. Lemma \ref{toric tower topology}).

Combining these considerations, one concludes that there exists no sequence of algebraic varieties
\[
X_3 \xrightarrow{p_3} X_2 \xrightarrow{p_2} X_1,
\]
such that both $p_3$ and $p_2$ are holomorphic principal bundles and each $\pi_1(X_j)$ is $j$-step (but not $j-1$-step) nilpotent, for $j=1, \ 2, \ 3 $. 

On the other hand, as we show, if $\Alb^s(X)$ is algebraic for some $s \ge 3$, then the  truncated   higher Albanese tower  $\Alb^s(X) \to \Alb^{s-1}(X) \to \ldots \to \Alb(X)$ is algebraic and the lower levels of this tower would provide  a triple of algebraic varieties as above. This leads to a contradiction.

\subsection{Organisation of the paper}

The paper's organization is as follows:

Sections \ref{hodge basic sec} to \ref{higher albanese basic} are preliminary and can be safely bypassed by a specialist. 

Section \ref{hodge basic sec} contains basics on mixed Hodge structures and mixed Hodge varieties and is essentially included for the sake of fixing the notations. 

Subsection \ref{hodge structures subsec} contains a reminder on the approach to mixed Hodge structures via Deligne torus formalism. 

Subsection \ref{hodge varieties subsec} presents a quick overview of the theory of mixed Hodge data and mixed Hodge varieties in the spirit of \cite{Kling}. This level of abstractness is necessary for the proof of the Embedding Theorem in subsection \ref{embedding subsec}; apart from that, the reader who is not comfortable with such formalism can think in terms of a more classical approach to (mixed) period domains, e.g. as in \cite{CKS}.

In subsection \ref{purification subsec} we discuss \emph{purification maps}. These are canonical maps from mixed Hodge varieties to pure Hodge varieties, analogous to passing from a mixed Hodge structure to the direct sum of the associated graded pieces of its weight filtration.

In subsection \ref{about splitting subsec} we recall the $\operatorname{sl}_2$-splitting -- a technical tool from mixed Hodge theory that plays important role in the o-minimal approach to mixed Hodge varieties in \cite{BBKT}.

Subsection \ref{o-minimal sec} contains necessary facts from o-minimal geometry. We are not giving complete and rigorous overview here, referring the reader to the great expositions in \cite{VdD} and \cite{BBT}. Rather, we collect the necessary facts and try to present few motivating examples and vague slogans for the reader not familiar with the topic. We recall the general principles of the o-minimal geometry in subsection \ref{o-minimal basic} and discuss Bakker - Brunebarbe - Tsimerman's o-minimal complex analytic geometry in subsection \ref{complex o-minimal}. In the next two subsections we collect the main applications of o-minimality in complex geometry: algebraisation results (subsection \ref{algebraisation results}) and the definability of period maps (subsection \ref{o-minimal hodge}).

In Section \ref{higher albanese basic} we introduce higher Albanese manifolds. We recall generalities on nilpotent groups and Malcev completions in subsection \ref{malcev}. We discuss Morgan - Hain mixed Hodge structure on Malcev completion of $\pi_1(X)$ in subsection \ref{morgan hain subsec} and Hain - Zucker theory of higher Albanese manifolds in subsection \ref{higher albanese subsec}.

In Section \ref{niljacobians section} we develop the theory of nil-Jacobians. We define nil-Jacobians and morphisms thereof and discuss their elementary properties in subsection \ref{niljacobians basic}. In subsection \ref{embedding subsec} we prove the Embedding Theorem (Theorem \ref{embedding theorem}) that says that every nil-Jacobian admits an embedding to a mixed Hodge variety. In subsection \ref{definability of nil-Jacobians} we show that every nil-Jacobian can be endowed with a canonical $\R_{\alg}$-definable complex manifold structure in such a way that morphisms of nil-Jacobians are definable.

In Section \ref{main sec} we prove Theorem \ref{definabilisation main}. We prove the definability of higher Albanese manifolds and higher Albanese maps in subsection \ref{main subsec}; we also deduce some consequences of surjectivity of higher Albanese maps that will be important in Section \ref{algebraic sec}.  In subsection \ref{shafarevich subsection} we discuss applications of our results to nilpotent Shafarevich reductions.

Section \ref{algebraic sec} is dedicated to the proof of Theorem \ref{no algebraic main}. We discuss some general facts about commutative algebraic groups in subsection \ref{commutative groups subsec}. In subsection \ref{chern hoefer subsec} we explain the general approach to study the topology of total spaces of holomorphic principal bundles with commutative structure group. In subsection \ref{blanchard subsec} we recall Blanchard's theorem on holomorphic principal compact torus bundles with K\"ahler total space. In subsection \ref{toric subsec} we obtain results on geometry and topology of total spaces of algebraic toric bundles. We complete the proof of Theorem \ref{no algebraic main} in subsection \ref{algebraic proof subsec}.

We conclude with some conjectures and open questions in Section \ref{conclusion}.
\\

\textbf{Acknowledgements.} I am thankful to Jacques Audibert, Yohanne Brunebarbe, Richard Hain and Bruno Klingler for fruitful conversations on various parts of this work. I am thankful to anonymous referee for their valuable comments. I am also thankful to Max Planck Institute of Mathematics in Science in Leipzig for the perfect working conditions.
\\

\textbf{Conventions.} All algebraic varieties are assumed to be connected, irreducible and over the complex numbers, unless different is explicitly stated. If $X$ is a variety and $x \in X(\C)$ a closed point, we write $\pi_1(X; x)$ for its topological fundamental group $\pi_1^{\operatorname{top}}(X^{\an}_{\C}; x)$. We omit the marked point from the notation when its choice is not important.

If $V$ is a module over a ring $K$ and $K \subseteq L$ is a ring extension, we write $V_{L}:=V_K \o L$.

Throughout this paper, we sometimes work simultaneously in the category of algebraic spaces and the category of (definable) complex analytic spaces. Whenever this happens, we denote the algebraic spaces and morphisms between them by fraktur letters ($\mathfrak{A}, \mathfrak{B}, \mathfrak{C}, \ldots, \mathfrak{X}, \mathfrak{Y}, \mathfrak{Z}, \ldots, \mathfrak{f}, \mathfrak{g}, \mathfrak{h},\ldots)$ and the analytic or definable spaces and holomorphic maps between them by the regular font ($A, B, C, \ldots X, Y, Z, \ldots, f,g,h \ldots)$.

\section{Preliminaries from Hodge theory}\label{hodge basic sec}

\subsection{Mixed Hodge structures}\label{hodge structures subsec}
We briefly recall the theory of mixed Hodge structures following the Deligne torus formalism. The main references for this and the next  subsections are \cite{Kling} and \cite{Pink}, see also \cite{BBKT}.

The \emph{Deligne torus} is the group $\mathbb{S}:=\operatorname{Res}_{\C/\R} \C^{\times}$. As a complex algebraic group, $\mathbb{S}(\C) \simeq \C^{\times} \times \C^{\times}$, but the real structure is non-standard: $\mathbb{S}(\R)$ is embedded into $\mathbb{S}(\C)$ as $\{(z, \overline{z})\} \subset \C^{\times} \times \C^{\times}$.

The datum of a Hodge structure is the same as the datum of an $\mathbb{S}$-module. More explicitly, let $V$ be a finite-dimensional $\Q$-vector space. Let $w$ denote the morphism of $\R$-algebraic groups $w \colon \mathbb{G}_{m, \R} \to \mathbb{S}$ given on the real points by the embedding $\R^{\times} \to \C^{\times}$. Let $h \colon \mathbb{S} \to \GL(V \o \R)$ be a representation such that the composition 
\[
\mathbb{G}_{m}(\R) \xrightarrow{w} \mathbb{S} \xrightarrow{h} \GL(V \o \R)
\]
is of the form $t \mapsto t^{-n}\operatorname{Id}$. This equips $V$ with a weight $n$ pure Hodge structure: the  Hodge decomposition $V_{\C} = \bigoplus_{p+q=n} V^{p,q}$ is given by the decomposition of the representation $h_{\C} \colon \mathbb{S}(\C) \to \GL(V_{\C})$ into isotypic components. The action of $\mathbb{S}(\C)=\C^{\times} \times \C^{\times}$ on $V^{p,q}$ is given by $(z,w) \mapsto z^{-p}w^{-q}$.

Vice versa, any pure  $\Q$-Hodge structure of weight $n$ is obtained this way, and  morphisms of Hodge structures are precisely $\mathbb{S}$-equivariant $\Q$-linear maps.

The extension of the correspondence between Hodge structures and  Deligne torus representations to the mixed case is based on the existence of the so-called \emph{Deligne splitting}.

\begin{lemma}[Deligne, \cite{Del94}]\label{Deligne}
Let $(V_{\Q}, W_{\bullet}V_{\Q}, F^{\bullet}V_{\C})$ be a mixed $\Q$-Hodge structure. There exists a functorial splitting $V_{\C}=\bigoplus_{r,s}I^{r,s}$ such that:
\begin{itemize}
\item[(i)] $W_{k}V=\bigoplus_{r+s\le k} I^{r,s}$;
\item[(ii)] $F^pV=\bigoplus_{r \ge p} I^{r,s}$;
\item[(iii)] $I^{r,s} \equiv \overline{I^{s,r}} \operatorname{mod} \bigoplus_{\substack{r'<r \\s'<s}}I^{r',s'}$.
\end{itemize}
\end{lemma}

Therefore, to a mixed $\Q$-Hodge structure $(V, W_{\bullet}V, F^{\bullet}V)$ one can functorially associate a representation $h_{\C} \colon \mathbb{S} \to \GL(V_{\C})^W$ for which the Deligne splitting of $V_{\C}$ is the isotypic decomposition of the $\mathbb{S}$-module $V_{\C}$ (here $\GL(V_{\C})^W$ denotes the subgroup of $\GL(V_{\C})$ preserving the weight filtration $W_{\bullet}V$). The projection $\GL(V_{\C})^W \to \prod_k \GL(\Gr_k^W V_{\C})$ sends $h$ to $\bigoplus_k h_k$, where $h_k$ are representations of $\mathbb{S}$ corresponding to the pure weight $k$ Hodge structures on the graded pieces $W_kV/W_{k-1}V$.

Pink gave a complete description of representations of $\mathbb{S}$ that come from mixed Hodge structures \cite{Pink}. More precisely, he proved the following.

\begin{thrm}[Pink]\label{pink}
Let $\mathbf{G}$ be a connected algebraic group over $\Q$. Denote by $\mathbf{U}$ its unipotent radical, by $\mathbf{H}:=\mathbf{G}/\mathbf{U}$ its reductive quotient, by $\g$ its Lie algebra and by $\u$ the Lie algebra of $\mathbf{U}$. Let $\rho \colon \mathbf{G} \to \GL(V_{\Q})$ be a faithful finite-dimensional representation of $\mathbf{G}$ over $\Q$ and $h \colon \mathbb{S}_{\C} \to \mathbf{G}_{\C}$ a morphism of complex algebraic groups.

Then there exists a unique mixed Hodge structure on $V_{\Q}$ inducing the  representation $\mathbb{S}_{\C} \xrightarrow{h} \mathbf{G}_{\C} \xrightarrow{\rho_{\C}} \GL(V_{\C})$ if and only if the following holds:
\begin{itemize}
\item[(i)] $\mathbb{S}_{\C} \xrightarrow{h} \mathbf{G}_{\C} \to \mathbf{H}_{\C}$ is defined over $\R$;
\item[(ii)] $\mathbb{G}_m \xrightarrow{w} \mathbb{S}_{\C} \xrightarrow{h} \mathbf{G}_{\C} \to \mathbf{H}_{\C}$ is defined over $\Q$.
\end{itemize}
In this case the $\mathbf{G}$-action on $V_{\Q}$ preserves the weight filtration. Moreover, the group $\mathbf{U}$ acts trivially on the associated graded $\bigoplus \Gr^W_k V_{\Q}$ if and only if
\begin{itemize}
\item[(iii)] the composition $\mathbb{S} \xrightarrow{h} \mathbf{G} \xrightarrow{\operatorname{Ad}} \GL(\g)$ endows $\g$ with a rational mixed Hodge structure, such that $W_{-1}\g=\u$.
\end{itemize}
\end{thrm}

Remarkably, the conclusion of Pink's Theorem does not depend on the representation $\rho$, but only on the homomorphism $h$.

\begin{df}
A \emph{Hodge cocharacter} is a morphism $h \colon \mathbb{S}_{\C} \to \mathbf{G}_{\C}$ satisfying conditions \textit{(i)-(iii)} of Theorem \ref{pink}
\end{df}

Let $V=(V_{\Q}, W_{\bullet}V, F^{\bullet}V)$ be a mixed $\Q$-Hodge structure and $V_{\C} = \bigoplus_{r,s} I_V^{r,s}$ its Deligne splitting. The \emph{complex conjugate} mixed Hodge structure is the unique $\Q$-mixed Hodge structure  $\overline{V}$ on $V_{\Q}$ for which  the pieces of its Deligne splitting are $I_{\overline{V}}^{r,s} = \overline{I_{V}^{s,r}}$. We say that a mixed Hodge structure \emph{splits over $\R$} if it is isomorphic to its complex conjugate. If $V$ is a pure Hodge structure, its Deligne splitting coincides with the Hodge splitting, and Hodge duality $V^{p,q}=\overline{V^{q,p}}$ means that every pure Hodge structure automatically splits over $\R$.

Let $\mathbf{G}$ be a connected algebraic group over $\Q$ and $\rho \colon \mathbf{G} \to \GL(V_{\Q})$ a finite dimensional faithful representation. Every Hodge cocharacter $h \colon \mathbb{S}_{\C} \to \mathbf{G}$ gives rise to a mixed $\Q$-Hodge structure $V_h=(V_{\Q}, W^h_{\bullet}V, F^{\bullet}_hV)$. The weight filtration $W_{\bullet}V=W^h_{\bullet}V$ can be recovered as the maximal flag fixed by the unipotent radical $\mathbf{U} \subset \mathbf{G}$ and does not depend on $h$. The complex conjugation on mixed Hodge structures induces an involution on the set of $\mathbf{G}$-valued Hodge cocharacters via $h \mapsto \overline{h}, \ V_{\overline{h}}=\overline{V_h}$. This involution does not depend on the choice of representation $\rho$ and, in fact, is induced by the simultaneous complex conjugation on $\mathbb{S}(\C)$ and $\mathbf{G}(\C)$. In particular, a mixed Hodge structure $V_h$ splits over $\R$ if and only if its Hodge cocharacter $h \colon \mathbb{S}_{\C} \to \mathbf{G}_{\C}$ is defined over $\R$.

Another important tool of the theory of mixed Hodge structures that we need to recall is the notion of a \emph{Jacobian}.

\begin{df}\label{jacobian def}
Let $V=(V_{\Z}, W_{\bullet}V, F^{\bullet}V)$ be a mixed $\Z$-Hodge structure. Its \emph{$p$-th Jacobian} is defined as the double coset
\[
J^pV:=V_{\Z}\backslash V_{\C}/F^pV.
\]
\end{df}

\begin{prop}\label{carlson theorem}
Let $V=(V_{\Z}, W_{\bullet}V, F^{\bullet}V)$ be a mixed $\Z$-Hodge structure. Suppose that $V$ has  negative weights, i.e. $W_{-1}V=V$. Then $V_{\Z}$ acts on $V_{\C}/F^0V$ properly discontinuous. In particular, $J^0V$ is a connected commutative complex Lie group.
\end{prop}
\begin{proof}
The condition $W_{-1}V=V$ implies $F^0V \cap \overline{F^0V}=0$. Thus $V_{\Z} \cap F^0V=0$ and $V_{\Z}$  embeds as a discrete subgroup in $V_{\C}/F^0V$ (cf. \cite[Lemma 3.29.]{PetSteen}) .
\end{proof}

Jacobians are functorial. Namely, let  $V$ and $V'$ be mixed $\Z$-Hodge structures with $W_{-1}V=V$ and $W_{-1}V'=V'$ and let $f \colon V  \to V'$ be a morphism of mixed Hodge structures. Then it induces a morphism of complex Lie groups $J^0f \colon J^0V \to J^0V'$.

\begin{ex}\label{albanese is jacobian}
Let $X$ be a normal quasi-projective variety. By Deligne, $H_1(X,\Z)$ carries a mixed $\Z$-Hodge structure with negative weights. Then, 
\[
J^0H_1(X,\Z)=H_1(X,\Z) \backslash H_1(X,\C) / F^0H_1(X,\C) = \Alb(X).
\]
\end{ex}

In subsection \ref{niljacobians basic} we introduce the notion of a \emph{nil-Jacobian} that allows to generalise the example above and realise higher Albanese manifolds as nil-Jacobians of Hodge structures on quotients of the Malcev completion of $\pi_1(X)$. 

\subsection{Mixed Hodge varieties}\label{hodge varieties subsec}

We briefly recall the modern approach to period domains and their arithmetic quotients in the mixed setting, based on the notion of a \emph{mixed Hodge datum} introduced in \cite{Kling}. This formalism generalises Pink's approach to mixed Shimura varieties \cite{Pink}. See \cite[Subsection 3.2.]{Kling} on the precise relation between the two theories. 

Fix the following notations. As in Theorem \ref{pink}, we denote by $\mathbf{G}$ a connected $\Q$-algebraic group, by $\g$ its Lie algebra, by $\mathbf{U}$ its unipotent radical, by $\mathbf{H}:=\mathbf{G}/\mathbf{U}$ the reductive quotient, and by $\u$ the Lie algebra of $\mathbf{U}$. Let $\widetilde{G}$ be the preimage of $\mathbf{H}(\R) \subset \mathbf{H}(\C)$ under the map $\mathbf{G}(\C) \to \mathbf{H}(\C)$. This is a real algebraic group, which is an extension
\[
1 \to \mathbf{U}(\C) \to \widetilde{G} \to \mathbf{H}(\R) \to 1.
\]
Observe that $\mathbf{G}(\R) \subseteq \widetilde{G} \subseteq \mathbf{G}(\C)$.

Let $h \colon \mathbb{S}_{\C} \to \mathbf{G}_{\C}$ be a Hodge cocharacter. Denote by $X_{\mathbf{G}}$ its $\widetilde{G}$-conjugacy class. The set $X_{\mathbf{G}}$ carries a structure of a real semialgebraic domain inside a complex algebraic variety and admits a transitive real semialgebraic action of $\widetilde{G}$. 

For a Hodge cocharacter $h \in X_{\mathbf{G}}$ the composition $\mathbb{S} \xrightarrow{h} \mathbf{G} \xrightarrow{\operatorname{Ad}} \End(\g)$
defines a mixed Hodge structure $(W_{\bullet}\g_{\Q}, F_h^{\bullet}\g_{\C})$ on $\g$. The Hodge filtration on $\g_{\C}$ is respected by the Lie bracket: $[F^p\g, F^q\g] \subseteq F^{p+q}\g$. Thus, $F^0\g$ is a Lie subalgebra. Denote $F_h^0\mathbf{G}:=\exp(F_h^0\g) \subseteq \mathbf{G}(\C)$ and $F_h^0\widetilde{G}:=F_h^0\mathbf{G} \cap \widetilde{G}$. 

Fix a faithful finite-dimensional $\Q$-representation $\rho \colon \mathbf{G} \to \GL(V_{\Q})$. A cocharacter $h \in X_{\mathbf{G}}$ determines a mixed $\Q$-Hodge structure $(W_{\bullet}V, F_h^{\bullet}V)$ on $V_{\Q}$. For $g \in \mathbf{G}(\C)$ the operator $\rho(g) \in \GL(V_{\C})$ preserves the Hodge filtration $F^{\bullet}_hV$ if and only if $g \in F^0_h\mathbf{G}$. Therefore,  all possible mixed Hodge structures on $V$  given by fixed $\rho$ and some $h \in X_{\mathbf{G}}$  are parametrised by
\[
\mathcal{D}_{\mathbf{G}, X_{\mathbf{G}}}:=X_{\mathbf{G}}/F_h^0\widetilde{G}.
\]

The quotient $\mathcal{D}_{\mathbf{G}, X_{\mathbf{G}}}$ is again a real semiaglebraic domain that does not depend neither on the choice of $h \in X_{\mathbf{G}}$, nor on the representation $\rho$ (\cite{Pink}) and generalises the classical notion of a period domain.

A technical issue arising here is that $\mathcal{D}_{\mathbf{G}, X_{\mathbf{G}}}$ might be not connected. To choose  a connected component of $\mathcal{D}_{X_{\mathbf{G}}}$  is the same as to choose a connected component of $\mathbf{G}(\R)$.

\begin{df}
A \emph{connected mixed Hodge datum} is a triple $(\mathbf{G}, X_{\mathbf{G}}, \mathcal{D})$, where $\mathbf{G}$ is a geometrically connected $\Q$-algebraic group, $X_{\mathbf{G}}$ is a $\widetilde{G}$-conjugacy class of a $\mathbf{G}$-valued Hodge cocharacter and $\mathcal{D} \subseteq \mathcal{D}_{\mathbf{G}, X_{\mathbf{G}}}$ is a connected component. One says that $\mathcal{D}$ is a \emph{connected Hodge domain}. A choice of a faithful representation $\rho \colon \mathbf{G} \to \GL(V_{\Q})$ identifies the points of $\mathcal{D}$ with mixed Hodge structures on $V_{\Q}$ of a certain type.

A connected mixed Hodge datum is \emph{pure} if the unipotent radical $\mathbf{U} \subseteq \mathbf{G}$ is trivial. In this case, the Hodge structures parametrised by $\mathcal{D}$ are pure\footnote{ We follow a convention, in which a \emph{pure} Hodge structure is a direct sum of several pure Hodge structures of given (perhaps, different) weights. This is also natural from categorical point of view, as one wants the category of pure Hodge structures to be abelian.}.

A \emph{morphism of connected mixed Hodge data} $(\mathbf{G}, X_{\mathbf{G}}, \mathcal{D}) \to (\mathbf{G}', X'_{\mathbf{G}'}, \mathcal{D}')$ is a morphism $F \colon \mathbf{G} \to \mathbf{G}'$ of algebraic groups over $\Q$ such that $X'_{\mathbf{G}'}$ is the conjugacy class of $h \circ F$ for some $h \in X_{\mathbf{G}}$ and $F$ sends the connected component of $\mathbf{G}(\R)$ corresponding to $\mathcal{D}$ to the component of $\mathbf{G}'(\R)$ corresponding to $\mathcal{D}'$.

Let ($\mathbf{G}, X_{\mathbf{G}}, \mathcal{D})$ be a connected mixed Hodge datum, $\mathbf{G}^{+}(\R)\subseteq \mathbf{G}(\R)$ the connected component corresponding to $\mathcal{D}$ and $\mathbf{G}^{+}(\Q):=\mathbf{G}(\Q) \cap \mathbf{G}^{+}(\R)$. A \emph{mixed Hodge variety} is a quotient
\[
M:=\Gamma \backslash \mathcal{D},
\]
where $\Gamma \subset \mathbf{G}^{+}(\Q)$ is an arithmetic subgroup.

A \emph{morphism of mixed Hodge varieties}
\[
\Gamma\backslash \mathcal{D}=M \to M'=\Gamma'\backslash \mathcal{D}'
\]
is a morphism of the underlying Hodge data $(\mathbf{G},X_{\mathbf{G}}, \mathcal{D}) \to (\mathbf{G}', X'_{\mathbf{G}'},  \mathcal{D}')$ such that the image of $\Gamma \subset \mathbf{G}(\Q)$ is contained in $\Gamma' \subseteq \mathbf{G}'(\Q)$.

A \emph{pure}  Hodge variety is a mixed Hodge variety whose underlying Hodge datum is pure.
\end{df}

The reader should not be mislead by the terminology: mixed Hodge varieties are rarely algebraic varieties; the natural structure possesed by a Hodge variety is the structure of an analytic Deligne-Mumford stack or a complex orbifold (depending on the reader's tastes).

Every mixed Hodge domain can be realised as a real semialgebraic homogeneous domain inside a complex algebraic variety, in particular it is a homogeneous complex manifold. A morphism of mixed Hodge data induces a holomorphic map of mixed Hodge domains. Similarly, a morphism of mixed Hodge varieties induces a holomorphic map between them (or a morphism of analytic DM-stacks).

A Hodge datum is called \emph{graded polarisable} if for some (equivalently, any) $h \in X_{\mathbf{G}}$ and some (equiv. any) faithful representation $ \rho \colon \mathbf{G} \to \GL(V_{\Q})$ the resulting mixed Hodge structure on $V_{\Q}$ is graded polarisable. In this case, the reductive quotient $\mathbf{H}$ is semisimple.

A mixed Hodge variety is said to be graded polarisable if the underlying Hodge datum is graded polarisable.

\subsection{The purification map}\label{purification subsec}

Every  graded polarised mixed Hodge variety admits a canonical morphism to a pure Hodge variety, which we refer to as \emph{the purification map}.

Namely, let $(\mathbf{G}, X_{\mathbf{G}}, \mathcal{D})$  be a connected graded polarisable mixed Hodge datum. Pick a Hodge cocharacter $h \in X_{\mathbf{G}}$. Let $h_{\sigma}$  be the composition $\mathbb{S}_{\C} \xrightarrow{h} \mathbf{G}_{\C} \to \mathbf{H}_{\C}$. If $X^{\sigma}_{\mathbf{H}}$ is its $\mathbf{H}(\R)$-conjugacy class and $\mathcal{D}_{\sigma}$ is the corresponding connected Hodge domain,  the projection $\mathbf{G} \to \mathbf{H}$ defines a morphism to a pure Hodge datum

\begin{equation}\label{purification above}
(\mathbf{G}, X_{\mathbf{G}}, \mathcal{D}) \to (\mathbf{H}, X^{\sigma}_{\mathbf{H}}, \mathcal{D}_{\sigma}).
\end{equation}

Let $\rho \colon \mathbf{G} \to \GL(V_{\Q})$ be a faithful representation, so that every point in $\mathcal{D}$ can be interpreted is a mixed Hodge structure $(W_{\bullet}V_{\Q}, F^{\bullet}V_{\C})$ on $V_{\Q}$. Then the map $\mathcal{D} \to \mathcal{D}_{\sigma}$ induced by (\ref{purification above}) corresponds to the operation of taking associated graded of the weight filtration 
\[
(W_{\bullet}V_{\Q}, F^{\bullet}V_{\C}) \mapsto \bigoplus_k ( \Gr^W_k V_{\Q}, \Gr^W_k F^{\bullet}V_{\C}).
\]

(the period domain $\mathcal{D}_{\sigma}$ canonically splits into a product $\prod_k \mathcal{D}_{\sigma_k}$, where each $\mathcal{D}_{\sigma_k}$ parametrises pure weight $k$ Hodge structures on $W_kV/W_{k-1}V$).

 Let $\Gamma \subset \mathbf{G}^{+}(\Q)$ be an arithmetic subgroup and $M=\Gamma \backslash \mathcal{D}$ a mixed Hodge variety. Choose a representation $\rho \colon \mathbf{G} \to \GL(V_{\Q})$ in such a way that $\rho(\Gamma)$ is conjugate to $\SL(V_{\Z}) \cap \rho(\mathbf{G}(\Q))$ for some $\Z$-structure $V_{\Z} \subset V_{\Q}$. A choice of a Hodge cocharacter $h \in X_{\mathbf{G}}$ endows $V_{\Q}$ with a mixed $\Z$-Hodge structure $(V_{\Z}, W_{\bullet}V_{\Q}, F^{\bullet}V_{\C})$. The representation $\rho$ descends to $\rho_{\sigma} \colon \mathbf{H} \to \GL(\bigoplus_k\Gr^W_kV)$ and the image $\Gamma_{\sigma}$ of $\Gamma$ in $\mathbf{H}(\Q)$ is conjugate to $\rho_{\sigma}(\mathbf{H}(\Q)) \cap \SL(\bigoplus \Gr^W_kV_{\Z})$. This shows, that $\Gamma_{\sigma}$ is arithmetic and $M_{\sigma}:=\Gamma_{\sigma} \backslash \mathcal{D}_{\sigma}$ is  a pure Hodge variety. The morphism of Hodge data (\ref{purification above}) descends to a morphism of mixed Hodge varieties
\[
\sigma_M\colon M \to M_{\sigma}.
\]
The described procedure is functorial: if $f \colon M \to M'$ is a morphism of mixed Hodge varieties, there exists a canonical morphism $f_{\sigma} \colon M_{\sigma} \to M'_{\sigma}$ of their purifications such that the diagram 
\[
\xymatrix{
M \ar[r]^f \ar[d]_{\sigma_M} & M' \ar[d]^{\sigma_{M'}}\\
M_{\sigma} \ar[r]_{f_{\sigma}} & M'_{\sigma}
}
\]
commutes (\cite[Proposition 2.9]{Pink}). The purification map $M \to M_{\sigma}$ is an isomorphism if and only if $M$ is pure. Every morphism from a mixed Hodge variety $M$ to a pure Hodge variety factorises through $\sigma$.

The fibres of the purification map can be explicitly described. Namely, let  $x \in M_{\sigma}$ and  denote by $N_{x}:=\sigma_M^{-1}(\{x\}) \subseteq M$ the fibre.  Choose any $h \in X_{\mathbf{G}}$ such that its image under the composition $X_{\mathbf{G}} \to \mathcal{D} \to M$ belongs to $N_x$. Then $h$ determines a Hodge filtration $F_h^{\bullet}\g$ on $\g$. Let $F^0_x\mathbf{U}:=\mathbf{U}(\C) \cap \exp(F^0\g)$ and $\Gamma_U:=\Gamma \cap \mathbf{U}(\Q)$. In this notations, 
\[
N_{x}=\Gamma_{U} \backslash \mathbf{U}(\C)/F_x^0\mathbf{U}.
\]

\begin{ex}
Let $\mathcal{A}_g=[\Sp_{2g}(\Z) \backslash \mathcal{H}_g]$ be the moduli stack of principly polarised abelian varieties of dimension $g$ and $\mathcal{X}_g \to \mathcal{A}_g$ the universal abelian variety. Then $\mathcal{X}_g$ is a mixed Hodge variety (\cite[Example 2.25.]{Pink}) and $\mathcal{X}_g \to \mathcal{A}_g$ is its purification map. Although both $\mathcal{X}_g$ and $\mathcal{A}_g=(\mathcal{X}_g)_{\sigma}$ are locally homogeneous, the complex structure on the fibres of the purification map evidently varies from point to point.
\end{ex}

Let $X$ be a smooth complex quasi-projective variety and $\V$ a graded polarised admissible $\Z$-variation of mixed Hodge structures on it. One associates to it the \emph{period map}
\[
\Phi_{\V} \colon X^{\an} \to M,
\]
whose target is a mixed Hodge variety $M$ (see \cite[Subsection 3.5.]{Kling}). The variation $\V$ is said to be \emph{unipotent} if one of the following equivalent conditions is satisfied:
\begin{itemize}
\item the associated graded variation of pure Hodge structures $\Gr^W\V$ is trivial;
\item the global monodromy of $\V$ is quasi-unipotent;
\item the image of $\Phi_{\V}$ is contained in a fibre of the purification map $M \to M_{\sigma}$.
\end{itemize}
\subsection{The $\operatorname{sl}_2$-splitting}\label{about splitting subsec}

We fix a graded polarisable mixed Hodge datum $(\mathbf{G}, X_{\mathbf{G}}, \mathcal{D})$. Recall, that $\mathcal{D}$ admits a transitive holomorphic action of the group $\widetilde{G}$. The so-called \emph{$\operatorname{sl}_2$-splitting} is a canonical real semialgebraic retraction of $\mathcal{D}$ on a certain orbit of $\mathbf{G}(\R) \subset \widetilde{G}$. It plays the crucial role in the construction of the definable structure on mixed Hodge varieties due to \cite{BBKT} (see also subsection \ref{o-minimal hodge}).

Fix a faithful finite-dimensional rational representation $\rho \colon \mathbf{G} \to \GL(V_{\Q})$. As we mentioned, the Hodge domain $\mathcal{D}$ can be viewed as the space of  Hodge filtrations on the rational filtred vector space $(V_{\Q}, W_{\bullet}V)$. Under a \emph{splitting} of the weight filtration $W_{\bullet}V$ we understand a splitting $V=\bigoplus V_k$ such that $W_kV=\bigoplus_{l \le k} V_{l}$. The set of all possible splittings is naturally an algebraic variety over $\Q$ (\emph{the variety of splittings}), which we denote by $\mathcal{S}(W)$. There is a real semialgebraic map
\begin{equation}\label{splitting above}
\mathcal{S}(W) \times \mathcal{D}_{\sigma} \to \mathcal{D}
\end{equation}
and the existence of the Deligne splitting (Lemma \ref{Deligne}) implies that this map is surjective.

Let $T=(W_{\bullet}, F^{\bullet} )\in \mathcal{D}$ be a mixed Hodge structure on $V_{\Q}$ and $V_{\C}=\bigoplus I_{T}^{r,s}$ the Deligne splitting of the corresponding mixed Hodge structure on $V_{\Q}$. The \emph{complex conjugate} mixed Hodge structure  $\overline{T}$ is the one whose Deligne splitting satisfies 
\[
I^{r,s}_{\overline{T}} = \overline{I^{s,r}_T}.
\]
As we mentioned, this operation does not depend on the representation $\rho$ and gives  a well-defined antiholomorphic involution  $\mathcal{D} \to \mathcal{D}$.

The involution $T \mapsto \overline{T}$ preserves the fibres of the purification map $\mathcal{D} \to \mathcal{D}_{\sigma}$ and fibrewise  lifts to the complex conjugation on $\mathcal{S}(W)(\C)$. In particular, if $\mathbf{u}$ is an element of the unipotent radical $\mathbf{U}(\C) \subseteq \widetilde{G}$, then 
\[
\overline{\mathbf{u} \cdot T}=\overline{\mathbf{u}} \cdot \overline{T},
\]
where $\mathbf{u} \mapsto \overline{\mathbf{u}}$ is the conjugation on $\mathbf{U}(\C)$ induced by the real form $\mathbf{U}(\R) \subset \mathbf{U}(\C)$.

Therefore, a mixed Hodge structure splits over $\R$ (i.e. $T \simeq \overline{T}$) if and only if it lies in the image of $\mathcal{S}(W)(\R) \times \mathcal{D}_{\sigma}$ under the map (\ref{splitting above}).

We denote the subset of Hodge structures that split over $\R$ by $\mathcal{D}_{\R} \subseteq \mathcal{D}$. The group $\mathbf{G}(\R)$ acts transitively on it. The following Lemma is \cite[Proposition 2.20]{CKS}

\begin{lemma}[Cattani-Kaplan-Schmid]
Let $T \in \mathcal{D}$. Then there exists a unique element $\delta_T \in \g$ such that $\overline{T}=e^{-2\i \delta_T}\cdot T$. Moreover, $e^{-\i \delta_T} \cdot T$ splits over $\R$ and the map $T \mapsto e^{-\i \delta_T} \cdot T$ is a real semialgebraic retraction of $\mathcal{D}$ on $\mathcal{D}_{\R}$ (\emph{''the $\delta$-splitting''}).
\end{lemma}

The Bakker-Brunebarbe-Klingler-Tsimerman definable structure on mixed Hodge varieties is based not on the $\delta$-splitting itself, rather on its technical upgrade, the \emph{$\operatorname{sl_2}$-splitting}. It is of the form
\[
r \colon \mathcal{D} \to \mathcal{D}_{\R}, \ T \mapsto e^{-\i\zeta_T} e^{-\i\delta_T} \cdot T,
\]
where $\zeta_T$ is a certain canonical polynomial in $\delta_T$ (see \cite{CKS} for more details and \cite{KNU00} where the first terms of $\zeta_T$ are explicitly computed). The $\operatorname{sl}_2$-splitting is again real semi-algebraic and functorial under morphisms of mixed Hodge data.

\section{Preliminaries from o-minimal geometry}\label{o-minimal sec}

We briefly recall the necessary facts from o-minimal geometry.

Theory of o-minimal structures originated in the context of model theory and found unexpected applications in algebraic geometry and Hodge theory in the last decade (\cite{BKT}, \cite{BBKT}, \cite{BBT}, \cite{EK} etc.). There are two main contexts in which o-minimality arises in complex geometry. First, it provides strong \emph{algebraisation criteria}, that is, it gives tools to verify the algebraicity of some objects of a priori complex analytic nature. Second, it  gives a suitable setup for the study of period maps of variations of Hodge structures and their behaviour at infinity.

In subsection \ref{o-minimal basic} we briefly recall the basic notions related to o-minimality with an emphasis on geomeric applications and prove some simple statements about definable spaces. In subsection \ref{complex o-minimal} we recap Bakker-Brunebarbe-Tsimerman's theory of definable analytic spaces. In subsection \ref{algebraisation results} we collect the algebraisation results that will be used further. In subsection \ref{o-minimal hodge} we discuss the Bakker - Brunebarbe - Klingler - Tsimerman definable structures on mixed Hodge varieties.

\subsection{o-minimal geometry}\label{o-minimal basic}

The primary reference is \cite{VdD}.

Recall that \emph{a structure} $\Sigma=(\Sigma_n)$ is a collection $\Sigma_n$ of boolean subalgebras of $\R^n$ satisfying the following conditions:
\begin{itemize}
\item if $A \in \Sigma_n$ and $B \in \Sigma_m$ then $A \times B \in \Sigma_{n+m}$;
\item if $A \in \Sigma_n$ and $p \colon \R^n \to \R^{n-1}$ is a linear projection, then $p(A) \in \Sigma_{n-1}$;
\item if $A \subseteq \R^n$ is real semialgebraic\footnote{This means that $A$ can be written as $A=A_1 \cup \ldots \cup A_k$, where each $A_k$ is given by a finite collection of real polynomial equalitites and inequalities.}, then $A \in \Sigma_n$.
\end{itemize}
Sets belonging to $\Sigma_n$ are called \emph{($\Sigma$-)definable}. A map between definable sets $f \colon A \to B$ is \emph{definable} if its graph $\Gamma_f \subseteq A \times B$ is definable.

Every structure is closed under basic topological and set-theoretic operations: taking topological closures, interiours, frontiers, images and preimages under definable maps etc. (\cite[Chapters 1,2.]{VdD}.

A typical example is the structure of real semialgebraic sets $\R_{\alg}$.

\begin{df}
A structure $\Sigma$ is called \emph{o-minimal} if $\Sigma_1=(\R_{\alg})_1$. Equivalently, the only $1$-dimensional definable sets are finite unions of points, intervals and rays.
\end{df}

The structure $\R_{\alg}$ is clearly o-minimal.  Another o-minimal structure which is often used in the geometric applications is $\R_{\an, \exp}$. This is the minimal structure such that the graphs of analytic functions with compact support $f \colon \R^n \to \R^m$ and the graph of the real exponent $\exp \colon \R \to \R$ are definable (its o-minimality is a deep theorem of \cite{VdDM}).

Throughout this paper we work exclusively with o-minimal structures. We will often abbreviate <<definable>> for <<definable in some o-minimal structure $\Sigma$>> without referring to the particular choice of $\Sigma$. Usually, it is sufficient to work in $\R_{\alg}$, although for some applications one has to pass to  $\R_{\an,\exp}$.

The o-minimality condition imposes  strong restrictions on the topology and geometry of definable sets:

\begin{itemize}
\item A set definable in an o-minimal structure is connected if and only if it is path-connected.
\item A set definable in an o-minimal structure admits a \emph{finite definable cell decomposition}. This means that a definable set $A$ can be written as a finite union of non-intersecting definable locally closed subsets $A=C_1 \sqcup \ldots \sqcup C_k$ with each $C_j$ being definably homeomorphic to $\R^{n_j}$. (\cite[Chapter 3, Theorem 2.11]{VdD}). In particular, every definable set has the homotopy type of a finite CW-complex.
\item If $A$ and $B$ are definable in an o-minimal structure and $f \colon A \to B$ is a definable map, then $B$ admits a finite definable stratification, such that the restrictions of  $f$ on the strata are locally trivial fibrations (\cite[Chapter 9]{VdD}).
\end{itemize}

The general slogan is that every topological or geometrical statement that is true for real semialgebraic sets, also holds for sets definable in an arbitrary o-minimal structure.

A prototypical example of a set that is not definable in any o-minimal structure is the graph of the sine function $\sin \colon \R \to \R$. Indeed, if it were definable,  the intersection of its graph $\{(x, \sin x)\} \subset \R^2$ with the algebraic set $\{y=0\}$  would be definable, which contradicts the existence of a finite cell decomposition.

Similarly, let $f \colon \C \to \C$ be an entire function and suppose that its graph $\Gamma_f \subseteq \C^2$ is definable in some o-minmal structure (as a subset of $\C^2=\R^4$). Then $f$ is algebraic. Indeed, $f$ is not algebraic, it must have an essential singularity and, by Picard's great theorem, its level sets $\{f(z)=a\}=\Gamma_f \cap \left(\C \times \{a\} \right)$  are discrete countable subsets of $\C$. This contradicts the definability of $\Gamma_f$. This observation provides an insight into the connection between definability and algebraicity in complex analysis, see subsection  \ref{algebraisation results}.

One can also work in the more abstract setting of definable topological spaces. 

\begin{df}
Let $\Sigma$ be a structure and $X$ a topological space. A \emph{($\Sigma$)-definable atlas} on $X$ is a finite collection $(U_{\alpha}, \phi_{\alpha})$ is a finite collection of open subsets $U_{\alpha}$ and continuous open embeddings $\phi_{\alpha} \colon U_{\alpha} \hookrightarrow \R^{n_{\alpha}}$ such that:
\begin{itemize}
\item $\phi_{\alpha}(U_{\alpha})$ is definable for every $\alpha$;
\item $\phi_{\alpha}(U_{\alpha} \cap U_{\beta})$ is definable for every $\alpha$ and $\beta$;
\item $\phi_{\beta} \circ \phi_{\alpha}^{-1} \colon \phi_{\alpha}(U_{\alpha}\cap U_{\beta}) \to \phi_{\beta}(U_{\beta} \cap U_{\alpha})$ is definable for every $\alpha$ and $\beta$.
\end{itemize}
Two atlases $(U_{\alpha}, \phi_{\alpha})$ and $(U'_{\alpha}, \phi'_{\alpha})$ are \emph{equivalent}, if their union is a definable atlas. A \emph{definable space} is a pair $(X, \xi)$, where $\xi$ is an equivalence class of definable atlases.
\end{df}

A subset $Z$ of a definable space $(X, \xi=[(U_{\alpha}, \phi_{\alpha})])$ is \emph{definable} if $\phi_{\alpha}(Z \cap U_{\alpha})$ is definable for every $\alpha$. In this case, $Z$ canonically inherits a  structure of a definable space.

The product of definable topological spaces admits a natural definable space structure. A map $f \colon (X, \xi_X) \to (Y, \xi_Y)$ is definable if its graph $\Gamma_f \subseteq (X \times Y, \xi_X \times \xi_Y)$ is.

All geometric and topological properties of definable sets translate directly to the context of definable spaces. For example, a space definable in an o-minimal structure admits a finite definable cell decomposition and, therefore, has a finite homotopy type.

Abusively, we tend to omit the definable atlas $\xi_X$ from the notation, indentifiying a definable space $(X, \xi_X)$ with the underlying topological space $X$, when it leads to no misunderstanding.

A \emph{definable manifold} is a definable topological space, such that the atlas of definable charts provides a smooth manifold structure on it. 

\begin{prop}\label{finite covers}
Let $\tau \colon \widehat{X} \to X$ be a finite cover of connected locally contractible topological spaces. The following holds:
\begin{itemize}
\item[(i)] for a definable space structure $(X, \xi)$ on $X$ there exists  unique definable space structure $\widehat{\xi}$ on $\widehat{X}$ such that $\tau \colon \widehat{X} \to X$ is definable. Vice versa, for a definable space structure $(\widehat{X}, \widehat{\xi})$ there exists unique definable space structure $\xi$ on $X$ such that $\tau$ is definable;
\item[(ii)] let $(\widehat{X}, \widehat{\xi}) \xrightarrow{\tau} (X, \xi)$ be as above, $(Y, \eta)$ be a definable connected topological space and $f \colon Y \to X$ a  continuous map that admits a continuous lift $g \colon Y \to \widehat{X}$ (that is, the composition $\pi_1(Y) \xrightarrow{f_*} \pi_1(X) \to \operatorname{Aut}_X(\widehat{X})$ is trivial). 
\[
\xymatrix{
& \widehat{X} \ar[d]^{\tau}\\
Y \ar[ru]^{g} \ar[r]_f & X
}
\]
Then $f$ is definable if and only if $g$ is;
\item[(iii)] let $f \colon Y \to X$ be a continuous  map between definable spaces and $ \sigma \colon \widehat{Y} \to Y$ (resp. $\tau \colon \widehat{X} \to X$) be finite covers. Suppose there exists a lift $\widehat{f} \colon \widehat{Y} \to \widehat{X}$ of $f$.
\[
\xymatrix{
\widehat{Y} \ar[d]_{\sigma} \ar[r]^{\widehat{f}} & \widehat{X} \ar[d]^{\tau}\\
Y \ar[r]_{f} & X
}
\]
Then $\widehat{f}$ is definable if and only if $f$ is.
\end{itemize}
\end{prop}
\begin{proof}
 \textit{(i)}. Let $\xi$ be a definable space structure on $X$. By Finite Cell Decomposition Theorem each definable set admits a finite covering by open simply connected definable subsets. Therefore, there exists a definable atlas $(U_{\alpha}, \phi_{\alpha})$ on $X$ in the equivalence class $\xi$ with each $U_{\alpha}$ being simply connected. Let $\widehat{U}_{\alpha, \beta} \subset \widehat{X}$ be connected components of $\tau^{-1}(U_{\alpha})$. Then $(\widehat{U}_{\alpha, \beta}, \phi_{\alpha} \circ \pi)$ is a definable atlas on $\widehat{X}$. Vice versa, if $\widehat{\xi}$ is a definable space structure on $\widehat{X}$, there exists a definable atlas $(\widehat{U}_{\beta}, \widehat{\phi}_{\beta}) \in \widehat{\xi}$, such that $\tau|_{\widehat{U}_{\beta}}$ is a homeomorphism on its image. Then $(\tau(\widehat{U}_{\beta}), \phi_{\beta} \circ (\tau|_{\widehat{U}_{\beta}})^{-1})$ is a definable atlas on $X$. The definability of $\tau$ in both cases is immediate.
\\

Item \textit{(ii)} is the special case of \textit{(iii)} when $\widehat{Y}=Y$.
\\

Item \textit{(iii)} follows from the fact that the graph of $\widehat{f}$ is the preimage of the graph of $f$ under the map $\widehat{Y} \times \widehat{X} \xrightarrow{\sigma \times \tau} Y \times X$. Images and preimages of definable sets are definable.
\end{proof}

A \emph{definable Lie group} is a definable manifold $G$ endowed with a definable smooth group structure $m \colon G \times G \to G, \ \iota \colon  G \to G$. 

Notice, that the definability of the inversion map $\iota$ follows from the definabiliy of the multiplication map $m$, since its graph $\Gamma_{\iota} \subseteq G \times G$ is the same as the fibre $m^{-1}(e_G)$.

We say that the action of a definable Lie group $G$ on a definable  manifold $M$ is definable if the action map $G \times M \to M$ is definable.

\begin{prop}\label{lifting of action}
\begin{itemize}
\item[(i)] Let $G$ be a definable manifold. Assume that $G$ is endowed with a Lie group structure (a priori not definable). Assume also that there exists a definable manifold $M$ and a free smooth action of $G$ on $M$, such that the action map $G \times M \to M$ is definable. Then the group structure on $G$ is definable.
\item[(ii)] Let $G$ be a definable Lie group acting smoothly on a definable manifold $M$. Suppose $\widehat{G} \to G$ and $\widehat{M} \to M$ are finite covers of $G$ and $M$ respectively, and the action of $G$ on $M$ lifts to an action of $\widehat{G}$ on $\widehat{M}$. Then the first action is definable if and only if the second is.
\end{itemize}
\end{prop}
\begin{proof}
\textit{(i)}. The orbits of the action are definable subsets of $M$. Indeed, they are given by images of the sets of the form $G \times \{x\} \subset G \times M$ under the definable action map $G \times M \to M$. 

Therefore, replacing $M$ with an orbit of the action, we may assume that the action is transitive and $M$ is a $G$-torsor. In this case, the multiplication map coincides with the action after choosing a definable diffeomorphism $G \xrightarrow{\sim} M, g \mapsto g \cdot x$.
\\

\textit{(ii)}. Follows by applying \textit{(iii)} of Proposition \ref{finite covers} to the action maps $G \times M \to M$ and $\widehat{G} \times \widehat{M} \to \widehat{M}$.
\end{proof}

\subsection{Complex analytic o-minimal geometry}\label{complex o-minimal}

We recall the theory of definable analytic spaces after Bakker, Brunebarbe and Tsimerman. See \cite{BBT} for more details.

The \emph{definable site} $\underline{X}$ of a definable space $X$ is the site whose underlying category is the category of definable open subsets of $X$ (with inclusions as morphisms) and whose coverings are finite coverings by definable open sets. Further, one can make sense of \emph{definable sheaves} and \emph{definable locally ringed spaces} (\cite[Subsection 2.2]{BBT}).

Take the standard identification of $\C^n$ with $\R^{2n}$. Denote by $\O^{\odef}_{\C^n}$ the sheaf of definable holomorphic functions on the definable site $\underline{\C^n}$. Explicitly, for a definable open subset $U \subseteq \C^n$ we put 
\[
\O^{\odef}_{\C^n}(U):=\{f \colon U \to \C \ | f \text{ is holomorphic and definable} \}.
\]
This turns $(\C^n, \O^{\odef}_{\C^n})$ into a \emph{$\C$-locally ringed definable space}.

Given a definable open $U \subseteq \C^n$ and a finitely generated ideal $\mathcal{I}_X \subseteq \O^{\odef}_U=\O^{\odef}_{\C^n}|_{U}$, the vanishing locus $X=|V(\mathcal{I}_X)|$ is a definable subspace of $U$ and inherits  a definable space structure.

A \emph{basic definable complex analytic space} is a $\C$-locally ringed definable space $(X, \O^{\odef}_X)$ isomorphic to $(|V(\mathcal{I})|, \O^{\odef}_U/\mathcal{I})$ for some $U$ and $ \mathcal{I}$ as above. A \emph{definable complex analytic space} is a $\C$-locally ringed definable space $(X, \O_X)$ locally definably isomorphic to a basic definable complex analytic space. A \emph{definable complex manifold} is a smooth definable complex analytic space.

The authors of \cite{BBT} also introduce the notion of a \emph{coherent sheaf over a definable complex analytic space} and show that the structure sheaf $\O^{\odef}_X$ of on a definable analytic space $(X, \O^{\odef}_X)$ is coherent (\cite[Theorem 2.38]{BBT}). 

There is a natural analytification functor $(-)^{\an}$ from the category of definable analytic spaces $\operatorname{DefAnSp}_{\C}$ to the category of complex analytic spaces $\operatorname{AnSp}_{\C}$ and the functor of sheaf analytification $(-)^{\an} \colon \mathbf{Coh}(X) \to \mathbf{Coh}(X^{\an})$ with canonical isomorphism $\O_X^{\an} \simeq(\O^{\odef}_X)^{\an}$.  

The reader not familiar with the discussed concepts can see  the following proposition  as a - perhaps slightly tedious - exercise.

\begin{prop}
Propositions \ref{finite covers} and \ref{lifting of action} also hold in the category of definable complex analytic spaces.
\end{prop}

Let $\operatorname{AlgSp}_{\C}$ denote the category of separated complex algebraic spaces of finite type.

\begin{thrm}[Bakker - Brunebarbe - Tsimerman's definable GAGA; \cite{BBT}]
.

\begin{itemize}
\item[(i)] There is a \emph{definabilisation functor} $(-)^{\odef} \colon \operatorname{AlgSp}_{\C} \to \operatorname{DefAnSp}_{\C}$. The diagram
\[
\xymatrix{
\operatorname{AlgSp}_{\C} \ar[rr]^{(-)^{\odef}} \ar[rd]_{(-)^{\an}}&& \operatorname{DefAnSp}_{\C} \ar[ld]^{(-)^{\an}}\\
& \operatorname{AnSp}_{\C}
}
\]
is commutative up to a natural transformation;
\item[(ii)] Let $\mathfrak{X} \in \operatorname{AlgSp}_{\C}$. There is a fully faithful exact \emph{sheaf definabilisation functor} $(-)^{\odef} \colon \mathbf{Coh}(\mathfrak{X}) \to \mathbf{Coh}(\mathfrak{X}^{\odef})$. Its essential image is closed under taking quotients and subobjects.
\end{itemize}
\end{thrm}

We say that a diagram of definable spaces \emph{admits an algebraisation}, or, simply, \emph{is algebraisable}, if it lies in the essential image of the functor $(-)^{\odef} \colon \operatorname{AlgSp}_{\C} \to \operatorname{DefAnSp}_{\C}$.

Although the definabilisation functor embeds the category of algebraic spaces into the category of definable complex analytic spaces (even as a full subcategory, see Corollary \ref{algebraisation basic}, \textit{(i)}), the world of definable complex analytic spaces if significantly bigger. For example a unit disk $\{|z| < 1\} \subset \C$ is an $\R_{\alg}$-definable complex manifold which is not a definabilisation of any algebraic space (the same is true for any semialgebraic domain in $\CP^n$).

Let us also mention the following analogue of Remmert - Stein Theorem due to Peterzil and Starchenko (\cite[Theorem 4.13.]{PetStar})

\begin{thrm}\label{definable Remmert Stein}
Let $X$ be definable complex analytic space and $Y \subseteq X$ a definable complex analytic subspace. Then its closure $\overline{Y}$ is a definable complex analytic space and $\dim \overline{Y}=\dim Y$.
\end{thrm}

\begin{rmk}\label{expansion}
One says that a structure $\Sigma'$ is an \emph{expansion} of a structure $\Sigma$ if every $\Sigma$-definable set is $\Sigma'$-definable.  For example, $\R_{\an, \exp}$ is an expansion of $\R_{\alg}$. If $\Sigma'$ is an expansion of $\Sigma$, there is a natural fully faithful functor from the category of $\Sigma$-definable spaces (resp. $\Sigma$-definable complex analytic spaces) to the  category of $\Sigma'$-definable spaces (resp. $\Sigma'$-definable complex analytic spaces). Every structure is an expansion of $\R_{\alg}$ and the Bakker - Brunebarbe -Tsimerman's definabilisation functor $(-)^{\odef}$ hits to the category of $\R_{\alg}$-definable complex analytic spaces, and then can be canonically extended to a functor to the category of $\Sigma$-definable complex analytic spaces for any o-minimal $\Sigma$.

We sometimes say that a map between two $\R_{\alg}$-definable spaces is $\R_{\an,\exp}$-definable, meaning that it is definable as the map between underlying $\R_{\an,\exp}$-definable spaces.
\end{rmk}

\subsection{Algebraisation theorems}\label{algebraisation results}

One of the main purposes of the theory of definable analytic spaces is to provide various algebraisation criteria. The work \cite{BBT} was to a large extent  motivated by a theorem of Peterzil and Starchenko, known as \emph{definable Chow theorem}(\cite[Theorem 5.1.]{PetStar}).

\begin{thrm}[Peterzil-Starchenko's definable Chow Theorem]\label{definable chow}
Let $\mathfrak{X}$ be a reduced complex algebraic space and $X=\mathfrak{X}^{\odef}$. Let $Y \subseteq X$ be a closed analytic definable subset. Then $Y=\mathfrak{Y}^{\odef}$ for some algebraic subspace $\mathfrak{Y} \subseteq \mathfrak{X}$.
\end{thrm}

\begin{cor}\label{algebraisation basic}
Let $\mathfrak{X}$ and $\mathfrak{Y}$ be complex algebraic spaces and put $X=\mathfrak{X}^{\odef}$(respectively, $Y=\mathfrak{Y}^{\odef}$).
\begin{itemize}
\item[(i)] Let $f \colon X \to Y$ be a morphism of definable analytic spaces. Then it is algebraisable;
\item[(ii)] suppose $X$ admits a definable complex Lie group structure. Then the group structure on $X$ admits algebraisation. Namely, $\mathfrak{X}$ admits a structure of a complex algebraic group and the group structure morphisms $m \colon X \times X \to X, \ i \colon X \to X$ algebraise to group structure morphisms on $\mathfrak{X}$;
\item[(iii)] suppose $\mathfrak{X}$ is an algebraic group (thus, $X$ is a definable complex Lie group). Then every definable action $a \colon X \times Y \to Y$ admits algebraisation. 
\end{itemize}
\end{cor}
\begin{proof}
Apply Theorem \ref{definable chow} to the graph of $f$ (respectively, to the graphs of the group structure morphisms and to the graph of $a$).
\end{proof}

The following generalisation of definable Chow Theorem is  \cite[Theorem 4.2]{BBT}

\begin{thrm}[Baker -Brunebarbe -Tsimerman]\label{algebraisation proper}
Let $\mathfrak{X}$ be an algebraic space, $X = \mathfrak{X}^{\odef}$ and $f \colon X \to Y$ a proper dominant definable holomorphic map to a definable analytic space $Y$. Then it is algebraisable: there exist unique up to an isomorphism algebraic space $\mathfrak{Y}$ and morphism $\mathfrak{f} \colon \mathfrak{X} \to \mathfrak{Y}$ such that, $\mathfrak{Y}^{\odef}=Y$ and $\mathfrak{f}^{\odef}=f$.
\end{thrm}

\subsection{o-minimal geometry and Hodge theory}\label{o-minimal hodge}

Apart from the algebraisation results, another source of the rise of interest in o-minimality in the complex algebraic geometry in the last years is the definability of period maps. The following theorem was proved in \cite{BKT} in the pure case and in \cite{BBKT} in general.

\begin{thrm}[Bakker- Brunebarbe - Klingler - Tsimerman, \cite{BBKT}]\label{period maps o-minimal}
Let $M$ be a mixed Hodge variety. Then $M$ admits a structure of a $\R_{\alg}$-definable complex analytic space in such a way that:
\begin{itemize}
\item[(i)] every morphism  of mixed Hodge varieties is definable;
\item[(ii)] if $\mathfrak{X}$ is a smooth quasi-projective variety,  and $\Phi \colon \mathfrak{X}^{\an} \to M$ a period map of an admissible variation of mixed $\Z$-Hodge structures targeted to $M$, then $\Phi \colon \mathfrak{X}^{\odef} \to M$ is $\R_{\an, \exp}$-definable\footnote{ See Remark \ref{expansion}.} .
\end{itemize}
\end{thrm}

Let us say few words about how the definable analytic space structure on a mixed Hodge variety is constructed.

Suppose that $\Omega$ is a definable manifold (or a definable smooth complex analytic space) and $\Gamma$ a finitely generated group which acts on $\Omega$  by definable (holomorphic) transformations, properly discontinuously, and with finite stabilisers. Let $\pi \colon \Omega \to S:=\Gamma \backslash \Omega$ be the quotient.

A choice of a definable $\Gamma$-fundamental domain $\Xi \subseteq \Omega$ determines a  unique structure of a definable space (resp. definable complex analytic space) on $S$ such that $\pi|_{\Xi} \colon \Xi \to S$ is definable (\cite[Proposition 2.3]{BBKT}. It is characterised by the following property.
\begin{prop}\label{definable criterion}
A subset $A \subset S$ is definable if and only if $\pi^{-1}(A) \cap \Xi$ is definable in $\Xi$.
\end{prop}
\begin{proof}
Set $\widetilde{A}:=\pi^{-1}(A)$ and $\widetilde{A}_{\Xi}:=\widetilde{A} \cap \Xi$.

If $\widetilde{A}_{\Xi} \subseteq \Xi$ is definable, then $\pi(\widetilde{A}_{\Xi})\subseteq \pi(\Xi) \subset S$ is definable. Notice that $\pi(\Xi)$ is a definable dense open subset of $S$ and $\pi(\widetilde{A}_{\Xi})=\pi(\Xi) \cap A$ is dense and open in $A$. Thus, $A=\overline{\pi(\widetilde{A}_{\Xi})}$ is definable.

Vice versa, if $A$ is definable, then $A \cap \pi(\Xi)$ is definable and $\widetilde{A}_{\Xi}=(\pi|_{\Xi})^{-1}\left ( A \cap \pi(\Xi) \right )$ is definable.
\end{proof}

If $S$ is compact, the resulting definable space structure on $S$ does not depend on $\Xi$. In the non-compact case the situation changes drastically, and the definable geometry of $S$ becomes very sensible to the choice of $\Xi$. It is sometimes then useful to chose a $\Gamma$-invariant closed definable subset $\Omega_0 \subset \Omega$ such that $S_0:=\Gamma \backslash \Omega_0$ is compact, and a definable retraction $r \colon \Omega \to \Omega_0$. Then the fundamental domain $\Xi$ is constructed as $r^{-1}(\Xi_0)$, where $\Xi_0$ is a fundamental domain for the action of $\Gamma$ on $\Omega_0$.

The construction of \cite{BBKT} follows a similar idea. For a mixed Hodge variety $M=\Gamma \backslash \mathcal{D}$ they construct a certain definable fundamental domain $\Xi_{\R}$ for the action of $\Gamma$ on $\mathcal{D}_{\R}$. Then the structure of a definable manifold on $M$ is determined by the fundamental domain $\Xi:=r^{-1}(\Xi_{\R}) \subset \mathcal{D}$, where $r \colon \mathcal{D} \to \mathcal{D}_{\R}$ is the $\operatorname{sl}_2$-splitting (see subsection \ref{about splitting subsec}). Although the action of $\Gamma$ on $\mathcal{D}_{\R}$  might be not cocompact in general, it necessarily becomes so after restricting on the fibres of the purification map.

Combining Theorem \ref{period maps o-minimal} with Theorem \ref{algebraisation proper}, Bakker, Brunebarbe and Tsimerman proved (the mixed case of) a long-standing conjecture of Griffiths:
\begin{thrm}[Bakker - Brunebarbe - Tsimerman, \cite{BBT23Griff}]\label{griffiths}
Let $\mathfrak{X}$ be a smooth quasi-projective variety, $X=\mathfrak{X}^{\odef}$, and $M$ a mixed Hodge variety. Let $\Phi \colon X \to M$ be the period map of an admissible variation of mixed $\Z$-Hodge structures. Denote by $Y$ the reduced image of $\Phi$. Then there exist a quasi-projective variety $\mathfrak{Y}$ and a morphism $\mathfrak{F} \colon \mathfrak{X} \to \mathfrak{Y}$ such that $Y=\mathfrak{Y}^{\odef}$ and $\Phi=\mathfrak{F}^{\odef}$.
\end{thrm}

\begin{rmk}.
 A priori one would expect from Theorem \ref{algebraisation proper} that $\mathfrak{Y}$ is merely an algebraic space. The work \cite{BBT23Griff} uses geometry of period maps to conclude that it is actually a quasi-projective variety.
\end{rmk}

\begin{rmk}\label{definable criterion remark}
Suppose we are in the situation of Proposition \ref{definable criterion} and $\widetilde{B} \subseteq \Omega$ is a definable subset. There exists a simple criterion of definability of its image $B:=\pi(\widetilde{B}) \subseteq S$.  Let $\Gamma_B \subseteq \Gamma$ be the image of $\pi_1(B)$ under the composition $\pi_1(B) \to \pi_1(S) \to \Gamma$. Then $\pi^{-1}(B)=\bigsqcup_{[\gamma] \in \Gamma_B \backslash \Gamma} (\gamma \cdot \widetilde{B})$, where $\gamma$ runs through representatives of the classes in the quotient $\Gamma_B \backslash \Gamma$. Each $\gamma \cdot \widetilde{B}$ is definable, and the disjoint union of a collection of definable sets is definable if and only if only finitely many of them is nonempty. Therefore, $\pi^{-1}(B) \cap \Xi = \bigsqcup (\gamma \cdot \widetilde{B} \cap \Xi) $ is definable if and only if there is only finitely many classes $[\gamma] \in \Gamma_B \backslash \Gamma$ such that $\gamma \cdot \widetilde{B} \cap \Xi$ is nonempty.

\end{rmk}

\section{Higher Albanese manifolds}\label{higher albanese basic}

\subsection{Malcev completions}\label{malcev}

We make a brief algebraic interlude to recall the theory of Maltsev completions. For more detailed introduction into the subject see \cite[Appendix A]{ABCKT}, \cite[Section 2]{Merk} or \cite[Appendix A]{Quil}.

Let $\Gamma$ be a group. We denote its \emph{lower central series} by $\Gamma_s$, that is, $\Gamma_0=\Gamma$ and $\Gamma_{s+1}=[\Gamma_s, \Gamma]$. We also denote $\Gamma^s:=\Gamma/{\Gamma_s}$. In particular, $\Gamma^1=\Gamma^{\operatorname{ab}}$ is the abelinisation. 

We denote the canonical projections $\Gamma \to \Gamma^s$ by $\varpi_{\Gamma}^s$.

Recall that $\Gamma$ is said to be \emph{nilpotent} if there exists  $s$, such that $\Gamma_k=\{e\}$ for every $k >s$ (equivalently, $\Gamma^k=\Gamma$ for $k>s$). The minimal such $s$ is call  the \emph{nilpotency class} of $\Gamma$ and is denoted by $\nilp(\Gamma)$ (in this case we also say that $\Gamma$ is \emph{$s$-step nilpotent}).

If $\Gamma \to \Gamma'$ is an epimorphism of groups and $\Gamma$ is nilpotent, then $\Gamma'$ is nilpotent and $\nilp(\Gamma') \le \nilp(\Gamma)$.

The following proposition is classical.

\begin{prop}\label{nilpotency of  extensions}
Let $\Gamma$ be a torsion-free $s$-step nilpotent group, $A$ an abelian group and 
\[
1 \to A \to \Delta \to \Gamma \to 1
\]
a central extension.
Then $\Delta$ is nilpotent and $s \le \nilp (\Delta) \le s+1$. More precisely, $\nilp(\Delta)=s$ if and only if the class $[\Delta]\in H^2(\Gamma, A)$ pullbacks along some epimorphism $\Gamma \to \Gamma'$ with $\nilp(\Gamma')< \nilp(\Gamma)$ (otherwise, $\nilp(\Delta)=s+1$).
\end{prop}

For arbitrary $\Gamma$ the groups $\Gamma^s$ are always nilpotent. Moreover they are \emph{universal $s$-step nilpotent quotients of $\Gamma$} (i.e. every homomorphism from $\Gamma$ to an $s$-step nilpotent group factorises through  $\varpi_{\Gamma}^s \colon \Gamma \twoheadrightarrow \Gamma^s$).

\begin{thrm}[Malcev, Quillen]\label{malcev exists}
Let $\Gamma$ be a finiely presented group and $\Bbbk$ a field of characteristic zero. There exists a unique up to a canonical isomorphism pro-unipotent pro-algebraic group $\G_{\Bbbk}(\Gamma)$ over $\Bbbk$ and a homomorphism $\mu_{\Gamma} \colon \Gamma \to \G_{\Bbbk}(\Gamma)(\Bbbk)$ with the following properties:
\begin{itemize}
\item[(i)] if $\rho \colon \Gamma \to U(\Bbbk)$ is a Zariski dense representation to a unipotent geometrically connected $\Bbbk$-algebraic group, there exists a surjective morphism of $\Bbbk$-groups $\nu \colon \G \to U$ such that $\rho=\nu \circ \mu_{\Gamma}$;
\item[(ii)] it is functorial in the natural sense: if $\phi \colon \Gamma \to \Gamma'$ is a homomorphism of groups, there exists a morphism of pro-$\Bbbk$-groups $\G(\phi) \colon \G_{\Bbbk}(\Gamma) \to \G_{\Bbbk}(\Gamma')$ such that the diagram
\[
\xymatrix{
\Gamma \ar[r]^{\phi} \ar[d]^{\mu_{\Gamma}} & \Gamma' \ar[d]_{\mu_{\Gamma'}}\\
\G_{\Bbbk}(\Gamma) \ar[r]_{\G(\phi)}& \G_{\Bbbk}(\Gamma')
}
\]
commutes.
\end{itemize}
\end{thrm}

There group $\G_{\Bbbk}(\Gamma)$ is called the \emph{Malcev completion of $\Gamma$ over $\Bbbk$}. It can be viewed abstractly as the Tannakian fundamental group of the category of unipotent $\Gamma$-modules with coefficients in $\Bbbk$, but it also admits an explicit description in the terms of the group ring, see \cite{Quil}.

We collect the basic properties of Malcev completions below:

\begin{prop}\label{malcev properties}
\begin{itemize}
\item[(i)] If $\Gamma$ is commutative, then $\G_{\Bbbk}(\Gamma)=\Gamma \o_{\Z} \Bbbk$;
\item[(ii)] Let $\G^s_{\Bbbk}(\Gamma)$ be the  universal $s$-step nilpotent quotient of $\G_{\Bbbk}(\Gamma)$, that is, $\G^s_{\Bbbk}(\Gamma)=\G_{\Bbbk}(\Gamma)/\left(\G_{\Bbbk}(\Gamma)\right)_s$. Then $\G^s_{\Bbbk}(\Gamma)$ is an algebraic $s$-step unipotent group over $\Bbbk$ and $\varpi^s_{\Gamma} \colon \Gamma \to \Gamma^s$ induces an isomorphism $\G^s_{\Bbbk}(\Gamma) \xrightarrow{\sim} \G_{\Bbbk}(\Gamma^s)$. There is a commutative diagram
\[
\xymatrix{
\Gamma \ar[d]_{\varpi^s_{\Gamma}} \ar[r]^{\mu_{\Gamma}}& \G_{\Bbbk}(\Gamma) \ar[d]^{\varpi^s_{\G_{\Bbbk}(\Gamma)}}\\
\Gamma^s \ar[r]_{\mu_{\Gamma^s}} & \G^s_{\Bbbk}(\Gamma).
}
\]
\item[(iii)] the homomorphism $\mu_{\Gamma^s} \colon \Gamma^s \to \G_{\Bbbk}(\Gamma^s)(\Bbbk)$ has Zariski dense image and its kernel equals the torsion subgroup of $\Gamma^s$;
\item[(iv)] If $\Bbbk_1 \hookrightarrow \Bbbk_2$ is a field extension, then $\G_{\Bbbk_2}(\Gamma) = \G_{\Bbbk_1}(\Gamma) \o \Bbbk_2$;
\item[(v)] if $\Bbbk=\R$, the image of $\mu_{\Gamma^s}$ is a cocompact lattice inside the connected simply connected unipotent Lie group $\G_{\R}^s(\Gamma)$. The quotient is a smooth manifold.
\end{itemize}
\end{prop}
\begin{proof}
Items \textit{(i) - (iii)} follow from the universal property. 

Item \textit{(iv)} is proved in \cite{Hain93}.

The last item is a classical result on lattices in nilpotent Lie groups due  to Malcev, see \cite{Mal} or \cite[Theorem 2.1.]{Ragh}.

\end{proof}

Items \textit{(ii)} and \textit{(iii)} of Proposition \ref{malcev properties} imply that if $\Gamma$ is torsion-free and nilpotent, the homomorphism $\mu_{\Gamma}$ is injective.

Each group $\G^s_{\Bbbk}(\Gamma)$ is a central extension of $\G^{s-1}_{\Bbbk}(\Gamma)$. The groups $\G^j_{\Bbbk}(\Gamma)$ form an inverse system
\begin{equation}\label{tower of groups}
\ldots \to \G^s_{\Bbbk}(\Gamma) \to \G^{s-1}_{\Bbbk}(\Gamma) \to \ldots \to \G^1_{\Bbbk}(\Gamma).
\end{equation}
and $\G_{\Bbbk}(\Gamma)=\varprojlim \G^s_{\Bbbk}(\Gamma)$. 

In what follows, we denote by $\g_{\Bbbk}(\Gamma)$ (resp. $\g^s_{\Bbbk}(\Gamma)$) the Lie algebra of $\G_{\Bbbk}(\Gamma)$ (resp. $\G^s_{\Bbbk}(\Gamma)$). We sometimes omit the field from the notation, when $\Bbbk=\Q$, which is natural in the light of item \textit{(iv)} of Proposition \ref{malcev properties}. We also denote $\z^s_{\Bbbk}(\Gamma):=\ker(\g^s_{\Bbbk}(\Gamma) \to \g^{s-1}_{\Bbbk}(\Gamma))$ and $\mathcal{Z}^s_{\Bbbk}(\Gamma):=\ker(\G^s_{\Bbbk}(\Gamma) \to \G^{s-1}_{\Bbbk}(\Gamma))$. Thus, $\mathcal{Z}^s_{\Bbbk}(\Gamma)=(\G^s_{\Bbbk}(\Gamma))_{s-1}$ is isomorphic to the additive group of a finite-dimensional $\Bbbk$-vector space and $\z^s_{\Bbbk}(\Gamma)$ is its Lie algebra.



\subsection{Mixed Hodge theory of $\pi_1(X;x)$}\label{morgan hain subsec}

Let $X$ be a normal complex algebraic variety and $x \in X$ a base point. The Maltsev completion of $\pi_1(X; x)$ carries a functorial mixed Hodge structure. There are at least three different constructions of it: one is due to Morgan and is based on rational homotopy theory \cite{Morg}; another is due to Hain and is based on Chen's iterated integrals \cite{Hain87}; and the last is due to Simpson and is based on the $\mathbb{C}^{\times}$-action on the category of Higgs bundles \cite{Simp}. We follow Hain's approach, as the construction of higher Albanese manifolds is most natural in it.\footnote{Morgan's approach has an unfortunate disadvantage: the dependence of the mixed Hodge structure on $\mathcal{G}_{\mathbb{Q}}(\pi_1(X;x))$ on the base point $x \in X$ is very implicit; this dependence, however, plays a key role in Hain's theory and the construction of higher Albanese maps. This weakness in Morgan's approach can be fixed using Halperin's augmented version of rational homotopy theory; see \cite{Halp}. The equivalence of the three constructions is folklore and is not present in the literature. A closely related yet different statement about the equality of Morgan's and Hain's mixed Hodge structures on higher rational homotopy groups can be found in the unpublished manuscript \cite{BigRed}. The equivalence of Simpson's and Hain's constructions is claimed in \cite{Simp}.}

In what follows, we denote $\G^s_{\Bbbk}(X;x):=\G^s_{\Bbbk}(\pi_1(X;x))$. Similarly, we write $ \g_{\Bbbk}(X;x):=\g_{\Bbbk}(\pi_1(X;x))$, and so on.

\begin{thrm}[Hain, Morgan]\label{hain morgan}
Let $X$ be a normal complex algebraic variety and $x \in X$ a fixed point. For every $s$ there exists a mixed Hodge structure on the Lie algebra $\g^s(X;x)$ such that:
\begin{itemize}
\item[(i)] the Lie bracket $[-,-] \colon \Lambda^2\g^s_{\Q}(X;x) \to \g^s_{\Q}(X;x)$ is a morphism of mixed Hodge structures;
\item[(ii)] $W_{-1}\g^s_{\Q}(X;x)=\g^s_{\Q}(X;x)$;
\item[(iii)] If $f \colon X \to Y$ is a morphism of algebraic varieties with $f(x) = y$, it induces a morphism of mixed Hodge structures $f_* \colon \g^s_{\Q}(X;x) \to \g^s_{\Q}(Y;y)$;
\item[(iv)] the resulting mixed Hodge structure on $\g^1_{\Q}(X;x)=H_1(X, \Q)$ coincides with the one coming from the Deligne's mixed Hodge structure on $H^1(X, \Q)$ and the isomorphism $H_1(X, \Q)\simeq H^1(X, \Q)^*$.
\end{itemize}
\end{thrm}

Item \textit{(i)} of Theorem \ref{hain morgan} has two important implications.

First, the Hodge filtration $F^{\bullet}\g^s(X;x)$ on $\g^s_{\C}(X;x)$ satisfies $[F^p\g^s(X;x), F^q\g^s(X;x)] \subseteq F^{p+q}\g^s(X;x)$. In particular,  $F^0\g^s(X;x) \subseteq \g^s_{\C}(X;x)$ is a Lie subalgebra.

Second, the lower central series of $\g^s(X;x)$ is a filtration by $\Q$ -Hodge substructures. Thus, $\z^s(X;x)\subseteq \g^s(X;x)$ is a Hodge substructure and $\g^s(X;x) \to \g^s(X;x)/\z^s(X;x)=\g^{s-1}(X;x)$ is a morphism of mixed Hodge substructures.

\subsection{Higher Albanese manifolds}\label{higher albanese subsec}
Theory of higher Albanese manifolds was developed by Hain in Zucker in \cite{HZ}, see also \cite{Hain85}. 

Recall that  $\G_{\Z}^s(X;x)$ denotes the image of the map 
\[
\mu_{\pi_1(X;x)^s} \colon \pi_1(X;x)^s \to \G^s_{\Q}(X;x).
\]
It is a discrete Zariski dense  subgroup of $\G^s_{\R}(X;x)$ (Proposition \ref{malcev properties}, \textit{(v)}).

We denote by $F^0\G \subseteq \G^s_{\C}(X;x)$ the exponent of $F^0\g^s \subseteq \g^s_{\C}(X;x)$.

\begin{df}
The $s$-th Albanese manifold of $(X;x)$ is defined as:
\[
\Alb^s(X;x):=\G^s_{\Z}(X;x) \backslash \G^s_{\C}(X;x)/F^0\G^s(X;x).
\]
\end{df}

If $s=1$, the definition recovers the classical Albanese manifold $\Alb(X)=\Alb^1(X)$ (see Example \ref{albanese is jacobian}).

The tower of central extensions (\ref{tower of groups}) descends to a  holomorphic tower of complex manifolds
\begin{equation}\label{albanese tower}
\ldots \to \Alb^s(X;x) \xrightarrow{p^s} \Alb^{s-1}(X;x) \to \ldots \to \Alb^1(X;x)
\end{equation}
with each $p^s$ being a holomorphic principal $C^s$-bundle for a complex commutative Lie group 
\[
C^s=\left( \G^s_{\Z}(X;x) \cap \mathcal{Z}_{\Q}^s(X;x) \right ) \backslash \mathcal{Z}^s_{\C}(X;x)/ \left ( \mathcal{Z}^s_{\C}(X;x) \cap F^0 \G^s(X;x) \right )
\]

As mentioned above, $\z^s(X;x) \subseteq \g^s(X;x)$ is a sub-$
\Q$-Hodge structure and the exponential map identifies $\mathcal{Z}^s(X;x)$ with the additive group of the underlying vector space $\z^s$. The intersection with the lattice $\G_{\Z}^s$ endows it with a $\Z$-structure and $C^s = J^0\z^s(X;x)$ is the $0$-th Jacobian of the resulting mixed $\Z$-Hodge structure (see Definition \ref{jacobian def}).

From Theorem \ref{hain morgan}, \textit{(iii)}, it follows that every morphism of algebraic varieties $f \colon X \to Y$  with $f(x) =y$ induces a holomorphic map 
\[
\Alb^s(f) \colon \Alb^s(X;x) \to \Alb^s(Y;y).
\]

 In what follows, we will omit the base point from the notation of the higher Albanese manifolds, writing simply $\Alb^s(X)=\Alb^s(X;x)$. This is natural in the light of the following Proposition (\cite[Corollary 5.20]{HZ}). 

\begin{prop}
If $x_1$ and $x_2$ are two points in $X$, there is a canonical biholomorphism $\Alb^s(X;x_1) \xrightarrow{\sim} \Alb^s(X;x_2)$.
\end{prop}

An extremely important feature of Hain's construction is the existence of liftings of  the classical Albanese map $\alb \colon X \to \Alb(X)$ to the upper levels of the Albanese tower (\ref{albanese tower}).

\begin{thrm}[\cite{HZ}]
Let $X$ be a normal quasiprojective variety, and $x \in X$ a base point. There exists a sequence of holomorphic maps $\alb^s \colon X^{\an} \to \Alb^s(X)$ , such that the diagram
\[
\xymatrix{
&\Alb^s(X) \ar[d] \\
& \vdots \ar[d]\\
& \Alb^2(X) \ar[d]\\
X^{\an} \ar[r]_{\alb} \ar[ru]^{\alb^2} \ar[ruuu]^{\alb^s}& \Alb(X)
}
\]
commutes and  $\alb^1=\alb$ coincides with the classical Albanese map. If $f \colon X \to Y$ is a morphism of algebraic varieties with $f(x)=y$, the diagram
\[
\xymatrix{
X \ar[r]^{f} \ar[d]_{\alb_X^s} & Y \ar[d]^{\alb_Y^s}\\
\Alb^s(X) \ar[r]_{\Alb^s(f)} & \Alb^s(Y)
}
\]
also commutes.
\end{thrm}

The explicit construction of the maps $\alb^s$ is rather delicate and uses Chen's iterated integrals. 

\begin{rmk}\label{dependence on a point}
Although the higher Albanese manifolds do not depend on the choice of the base point $x \in X$, the higher Albanese maps do depend on it, as can  already be observed at  $s=1$. Nevertheless, we usually omit the base point from the notation. 
\end{rmk}

\begin{rmk}
The original construction of higher Albanese maps was given by Hain and Zucker in \cite{HZ} under the assumption that $X$ is smooth. Let us sketch the construction for $X$ normal\footnote{ This is a private communication by Richard Hain.}. Let $X$ be a normal algebraic variety and $X^{\circ} \subseteq X$ be the set of its smooth points. Let $\phi \colon Y \to X$ be a resolution of singularities and $Y^{\circ}:=\phi^{-1}(X^{\circ})$, so that $\phi|_{Y^{\circ}} \colon Y^{\circ} \to X^{\circ}$ is an isomorphism. Choose a point $x \in X^{\circ}$ and let $y=\phi^{-1}(x)\in Y^{\circ}$. The map $\phi$ induces a holomorphic map $\Alb^s(\phi) \colon \Alb^s(Y) \to \Alb^s(X)$. Let $\alpha \colon X^{\circ} \to \Alb^s(X)$ be the composition map $\alpha=\Alb^s(\phi) \circ \alb^s_Y \circ \phi^{-1}$ as in the diagram
\[
\xymatrix{
Y^{\circ} \ar[r]& Y \ar[r]^{\alb^s_Y}& \Alb^s(Y) \ar[d]^{\Alb^s(\phi)}\\
X^{\circ} \ar[u]^{\phi^{-1}} \ar[rr]_{\alpha}&&\Alb^s(X)
}
\]
Since $X$ is normal, the holomorphic map $\alpha$ defined on its smooth part $X^{\circ}$ extends globally as $\alb^s \colon X \to \Alb^s(X)$.
\end{rmk}

\section{Nil-Jacobians}\label{niljacobians section}
\subsection{Nil-Jacobians}\label{niljacobians basic}

In this section we introduce the notion of a \emph{nil-Jacobian}, which interpolates between the notion of a Jacobian of a mixed Hodge structure and the notion of a mixed Hodge variety.

\begin{df}
Let $\mathbf{W}$ be a connected simply connected unipotent group over $\Q$ and $\mathfrak{w}$ its Lie algebra. Under a \emph{Hodge structure on $\mathbf{W}$} we understand a graded polarisable mixed $\Q$-Hodge structure $(W_{\bullet}\mathfrak{w}, F^{\bullet}\mathfrak{w})$ on $\mathfrak{w}$ such that
\begin{itemize}
\item the Lie bracket is a morphism of mixed Hodge structures $\Lambda^2\mathfrak{w} \to \mathfrak{w}$;
\item $W_{-1}\mathfrak{w}=\mathfrak{w}$.
\end{itemize}
Let $\Gamma \subseteq \mathbf{W}(\Q)$ be a discrete Zariski dense subgroup. A \emph{nil-Jacobian} is the double quotient 
\[
N_{\mathbf{W}}=\Gamma \backslash \mathbf{W}(\C)/F^0\mathbf{W},
\]
where $F^0\mathbf{W}=\exp(F^0\mathfrak{w})$. A \emph{morphism of nil-Jacobians} is a continuous map 
\[
 \Gamma \backslash \mathbf{W}(\C)/F^0\mathbf{W} =N_{\mathbf{W}}\xrightarrow{f} N_{\mathbf{W}'}=\Gamma'\backslash \mathbf{W}'(\C)/F^{0}\mathbf{W}'
\]
which lifts to a homomorphism $\widetilde{f} \colon \mathbf{W} \to \mathbf{W}'$ of algebraic groups over $\Q$ such that $\widetilde{f}(\Gamma)\subseteq \Gamma'$ and $\operatorname{Lie}(\widetilde{f}) \colon \mathfrak{w} \to \mathfrak{w}'$ is a morphism of mixed Hodge structures.
\end{df}

The condition $W_{-1}\mathfrak{w}=\mathfrak{w}$ guarantees that a nil-Jacobian $\Gamma \backslash \mathbf{W}(\C)/F^0\mathbf{W}$ is a smooth complex manifold (cf. Proposition \ref{carlson theorem}). A morphism of nil-Jacobians induces a holomorphic map between underlying complex manifolds.

The main examples of nil-Jacobians are the following:
\begin{itemize}
\item if $V=(V_{\Z}, W_{\bullet}V_{\Q}, F^{\bullet}V_{\C})$ is a mixed $\Z$-Hodge structure with $W_{-1}V=V$, then its $0$-th Jacobian $J^0V$ is a nil-Jacobian for $\mathbf{W}$ being the additive group of $V_{\Q}$ and $\Gamma=V_{\Z}$ ;
\item  a higher Albanese manifold $\Alb^s(X)$ is a nil-Jacobian for $\mathbf{W}=\G^s_{\Q}(X;x)$ and $\Gamma=\G^s_{\Z}(X;x)$;
\item let $M=\Gamma_G \backslash \mathcal{D}$ be a mixed Hodge variety that underlies a mixed Hodge datum $(\mathbf{G}, X_{\mathbf{G}}, \mathcal{D})$. Let $\sigma \colon M \to M_{\sigma}$ be the purification map (see subsection \ref{purification subsec}). Let $x \in M_{\sigma}$ and $N_{x}:=\sigma^{-1}(x)$. Then $N_{x}$ is a nil-Jacobian for the group $\mathbf{W}=\mathbf{U}$ (the unipotent radical of $\mathbf{G}$) and $\Gamma=\Gamma_G \cap \mathbf{U}(\Q)$.
\end{itemize}

For a nil-Jacobian $N_{\mathbf{W}}=\Gamma \backslash \mathbf{W}(\C)/F^0\mathbf{W}$ we say that $N_{\mathbf{W'}} \subseteq N_{\mathbf{W}}$ if it is of the form $\Gamma' \backslash \mathbf{W}'(\C)/F^0\mathbf{W}'$, where $\mathbf{W}' \subseteq \mathbf{W}$ is a closed $\Q$-subgroup whose Lie algebra $\mathfrak{w}'$ is a sub-Hodge structure of $\mathfrak{w}=\operatorname{Lie}(\mathbf{W})$ and $\Gamma'=\Gamma \cap \mathbf{W}(\Q)$.

A connected finite cover of a nil-Jacobian is again a nil-Jacobian. A product of nil-Jacobians is a nil-Jacobain in a natural way.

\begin{prop}\label{graphsofmaps_prop}
Let $N_{\mathbf{W}}=\Gamma \backslash \mathbf{W}/F^0\mathbf{W}$ and $N_{\mathbf{W}'}=\Gamma' \backslash \mathbf{W}'/F^0\mathbf{W}'$ be two nil-Jacobians. Suppose that  $f \colon N_{\mathbf{W}} \to N_{\mathbf{W}'}$ is a morphism of nil-Jacobians. Then its graph $N_{f} \subseteq N_{\mathbf{W}} \times N_{\mathbf{W}'}$ is a sub-nil-Jacobian of  $N_{\mathbf{W}} \times N_{\mathbf{W}'}$.
\end{prop}
\begin{proof}
Suppose $f$ is a morphism of nil-Jacobians. It comes from a morphism of $\Q$-algebraic groups $\widetilde{f} \colon \mathbf{W} \to \mathbf{W}'$. Then $\mathbf{W}_f:=\{(x, f(x))\} \subseteq \mathbf{W} \times \mathbf{W}'$ is a Zariski closed rational subgroup of $\mathbf{W}\times \mathbf{W}'$. Thus, $\Gamma_f:=\mathbf{W}_f \cap (\Gamma \times \Gamma')$ is a cocompact lattice of $\mathbf{W}_f$.

Finally, its Lie algebra $\mathfrak{w}_f:=\operatorname{Lie}(\mathbf{W}_f)$ is the graph of $\operatorname{Lie}(\widetilde{f})$ in $\mathfrak{w}\oplus \mathfrak{w}'$. Since $\operatorname{Lie}(\widetilde{f})$ is a morphism of mixed Hodge structures, it commutes with the actions of the Deligne torus $\mathbb{S}$ and its graph is an $\mathbb{S}$-submodule, i.e. a sub-Hodge structure of $\mathfrak{w} \oplus \mathfrak{w}'$. With respect to this Hodge structure, $F^0\mathbf{W}_f=\mathbf{W}_f(\C) \cap F^0(\mathbf{W} \times \mathbf{W}')$ and 
\[
N_f=\Gamma_f \backslash \mathbf{W}_f/F^0\mathbf{W}_f \subseteq N_{\mathbf{W}} \times N_{\mathbf{W}'}
\]
is a sub-nil-Jacobian.

\end{proof}

Our discussion of the Albanese tower applies in the abstract setting of nil-Jacobians. Namely, if $(W_{\bullet}\mathfrak{w}, F^{\bullet}\mathfrak{w})$ is a mixed Hodge structure on a unipotent group $\mathbf{W}$, then the lower central series $\{0\} \subset \mathfrak{w}_{1} \subset \ldots \subset \mathfrak{w}_{s-1} \subset \mathfrak{w}_s=\mathfrak{w}$ give a filtration by Hodge substructures and the groups $\mathbf{W}^j:=\mathbf{W}/\mathbf{W}_{j+1}$ inherit Hodge structures (we still have $W_{-1}\mathfrak{w}^j=\mathfrak{w}^j$, as this property is preserved under taking quotients of mixed Hodge substructures). The lower central series filtration on $\mathbf{W}(\Q)$ restricts to the lower central series filtration on $\Gamma$ and the projections $\mathbf{W}^j \to \mathbf{W}^{j-1}$ descend to morphisms of nil-Jacobians
\[
N^j_{\mathbf{W}}:=\Gamma^j \backslash \mathbf{W}^j(\C)/F^0\mathbf{W}^j \xrightarrow{p^j_{\mathbf{W}}} \Gamma^{j-1}\backslash \mathbf{W}^{j-1}(\C)/F^0\mathbf{W}^{j-1}=: N^{j-1}_{\mathbf{W}}.
\]

To summarise, we get the following.

\begin{prop}\label{central towers exist}
 Let $N_{\mathbf{W}}=\Gamma \backslash \mathbf{W}(\C)/F^0\mathbf{W}$ be a nil-Jacobian. Then there exist a sequence of connected commutative complex Le groups $C^j_{\mathbf{W}}, j = 1, \ldots s,$ and a diagram
\begin{equation}\label{central tower}
N_{\mathbf{W}}=N^s_{\mathbf{W}} \xrightarrow{p^s_{\mathbf{W}}} N^{s-1}_{\mathbf{W}} \to \ldots \to N^2_{\mathbf{W}} \xrightarrow{p^2_{\mathbf{W}}} N^1_{\mathbf{W}} \xrightarrow{p^1_{\mathbf{W}}} N^0_{\mathbf{W}}=\{\operatorname{pt}\},
\end{equation}
where $N^j_{\mathbf{W}}$ are nil-Jacobians and $p^j_{\mathbf{W}}$ are morphisms of nil-Jacobians which are holomorphic principal $C^j_{\mathbf{W}}$-bundles. Moreover, each $C^j_{\mathbf{W}}$ is isomorphic to the Jacobian $J^0Z^j_{\mathbf{W}}$, where $Z^j_{\mathbf{W}}$ is a non-zero mixed $\Z$-Hodge structure with $W_{-1}Z^j_{\mathbf{W}}=Z^j_{\mathbf{W}}$. Moreover, the diagram \ref{central tower}  is functorial in the following sense: if $f \colon N_{\mathbf{W}} \to N_{\mathbf{W}'}$ is a morphism of  nil-Jacobians, then there exist morphisms of mixed Hodge structures $g^j \colon Z^j_{\mathbf{W}} \to Z^j_{\mathbf{W}'}$ and $J^0g^j$-equivariant morphisms of nil-Jacobians $f^j \colon N^j_{\mathbf{W}} \to N^j_{\mathbf{W}'}$ such that the diagram
\[
\xymatrix{
N_{\mathbf{W}}^s \ar[d]_{f^s} \ar[r]^{p^s_{\mathbf{W}}} & N_{\mathbf{W}}^{s-1} \ar[d]_{f^{s-1}} \ar[r]^{p^{s-1}_{\mathbf{W}}} & \ldots \ar[r] & N^2_{\mathbf{W}} \ar[d]_{f^2} \ar[r]^{p^2_{\mathbf{W}}}& N^1_{\mathbf{W}} \ar[d]_{f^1}\\
N_{\mathbf{W}'}^s  \ar[r]_{p^s_{\mathbf{W}'}} & N_{\mathbf{W}'}^{s-1}  \ar[r]_{p^{s-1}_{\mathbf{W}'}} & \ldots \ar[r] & N^2_{\mathbf{W}'} \ar[r]_{p^2_{\mathbf{W}'}}& N^1_{\mathbf{W}'}
}
\]
commutes.
\end{prop}

Here, as before, $Z^j_{\mathbf{W}}=\ker(\mathbf{W}^j \to \mathbf{W}^{j-1})$ are mixed $\Q$-Hodge structures with the $\Z$-structure $\Gamma^j_{Z}=\Gamma_Z \cap Z^j_{\mathbf{W}}=\ker(\Gamma^j \to \Gamma^{j-1})$. 

We refer to the diagram (\ref{central tower}) as the \emph{central tower} of a nil-Jacobian. The higher Albanese tower $\Alb^s(X) \xrightarrow{p^s} \Alb^{s-1}(X) \to \ldots \to \Alb(X)$ is  the central tower of the nil-Jacobian $\Alb^s(X)$.

\subsection{The Embedding Theorem}\label{embedding subsec}

As we mentioned above, a fibre of the purification map of a mixed Hodge variety is a nil-Jacobian. In this subsection we show that, up to a finite cover, every nil-Jacobian can be realised as a sub-nil-Jacobian of a fibre of the purification map on some mixed Hodge variety.

\begin{thrm}\label{embedding theorem}
Let $N_{\mathbf{W}}=\Gamma \backslash \mathbf{W}(\C)/F^0\mathbf{W}$ be a nil-Jacobian. Then there exists a nil-Jacobian $\widehat{N}_{\mathbf{W}}$ which is a finite cover of $N_{\mathbf{W}}$, a mixed Hodge variety $M$, and a fibre $N_x=\sigma^{-1}_M(x)$ of its purification map $\sigma_M \colon M \to M_{\sigma}$ such that $\widehat{N}_{\mathbf{W}}$ admits an injective morphism of nil-Jacobians $j \colon \widehat{N}_{\mathbf{W}} \hookrightarrow N_x$.
\end{thrm}

Theorem \ref{embedding theorem} can be seen as a Hodge-theoretic analogue of Ado's Theorem, which says that every connected unipotent Lie group can be embedded as a closed subgroup to the group of upper-triangular matrices.

\begin{proof}[Proof of Theorem \ref{embedding theorem}]
\textbf{Step 1.} \textit{Constructing the algebraic group.} Let $\mathfrak{w}$ be the Lie algebra of $\mathbf{W}$ and $\mathbf{P}$ its Mumford-Tate group. Recall, that this means that $\mathbf{P}$ is the $\Q$-Zariski closure of $h(\mathbb{S})$ inside $\GL(\mathfrak{w})$, where $h \colon \mathbb{S} \to \GL(\mathfrak{w})$ is the Hodge cocharacter determing the mixed Hodge structure on $\mathfrak{w}$ (see subsection \ref{hodge structures subsec}). Since $\mathfrak{w}$ is a Lie algebra in the category of mixed Hodge structures, the Lie bracket is preserved by the action of $h(\mathbb{S})$ and $\mathbf{P} \subseteq \GL(\mathfrak{w})$ acts on $\mathfrak{w}$ by Lie automorphisms. This action induces an action of $\mathbf{P}$ on $\mathbf{W}$ by group automorphisms. Set
\[
\mathbf{G}:= \mathbf{W} \rtimes \mathbf{P}.
\]
This is a connected algebraic group over $\Q$ and $\mathbf{W}$ is contained inside its unipotent radical $\mathbf{U} \subset \mathbf{G}$. The group $\mathbf{G}$ admits no non-constant homomorphism to an abelian variety, and therefore is linear by Chevalley - Barotti - Rosenlicht Theorem (\cite{Con}).
\\

\textbf{Step 2.} \textit{Constructing the mixed Hodge datum.} The Hodge cocharacter $h \colon \mathbb{S_{\C}} \to \mathbf{P}_{\C}$ admits a canonical lift to a Hodge cocharacter $\widehat{h} \colon \mathbb{S}_{\C} \to \mathbf{G}_{\C}$.  Let $X_{\mathbf{G}}$ be the conjugacy class of $\widehat{h}$ and $\mathcal{D}$ the corresponding connected component of $\mathcal{D}_{\mathbf{G}, X_{\mathbf{G}}}$. Thus, $(\mathbf{G}, X_{\mathbf{G}}, \mathcal{D})$ is a mixed Hodge datum. The cocharacter $\widehat{h} \in X_{\mathbf{G}}$ determines a mixed Hodge structure on $\g=\operatorname{Lie}(\mathbf{G})$ with $W_{-1}\g=\g$ and the natural embedding $\mathfrak{w} \hookrightarrow \g$ is a morphism of mixed Hodge structures with respect to it.
\\

\textbf{Step 3.} \textit{Constructing the arithmetic subgroup.} Every discrete  Zariski dense subgroup of a solvable group is arithmetic (\cite[Theorem 4.34]{Ragh}), thus $\Gamma \subset \mathbf{W}$ is arithmetic. As we explained in \textbf{Step 1}, the group $\mathbf{G}$  is linear. Let $\mathbf{G}^{+}(\R)$ be the connected component of $\mathbf{G}(\R)$ corresponding to $\mathcal{D}$. Choose a faithful representation $\rho \colon \mathbf{G}^{+}(\Q) \to \GL(\Q^r)$. Choose any $\Z$-structure $\Z^r \subset \Q^r$ and set 
\[
\Delta:=\rho^{-1}(\rho(\mathbf{G^{+}(\Q)}) \cap \SL_r(\Z)).
\]
This is an arithmetic subgroup of $\mathbf{G}$.

We claim that $\Gamma_{1}:=\Delta \cap \Gamma$ is a finite index subgroup in $\Gamma$. Indeed, 
\[
\Gamma_1= \Gamma \cap (\rho|_{\mathbf{W}})^{-1} \Big ( \im \rho|_{\mathbf{W}} \cap \SL_r(\Z) \Big)
\]
is arithmetic in $\mathbf{W}$ and every two arithmetic subgroups in a $\Q$-group are commensurable.
\\

\textbf{Step 4.} \textit{Constructing the mixed Hodge variety and the embedding}. Now we can set $M:=\Delta \backslash \mathcal{D}$. This is a mixed Hodge variety. Let $\sigma \colon M \to M_{\sigma}$ be its purification map and $x \in M_{\sigma}$ be the image of $[\widehat{h}] \in X_{\mathbf{G}}$ under the composition $X_{\mathbf{G}} \to \mathcal{D} \to M \to M_{\sigma}$. The fibre $N_x:=\sigma^{-1}(\{x\})$ is of the form
\[
N_x=\Delta_U \backslash \mathbf{U}(\C)/F^0_{\widehat{h}}\mathbf{U},
\]
where $\Delta_U = \Delta \cap \mathbf{U}(\Q)$ and $F^{0}_{\widehat{h}}\mathbf{U}=\exp(F^0_{\widehat{h}}\g)\cap \mathbf{U}(\C)$ for the Hodge structure $F^{\bullet}_{\widehat{h}}\g$ on $\g$ induced by $\operatorname{ad} \circ \widehat{h}$.

Observe that $\mathbf{W}$ is a closed subgroup of $\mathbf{U}$ and $\mathbf{W} \cap \Delta_U=\Gamma_1$, which is a finite index subgroup of $\Gamma$ (see \textbf{Step 3}).  

Since $(W_{\bullet}\mathfrak{w}, F^{\bullet}\mathfrak{w})$ is a mixed sub-Hodge-structure of $(W_{\bullet}, F^{\bullet}_{\widehat{h}}\g)$, we have 
\[
F^0\mathbf{W}=\mathbf{W}(\C)\cap F^0_{\widehat{h}}\mathbf{G} = \mathbf{W}(\C) \cap F^{0}_{\widehat{h}}\mathbf{U}.
\]
Let $\widehat{N}_{\mathbf{W}}:=\Gamma_1 \backslash \mathbf{W}(\C)/F^0\mathbf{W}$. This is a nil-Jacobian that is a finite cover of $N_{\mathbf{W}}$ and the embedding $\mathbf{W} \hookrightarrow \mathbf{U}$ extends to an injective morphism of nil-Jacobians
\[
N'_{\mathbf{W}}=\Gamma_1 \backslash \mathbf{W}(\C)/F^0\mathbf{W} \hookrightarrow \Delta_U\backslash \mathbf{U}(\C)/F^0_{\widehat{h}}\mathbf{U}=N_x.
\]
\end{proof}

\subsection{o-minimal geometry of nil-Jacobians}\label{definability of nil-Jacobians}
We always view mixed Hodge varieties as definable complex analytic spaces with the definable complex analytic space structure of \cite{BBKT}, see Theorem \ref{period maps o-minimal}. If $M$ is a mixed Hodge variety, its purification map $\sigma \colon M \to M_{\sigma}$ is definable and a fibre $N_x :=\sigma^{-1}(\{x\})$ inherits a definable complex manifold structure.

In this subsection, we realise the category of nil-Jacobians as a subcategory of the category of $\R_{\alg}$-definable complex analytic spaces. More precisely, we  prove the following Theorem.

\begin{thrm}\label{nil-jacobians o-min}
Every nil-Jacobian can be endowed with a structure of $\R_{\alg}$-definable complex manifold in such a way that:
\begin{itemize}
\item[(i)] morphisms of nil-Jacobians are definable;
\item[(ii)]  for each $j, \ 1 \le j \le s$ the morphisms $p^j_{\mathbf{W}} \colon N^j_{\mathbf{W}}\to N^{j-1}_{\mathbf{W}}$ in the central tower (\ref{central tower}) are definable;
\item[(iii)] if $N_{\mathbf{W}} =N_x \subset M$ is a fibre of the purification map of a mixed Hodge variety, this structure coincides with the one inherited from $M$;
\item[(iv)] for each $j$ the  group $C^j_{\mathbf{W}}$ acting on fibres of $p^j_{\mathbf{W}}$ can be endowed with a  definable complex Lie group structure in such a way that the action $C^j_{\mathbf{W}} \times N^j_{\mathbf{W}} \to N^j_{\mathbf{W}}$ is definable. 
\end{itemize}
\end{thrm}

Observe that \textit{(ii)} is just a special case of \textit{(i)}. Items \textit{(i)-(iii)} of the Theorem above will be proved in this section using the Embedding Theorem (Theorem \ref{embedding theorem}) and the following Lemma.

\begin{lemma}\label{subniljac}
Let $M$ be a mixed Hodge variety, $N_x=\sigma^{-1}(\{x\})\subset M$ a fibre of the purification map and $N_{\mathbf{W}} \subseteq N_x$ a sub-nil-Jacobian. Then $N_{\mathbf{W}}$ is a definable subset of $M$.
\end{lemma}

We prove Lemma \ref{subniljac} and Theorem \ref{nil-jacobians o-min} in few steps. First, let us recall the notations. 

We fix a connected mixed Hodge datum $(\mathbf{G}, X_{\mathbf{G}}, \mathcal{D})$ and a mixed Hodge variety $M=\Gamma \backslash \mathcal{D}$.  We denote by $\pi \colon \mathcal{D} \to M$ the projection. As before, we denote by $\mathbf{U} \subset \mathbf{G}$ the unipotent radical and  by $\mathbf{H}:=\mathbf{G}/\mathbf{U}$ the reductive quotient. We denote by $\widetilde{G}$ the preimage of $\mathbf{H}(\R)$ in $\mathbf{G}(\C)$ so that $X_{\mathbf{G}}$ is the $\widetilde{G}$-conjugacy class of a reference Hodge cocharacter $h_0 \colon \mathbb{S}_{\C} \to \mathbf{G}_{\C}$. Without loss of generality we may assume that $h_0$ is defined over $\R$.

For $h \in X_{\mathbf{G}}$ we denote by $[h]$ its image in $\mathcal{D}$. By $\mathcal{D}_{\R}$ we denote the $\mathbf{G}(\R)$-orbit of $[h_0]$ in $\mathcal{D}$. It is precisely the set of mixed Hodge structures in $\mathcal{D}$ that split over $\R$ (see subsection \ref{about splitting subsec}). We also denote $M_{\R}:=\pi(\mathcal{D}_{\R})$.

Let $N_x \subseteq M$ be a fibre of a purification map and $N_{\mathbf{W}} \subset N_x$ a sub-nil-Jacobian. This means that $N_{\mathbf{W}}=\Gamma_W \backslash \mathbf{W}(\C)/F^0\mathbf{W}$, where 
\begin{itemize}
\item $\mathbf{W} \subseteq \mathbf{U}$ is a closed connected subgroup over $\Q$ and its Lie algebra $\mathfrak{w}$ is a  sub-$\Q$-Hodge structure of $\u=\operatorname{Lie}(\mathbf{U})$ with respect to the mixed Hodge structure on $\u$ induced by a Hodge cocharacter $h, \ [h] \in N_x$;
\item $\Gamma_W=\Gamma \cap \mathbf{W}(\Q)$;
\item $F^0\mathbf{W}=\exp(F^0\mathfrak{w}) =\mathbf{W}(\C) \cap F^0\widetilde{G}$.
\end{itemize}

 If  $\mathcal{D}_{\mathbf{W}}$  denotes the orbit of $[h_0] \in \mathcal{D}$ over the action of $\mathbf{W}(\C) \subset \widetilde{G}$, then $N_{\mathbf{W}}=\pi(\mathcal{D}_{\mathbf{W}})$. We also denote $\mathcal{D}_{\mathbf{W},\R}:=\mathbf{W}(\R) \cdot[h_0]$ and $N_{\mathbf{W}, \R}:=\pi(\mathcal{D}_{\mathbf{W},\R})$. Observe that $\mathcal{D}_{\mathbf{W},\R}=\mathcal{D}_{\mathbf{W}} \cap \mathcal{D}_{\R}$ and $N_{\mathbf{W},\R}=N_{\mathbf{W}} \cap M_{\R}$.

\begin{prop}\label{compactness}
$N_{\mathbf{W}, \R}$ is compact.
\end{prop}
\begin{proof}
The surjective map $\mathbf{W}(\R) \to N_{\mathbf{W}, \R}$ factorises through $\mathbf{W}(\R) \to \Gamma_{W}\backslash \mathbf{W}(\R)$. At the same time, $\Gamma_W \backslash \mathbf{W}(\R)$ is compact (\cite{Mal}).
\end{proof}

\begin{prop}\label{splitting preserves subgroups}
Let $r \colon \mathcal{D} \to \mathcal{D_{\R}}$ be the $\operatorname{sl_2}$-splitting. Then $r(\mathcal{D}_{\mathbf{W}})\subseteq \mathcal{D}_{\mathbf{W}, \R}$.
\end{prop}
\begin{proof}
Suppose that $[h] \in \mathcal{D}_{\mathbf{W}}$. This means that $[h]$ can be written as $\mathbf{w}\cdot [h_0]$ for the reference Hodge cocharacter $[h_0] \in \mathcal{D}_{\R}$ and an element $\mathbf{w}\in\mathbf{W}(\C)$. Then the complex conjugate mixed Hodge structure corresponds to the class of the Hodge cocharacter 
\[
\overline{[h]}=\overline{\mathbf{w} \cdot [h_0]}=\overline{\mathbf{w}} \cdot [h_0],
\]
where $\mathbf{w} \mapsto \overline{\mathbf{w}}$ is the standard complex conjugation on $\mathbf{W}(\C)=\mathbf{W}(\R) \o \C$.

At the same time,
\[
\overline{\mathbf{w}\cdot [h]}=e^{-2\i \delta_{[h]}}\mathbf{w}\cdot[h],
\]
thus $e^{-2\i \delta_{[h]}}=\overline{\mathbf{w}\mathbf{w}^{-1}}$ belongs to $\mathbf{W}(\C)$ and $\delta_{[h]}$ is an element of  $\mathfrak{w}=\operatorname{Lie}(\mathbf{W})$. (See subsection \ref{about splitting subsec} for a remainder on $\operatorname{sl}_2$-splitting and the elements $\delta_{[h]}$ and $\zeta_{[h]}$). Since $\zeta_{[h]}$ is given by Lie polynomials in $\delta_{[h]}$, it is also contained in  $\mathfrak{w}$. 

It follows, that $r(\mathbf{w}\cdot h_0)=(e^{-\i\zeta_{[h]}}e^{-\i\delta_{[h]}}\mathbf{w})\cdot [h_0]$ is in $\mathcal{D}_{\mathbf{W}}$. Since, moreover, $r(\mathbf{w}\cdot h_0) \in \mathcal{D}_{\R}$, it belongs to $\mathcal{D}_{\mathbf{W}, \R}=\mathcal{D}_{\mathbf{W}} \cap \mathcal{D}_{\R}$. 
\end{proof}

\begin{proof}[Proof of Lemma \ref{subniljac}]

Let $\Xi_{\R} \subseteq \mathcal{D}_{\R}$ be a definable fundamental domain for the $\Gamma$-action. Let $\Xi=r^{-1}(\Xi_{\R})$. By Proposition \ref{definable criterion}, $N_{\mathbf{W}}$ is definable in $M$ if and only if $\pi^{-1}(N_{\mathbf{W}}) \cap \Xi$ is definable in $\mathcal{D}$.

Recall that $\pi$ denotes the projection from the mixed Hodge domain $\mathcal{D}$ to the mixed Hodge variety $M$. Then
\[
\pi^{-1}(N_{\mathbf{W}})=\bigsqcup_{[\gamma]\in \Gamma_W \backslash \Gamma} \gamma \cdot \mathcal{D}_{\mathbf{W}},
\]
where $\gamma$ runs through representatives of the classes of the quotient $\Gamma_{W} \backslash \Gamma$, and by Remark \ref{definable criterion remark} it is sufficient to show that the set of classes 
\[\{[\gamma]\in \Gamma_W \backslash \Gamma \ | \ \gamma \cdot \mathcal{D}_{\mathbf{W}} \cap \Xi \neq \varnothing\}
\]
is finite. 

Suppose $\gamma \cdot \mathcal{D}_{\mathbf{W}} \cap \Xi$ is non-empty. Then $r(\gamma \cdot \mathcal{D}_{\mathbf{W}})\cap \Xi_{\R}$ is non-empty. 

Recall that $r \colon \mathcal{D} \to \mathcal{D}_{\R}$ is $\mathbf{G}(\R)$-equivariant. In particular, it is $\Gamma$-equivariant, as $\Gamma \subset \mathbf{G}^{+}(\Q) \subset \mathbf{G}(\R)$. Combining this observation with  Propositon \ref{splitting preserves subgroups}, we deduce that $r(\gamma \cdot \mathcal{D}_{\mathbf{W}})\subseteq  \gamma \cdot \mathcal{D}_{\mathbf{W}, \R}$. Therefore, if $\gamma \cdot \mathcal{D}_{\mathbf{W}}$ intersects $\Xi$, then $\gamma \cdot \mathcal{D}_{\mathbf{W}, \R}$ intersects $\Xi_{\R}$.

We reduced the problem to showing that finitely many classes $[\gamma]$ in $\Gamma_W \backslash \Gamma$ such that $\gamma \cdot \mathcal{D}_{\mathbf{W}, \R} \cap \Xi_{\R}$ is non-empty. For every such class $[\gamma]$ the intersection $\gamma \cdot \mathcal{D}_{\mathbf{W}, \R} \cap \Xi_{\R}$ gives another connected component in $\pi^{-1}(N_{\mathbf{W}, \R}) \cap \Xi_{\R}$. The projection 
\[
\pi|_{\Xi_{\R}} \colon \pi^{-1}(N_{\mathbf{W}, \R}) \cap \Xi_{\R} \to N_{\mathbf{W}, \R} \cap \pi(\Xi_{\R})
\]
is a definable homeomorphism. The image $\pi(\Xi_{\R})\subseteq M_{\R}$ is $\R_{\alg}$-definable.

At the same time, $N_{\mathbf{W}, \R} \subseteq M_{\R}$ is a compact analytic submanifold by Proposition \ref{compactness} and its intersection with $\pi(\Xi_{\R})$ is an $\R_{\an}$-definable subset of $M_{\R}$. In particular, it has only finitely many connected components. We deduce that $\pi^{-1}(N_{\mathbf{W}, \R}) \cap \Xi_{\R}$ has finitely many connected components as well.
\end{proof}

For the proof of the last item of Theorem \ref{nil-jacobians o-min} we will also need the following linear algebraic lemma.

\begin{lemma}\label{linear algebra}
Let $V$ be a finite-dimensional $\Q$-vector space. Let $F \subset V_{\C}$ be a complex subspace of its complexification such that $V_{\R} \cap F= \{0\}$. Denote by $q$ the projection $q \colon V_{\C} \to V_{\C}/F$. Then there exists a rational linear subspace $L \subset V$ such that $V_{\C}/F=q(V_{\R}) \oplus q(\i L_{\R})$.
\end{lemma}
\begin{proof}
Since $V_{\R} \cap F=\{0\}$, we know that $F \cap \overline{F}=\{0\}$. Let $F_{r}:=(F\oplus \overline{F})\cap V_{\R}$. This is a real linear subspace and $(F \oplus \overline{F})=F_{r}\o \C$, in particular, its real dimension equals the complex dimension of $F$.

Let $L \subset V$ be any rational complement of $F_{r}$, i.e. $V_{\R}=L_{\R} \oplus F_{r}$. The restriction of $q$ on $V_{\R}$ is injective. Therefore the restriction of $q$ on $\i L_{\R}$ is injective as well: the map $q$ is $\C$-linear and 
\[
\ker q|_{\i L_{\R}}=\i L_{\R} \cap F=\i(L_{\R}\cap  F)=L_{\R} \cap F=\{0\}.
\]
(recall that $F \subset V_{\C}$ is complex linear, thus $\i F=F$). We claim that $q(\i L_{\R}) \cap q(V_{\R})=0$. Indeed, suppose $l \in L_{\R}$ is such that $q(\i l)=q(v)$ for some $v \in V_{\R}$. This means that
\[
\i l + v= f
\]
for some $f \in F$. Hence,
\[
f-\overline{f}=\i l + v - \overline{(\i l +v)}=2\i l.
\]
It follows that $2 \i l \in F \oplus \overline{F}$, thus $l \in F_{\R}$ and $l=0$.

It is left to count the dimensions. Let $n=\dim_{\Q} V$ and $k=\dim_{\C}F$. Then $\dim_{\R} V_{\C}/F=2(n-k)$. Since
\[
\dim_{\R} q(\i L_{\R}) =\dim_{\R}L_{\R} =\dim V_{\R}-\dim F_{r} = n-2k,
\]
we have 
\[
\dim_{\R} q(V_{\R})+\dim_{\R}q(\i L_{\R})=n+n-2k=2(n-k)=\dim_{\R}(V_{\C}/F)
\]
as required.
\end{proof}

\begin{proof}[Proof of Theorem \ref{nil-jacobians o-min}]
Let $N_{\mathbf{W}}$ be a nil-Jacobian. By Theorem \ref{embedding theorem} there exists a finite cover $\widehat{N}_{\mathbf{W}}$ of $N_{\mathbf{W}}$ that admits an embedding  $j \colon \widehat{N}_{\mathbf{W}} \hookrightarrow N_x$ where $N_x$ is a fibre of the purification map of a mixed Hodge variety. By Lemma \ref{subniljac} the image of this embedding is definable in $M$, so $\widehat{N}_{\mathbf{W}}$ inherits a definable complex manifold structure. By item \textit{(i)} of Proposition \ref{finite covers} it defines a definable complex manifold structure on $N_{\mathbf{W}}$.

Let us check that morphisms of nil-Jacobians are definable. As a consequence, we will see that the constructed definable manifold structure does not depend on the choice of the cover $\widehat{N}_{\mathbf{W}}$ and the embedding $j$. 

Let $f \colon N_{\mathbf{W}} \to N_{\mathbf{W}'}$ be a morphism of nil-Jacobians. Suppose first that both $N_{\mathbf{W}}$ and $N_{\mathbf{W}'}$ admit embeddings to fibres of purification maps of mixed Hodge varieties $j \colon N_{\mathbf{W}} \hookrightarrow N_x \subset M$ and $j' \colon N_{\mathbf{W'}} \hookrightarrow N_{x'} \subset M'$. Observe that if $M''=M \times M'$ is a product of mixed Hodge varieties, then $M''_{\sigma} = M_{\sigma} \times M'_{\sigma}$ and  purification map $\sigma''$  is the product of purification maps: $\sigma''=\sigma \times \sigma'$. The nil-Jacobian $N_{\mathbf{W}''}:=N_{\mathbf{W}} \times N_{\mathbf{W}'}$ embeds to $N_x \times N_{x'}=(\sigma'')^{-1}(x,x')$. We claim that the induced definable space structure on $N_{\mathbf{W}''}$ is the same as the definable space structure of the product. This follows from the fact that if $r \colon \mathcal{D} \to \mathcal{D}_{\R}$ and $r' \colon \mathcal{D}' \to \mathcal{D}'_{\R}$ are $\operatorname{sl}_2$-splittings of two connected mixed Hodge domains, then $r \times r'$ is the $\operatorname{sl}_2$-splitting of $\mathcal{D} \times \mathcal{D}'$ by the uniqueness of $\operatorname{sl}_2$-splitting (\cite[Propostion 2.20]{CKS}).

The graph of $f \colon N_{\mathbf{W}} \to N_{\mathbf{W}'}$ is a sub-nil-Jacobian $N_f \subset N_{\mathbf{W}} \times N_{\mathbf{W}'}$ (Proposition \ref{graphsofmaps_prop}) and $j \times j'$ embeds it as a sub-nil-Jacobian of $N_{x} \times N_{x'}\subset M''$. We get a chain of definable subsets
\[
N_f \subset N_{\mathbf{W}''} \subset N_x \times N_{x'} \subset M''.
\]
Therefore, $N_f$ is definable in $N_{\mathbf{W}} \times N_{\mathbf{W}'}$.

Suppose now $N_{\mathbf{W}}$ and $N_{\mathbf{W}'}$ are arbitrary. Choose finite covers $\tau \colon \widehat{N}_{\mathbf{W}} \to N_{\mathbf{W}}$ and $\tau' \colon \widehat{N}_{\mathbf{W}'} \to N_{\mathbf{W'}}$ such that $\widehat{N}_{\mathbf{W}}$ and $\widehat{N}_{\mathbf{W}'}$ embed to fibres of purification maps of some mixed Hodge varieties. The map $f \colon N_{\mathbf{W}} \to N_{\mathbf{W}'}$ lifts to a correspondence between $\widehat{N}_{\mathbf{W}}$ and $\widehat{N}_{\mathbf{W}'}$ whose graph $\widehat{N}_f$ is a sub-nil-Jacobian of $\widehat{N}_{\mathbf{W}} \times \widehat{N}_{\mathbf{W}'}$. By the same argument as above, it is definable, and the graph $N_f=\tau \times \tau'(\widehat{N}_f)$  of $f$ is definable in $N_{\mathbf{W}}\times N_{\mathbf{W}'}$ (cf. Proposition \ref{finite covers}, \textit{(iii)}). This proves \textit{(i)}.

Notice that if $j_1 \colon N_{\mathbf{W}} \hookrightarrow N_{x_1} \subset M_1$ and $j_2 \colon N_{\mathbf{W}} \hookrightarrow N_{x_2} \subset M_2$ are two different embeddings, they induce the same definable manifold structure by the definability of the identity map $\operatorname{id} \colon N_{\mathbf{W}} \to N_{\mathbf{W}}$. Combining this with Proposition \ref{finite covers}, we deduce that the constructed definable manifold structure does not depend on any choices. 
\\

Item \textit{(ii)}  follows from \textit{(i)} and the fact that  the central tower of a nil-Jacobian is a diagram in the category of nil-Jacobians.
\\

Item \textit{(iii)} is immediate from the construction.
\\

We are ready to proof \textit{(iv)}. It is sufficient to prove the statement for $j=s$ and then argue by induction on $s$. 

By Proposition \ref{lifting of action} it is enough to check the definability of the action map $C^s_{\mathbf{W}} \times N^s_{\mathbf{W}} \to N^s_{\mathbf{W}}$ without worrying about the definability of the group structure on $C^s_{\mathbf{W}}$.

The idea is to find a suitable decomposition of $C^s_{\mathbf{W}}$ into a product of two groups, a compact one (which we below denote by $T$) and a copy of $\R^k$ (which will be denoted by $L$). The action of the first factor is analytic, hence automatically definable in $\R_{\an}$. The definability of the action of the second factor follows from Proposition \ref{definable criterion} and Lemma \ref{linear algebra} above. 

For convenience, we break the proof into several steps.
\\

\textbf{Step 1.} Let $\mathbf{Z} \subseteq \mathbf{W}$ be the lowest term of the lower central series filtration and $\Gamma_Z:=\mathbf{Z} \cap \Gamma$, so that $C^s_{\mathbf{W}}=\Gamma_Z \backslash \mathbf{Z}(\C)/F^0\mathbf{Z}$. The group $\mathbf{Z}$ is the additive group of a $\Q$-Hodge structure $Z$ and $\Gamma_Z\subset Z$ is a $\Z$-structure on it. Applying Lemma \ref{linear algebra} to $V=Z$ and $F=F^0Z$, we construct a rational subspace $L \subset \mathbf{Z}$ such that $Z_{\C}/F^0Z=q(Z_{\R}) \oplus q(\i L_{\R})$, where $q \colon Z_{\C} \to Z_{\C}/F^0Z$ is the projection. Let $\mathbf{L}\subseteq \mathbf{Z} \subseteq \mathbf{W}$ be the corresponding $\Q$-subgroup.
\\

\textbf{Step 2.} The group $\mathbf{Z}/F^0\mathbf{Z}$ splits into the product of two real algebraic groups, $\mathbf{Z}/F^0\mathbf{Z}=\mathbf{Z}_{\R} \times \i\mathbf{L}_{\R}$ (we omit the projection $q$ from the notation since it is injective both on $\mathbf{Z}_{\R}$ and $\i \mathbf{L}_{\R}$). The image of $\Gamma_Z$ under the map $q$ is contained inside $\mathbf{Z}(\R)$, so $J^0Z$ can be written as $J^0Z=(\Gamma_Z \backslash\mathbf{Z}(\R)) \times \i\mathbf{L}_{\R}$.  Observe that $\i \mathbf{L}_{\R}$ acts freely by real semialgebraic transformations on $\mathcal{D}$ and on $\mathcal{D}_{\mathbf{W}}$ and freely on $M$ (and on $N_{\mathbf{W}}$).
\\

\textbf{Step 3.} Let us first check that the action of $\i\mathbf{L}_{\R}$ on $N_{\mathbf{W}}$ is definable. Let $h_0 \in \Xi_{\R}$ be a reference Hodge cocharacter which is split over $\R$. Then, as it follows from the proof of Lemma \ref{subniljac}, $\i\mathbf{L}_{\R} \cdot h_0$ is contained in a finite number of translates of the fundamental domain $\Xi=r^{-1}(\Xi_{\R})$. After replacing $\Gamma$ with a finite index subgroup, one may assume that each orbit of $\i\mathbf{L}_{\R}$ is contained in a single fundamental domain $\Xi$ (this does not affect the definability of the action by Proposition \ref{lifting of action}, \textit{(ii)}). The action of $\i\mathbf{L}_{\R}$ on $\Xi$ is real semialgebraic. Since the projection $\pi \colon \mathcal{D} \to M$ is a definable $\i\mathbf{L}_{\R}$-equivariant diffeomorphism, the action of $\i \mathbf{L}_{\R}$ on $\pi(\Xi)$ is definable. The latter is a definable dense open subset of $M$ and the graph of the action of $\i\mathbf{L}_{\R}$ on $M$ is the closure of the graph of its action on $\Xi$, hence also definable.
\\

\textbf{Step 4.} Denote $T:=\Gamma_Z \backslash \mathbf{Z}(\R)$. This is an analytic Lie group isomorphic to a compact torus, hence it is a $\R_{\an}$-definable Lie group which definably acts on $N_{\mathbf{W}}$. We also denote $L:=\i\mathbf{L}_{\R}$. This is the additive group of a real vector space. The orbit $C_{\mathbf{W}}^s \cdot x$ splits as 
\[
C^s_{\mathbf{W}} \cdot x= (T \times L)x= T \cdot x \times L \cdot x
\]
The action of $J$ on $M$, and hence on its definable $J$-invariant subset $N_{\mathbf{W}}$, is definable by \textbf{Step 3}. The projections $C_{\mathbf{W}}^s \cdot x \to L \cdot x$ and $C_{\mathbf{W}}^s \cdot x \to T \cdot x$ are definable, since they are quotients by  closed definable equivalence relations (\cite[Chapter 10]{VdD}). Therefore the definability of the actions of $T$ and $L$ implies the definability of action of $C_{\mathbf{W}}^s$.
\end{proof}

\section{o-minimal geometry of higher Albanese manifolds}\label{main sec}

\subsection{o-minimal geometry of higher Albanese manifolds}\label{main subsec}

Higher Albanese maps are closely related to  period maps of certain (\emph{canonical})  admissible unipotent variations of mixed $\Z$-Hodge structures, introduced by Hain and Zucker in \cite{HZ}. The following Theorem is essentially \cite[Corollary 5.20]{HZ}.

\begin{thrm}[Hain - Zucker]\label{albanese are periods}
Let $X$ be a smooth complex quasi-projective variety. For each $s$ there exists an admissible unipotent graded polarised variation of mixed $\Z$-Hodge structures $\V^s$ such that its period map $\Phi_s \colon X^{\an} \to M$ factorises as
\[
\xymatrix{
X^{\an} \ar[rr]^{\Phi_s} \ar[rd]_{\alb^s} & & M\\
& \Alb^s(X) \ar[ru]_{\Psi^{s}}
}
\]
where $\Psi^{s}$ is a cover on its image. Moreover, the image of the map $\Psi^{s}$ is contained in a  fibre of the purification map $M \to M_{\sigma}$ and it is a morphism of nil-Jacobians. In particular, it lifts to a homomorphism of unipotent $\Q$-algebraic groups $\widetilde{\Psi^{s}} \colon \G^s(X;x) \to \mathbf{U}$, where $\mathbf{U}$ is the unipotent radical of $\mathbf{G}$ and $M=\Gamma \backslash \mathcal{D}_{\mathbf{G}}$. The induced homomorphism of Lie algebras $\operatorname{Lie}(\widetilde{\Psi^{s}}) \colon \g^s(X;x) \to \operatorname{Lie}(\mathbf{U})$ is a moprhism of mixed Hodge structures.
\end{thrm}

Without going into details, let us mention that the variation $\V^s$ captures  the dependence of the mixed Hodge structure on $\g^{s+1}(X;x)$ on the base point $x \in X$ (cf. Remark \ref{dependence on a point}).

\begin{prop}\label{albanese are periods up to finite cover}
 The map $\Psi^s$ is a finite cover on its image.
\end{prop}
\begin{proof}
It is shown in \cite{HZ} that this map is a non-ramified cover on its image. Let $\mathbf{W}$ be the image of the map $\widetilde{\Psi^s} \colon \G^s(X;x) \to \mathbf{U}$. This is a closed  algebraic subgroup of $\mathbf{U}$ defined over $\Q$. In fact, \cite{HZ} observe that the homomorphism $\widetilde{\Psi^s}$ is injective.

It  suffices to check that the image of $(\Psi^s)_* \colon \pi_1(\Alb^s(X)) \to \pi_1(\im \Psi^s)$ is of finite index. In other words, we need to check that $\widetilde{\Psi^s}(\G^s_{\Z}) \subseteq \Gamma_W:=\Gamma \cap \mathbf{W}(\Q)$ is of finite index. Notice that both $\Gamma_W \subset \mathbf{W}(\Q)$ and $\G^s_{\Z} \subset \G^s_{\Q}(X;x)$ are Zariski dense. Since $\mathbf{W}=\im \widetilde{\Psi^s}$, the group $\widetilde{\Psi^s}(\G^s_{\Z}) \subset \mathbf{W}(\Q)$ is Zariski dense as well.

This is a general fact about lattices in unipotent groups: if $\Gamma_1 \subseteq \Gamma_2$ are two discrete Zariski dense subgroups of a  unipotent group over $\Q$, then $\Gamma_2/\Gamma_1$ is finite (this can be checked, for example, by induction in the length of lower central series).
\end{proof}

Now we are ready to proof Theorem \ref{definabilisation main} mentioned in the introduction.

Recall that the classical Albanese manifold $\Alb(X)$ of a normal projective variety $X$ is canonically a semiabelian variety, in particular it is quasi-projective (\cite[Lemma 3.8]{Fuj}).

\begin{thrm}\label{definabilisation}
Let $X$ be a complex normal quasi-projective variety. For every $s \ge 1$ the higher Albanese manifold $\Alb^s(X)$ can be endowed with a structure of an $\R_{\alg}$-definable complex manifold in such a way that
\begin{itemize}
\item[(i)] the projections $p_s \colon \Alb^s(X) \to \Alb^{s-1}(X)$ are definable;
\item[(ii)]for each $s$ there exists a definable commutative connected complex Lie group $C^s$ such that $p_s \colon \Alb^s(X) \to \Alb^{s-1}(X)$ is a definable holomorphic principal $C^s$-bundle, in particular, the action $C^s \times \Alb^s(X) \to \Alb^s(X)$ is definable. Each $C^s$ is abstractly isomorphic (as a complex Lie group) to the Jacobian of a mixed Hodge structure.
\item[(iii)] the higher Albanese maps  $\alb^s \colon X^{\odef} \to \Alb^s(X)$ are  $\R_{\an, \exp}$-definable;
\item[(iv)] if $s=1$, the resulting $\R_{\alg}$-definable structure on $\Alb^1(X)=\Alb(X)$ is the same as the one determined by the canonical algebraic structure on the $\Alb(X)$;
\item[(v)] if $f \colon X \to Y$ is a morphism of normal  varieties, $\Alb(f) \colon \Alb^s(X) \to \Alb^s(Y)$ is definable;
\end{itemize}
Moreover,
\begin{itemize}
\item[(vi)] the reduced image $\alb^s(X)^{\operatorname{red}}$ is the definable analytification of a quasi-projective variety and $\alb^s \colon X \to \alb^s(X)^{\operatorname{red}}$ is the analytification of an algebraic morphism.
\end{itemize}
\end{thrm}
\begin{proof}
Higher Albanese manifolds are nil-Jacobians, therefore they can be endowed with $\R_{\alg}$-definabe manifold structures such that \textit{(i)} and \textit{(ii)} hold (Theorem \ref{nil-jacobians o-min}).
\\

\textit{(iii)}. By Theorem \ref{period maps o-minimal} the period map of the $s$-th canonical variation is $\R_{\an, \exp}$-definable. The higher Albanese map is a lift of the period map along a finite cover $\Psi^s$ (Theorem \ref{albanese are periods}, Proposition \ref{albanese are periods up to finite cover}), hence it is also definable by Proposition \ref{finite covers}.
\\

\textit{(iv)}. Let $\mathfrak{X}$ be a normal algebraic variety and $\mathfrak{A} = \Alb(\mathfrak{X})$ its Albanese variety. Since $\mathfrak{A}$ is again a normal algebraic variety, we can  consider its Albanese variety $\Alb(\mathfrak{A})$. On one hand, the  map $\alb^1 \colon \mathfrak{A}^{\odef} \to \Alb^1(\mathfrak{A}^{\odef}) = \Alb^1(\mathfrak{X})$ is definable by \textit{(ii)}. On the other hand, it is a biholomorphism by the universal property of the Albanese map. Thus, $\Alb^1(X)$ is definably biholomorphic to the definabilisation of the algebraic variety $\mathfrak{A}$.
\\

Item \textit{(vi)} follows from Theorem \ref{griffiths}, Proposition \ref{albanese are periods up to finite cover} and the fact that a finite cover of a quasi-projective variety is quasi-projective (this is known as \emph{Riemann existence theorem}, see e.g. \cite[Th\'eor\`eme 5.1]{SGA1}).
\end{proof}

\subsection{Properties of higher Albanese maps}

As first applications of Theorem \ref{definabilisation} we deduce several geometric properties of higher Albanese maps.

\begin{prop}\label{Brunebarbe extension prop}
Let $X$ be a normal quasi-projective variety and $s \ge 0$ a natural number. There exists a normal quasi-projective vairety $X'$, and dense open embedding $\iota \colon X \hookrightarrow X'$ such that $\Alb^s(\iota) \colon \Alb^s(X) \to \Alb^s(X')$ is a biholomorphism and $\alb^s_{X'} \colon X' \to \Alb^s(X')$ is proper. 
\end{prop}
\begin{proof}
The homomorphism $\mu:=\mu^s_{\pi_1(X)} \colon \pi_1(X) \to \G^s_{\Q}(X)$ defines a $\Q$-local system on $X$. Its monodromy is torsion-free, therefore by \cite[Proposition 3.5]{Brun} there exists a maximal partial compactification $\iota \colon X \to X'$ to which this local system extends. This means that $\iota \colon X \to X'$ is an open embedding of algebraic varieties, and there exists a representation $\mu' \colon \pi_1(X') \to \G^s_{\Q}(X)$ that shares the following properties:
\begin{itemize}
\item[(a)] the composition $\pi_1(X) \xrightarrow{\iota_*} \pi_1(X') \xrightarrow{\mu'} \G^s_{\Q}(X)$ equals $\mu$;
\item[(b)] for any smooth projective compactification $j \colon X' \to \overline{X}$ with  snc boundary divisor $D=\overline{X} \setminus X$ and any holomorphic map from a disc $v \colon \Delta \to \overline{X}$ for which $v(\Delta) \cap D = \{v(0)\}$, the monodromy around the image of the generator of $\pi_1(\Delta \setminus \{0\})=\Z$ is of infinite order.
\end{itemize}
Since $\mu'$ is an $s$-step nilpotent representation of $\pi_1(X')$, it factorises through 
\[
\mu^s_{\pi_1(X')} \colon \pi_1(X') \to \G^s_{\Q}(X').
\]
We obtain a morphism of algebraic groups $\nu \colon \G^s_{\Q}(X') \to \G^s_{\Q}(X)$. A morphism in the reverse direction is induced by the embedding $\iota$.
\[
\xymatrix{
\pi_1(X) \ar[dd]_{\iota_*} \ar[rr]^{\mu^s_{\pi_1(X)}} && \G^s_{\Q}(X) \ar@/^1pc/[dd]^{\G^s_{\Q}(\iota_*)}\\
\\
\pi_1(X') \ar[rruu]_{\mu'} \ar[rr]_{\mu^s_{\pi_1(X')}} && \G^s_{\Q}(X') \ar@/^1pc/[uu]^{\nu}
}
\]

The maps $\nu$ and $\G^s_{\Q}(\iota_{*})$ are mutually inverse on Zariski dense subgroups $\G^s_{\Z}(X) \subset \G^s_{\Q}(X)$ and $\G^s_{\Z}(X') \subset \G^s_{\Q}(X')$, therefore they establish isomoprhisms of algebraic groups. In particular, $\mu'=\mu^s_{\pi_1(X')}$. Moreover, $\G^s_{\Q}(\iota_*)$ induces an isomorphism of mixed Hodge structure on the Lie algebras $\g^s(X) \to \g^s(X')$. Thus, the embedding $\iota$ induces biholomorphism of higher Albanese manifolds $\Alb(\iota) \colon \Alb^s(X) \to \Alb^{s}(X')$. 

 To show that $\alb_{X'}$ is proper, we need to check that it has closed image and compact fibers (in the analytic topology). 
 
 First, let us verify that the 
 fibers are compact. Choose a projective compactification $\overline{X'}=X' \cup D$ with a boundary simple normal crossing divisor $D$. Let $F:=(\alb^s_{X'})^{-1}(y)$ be a fiber. Theorem \ref{definabilisation} implies that $F$ is a definable analytic subset of $X'$. Hence it is algebraic by definable Chow  (Theorem \ref{definable chow}) and its closure in the analytic topology $\overline{F} \subseteq \overline{X}$ is a projective algebraic variety. Suppose that $\overline{F} \cap D$ is non-empty. Then we can choose a holomorphic map $v \colon \Delta \to \overline{F}$ with $v(0)\in \overline{F} \cap D$ and $v(\Delta\setminus \{0\})\subseteq F$. Since the image of this map is contained in a fiber of $\alb^s_{X'}$, the generator of $\pi_1(\Delta \setminus \{0\})$ is mapped to identity under the map $(\alb^s_{X'})_{*} \colon \pi_1(X') \to \Alb^s(X')$, therefore lies in the kernel of $\mu^s_{\pi_1(X')}$. This contradicts property \textit{(b)} of Brunebarbe's  maximal extension (see above). We conclude that $\overline{F}\cap D=\varnothing$, i.e. $\overline{F}=F$ and $F$ is compact in the analytic topology.
 
 The proof that the image $Y:=\alb^s(X')$ is closed goes in the similar spirit. Let $\overline{Y}$ be its closure inside $\Alb^s(X')$. By Theorems \ref{definabilisation} and \ref{definable Remmert Stein} this is a complex analytic subset of the same dimension, so $Z:=\overline{Y} \setminus Y$ is an analytic subvariety of $\overline{Y}$.
 
 Choose a point $z \in Z$ and a holomorphic map from a disk $v \colon \Delta \to \overline{Y}$ such that $v(0)=z$ and the image of $\Delta^{\times}=\Delta\setminus \{0\}$ is contained in $Y$. The map $X' \to Y$ extends to a proper map $\overline{X'} \to \overline{Y}$ for some partial compactification $\overline{X'}$ of $X'$ and we may choose a holomorphic map $\widetilde{v} \colon \Delta \to \overline{X'}$ that covers $v$ (perhaps with ramification at $0$). Under the composition
 \[
 \pi_1(\Delta \setminus{0}) \xrightarrow{(\widetilde{v}_*)} \pi_1(X') \xrightarrow{(\alb^s_{X'})_*} \pi_1(Y) \hookrightarrow \pi_1(\overline{Y})
 \]
 the generator is mapped to a power of a contractible loop, so it lies in the kernel which again leads to a contradiction.
\end{proof}

\begin{cor}\label{when tower is algebraic}
Let $X$ be a normal quasi-projective variety and $s$ a natural number. The following conditions are equivalent:
\begin{itemize}
\item[(i)] $\Alb^s(X)$ is definably biholomorphic to $(\mathfrak{A}^s)^{\odef}$ for some normal quasi-projective variety $\mathfrak{A}^s$;
\item[(ii)] there exists a normal quasi-projective variety  $X'$ and a morphism of algebraic varieties $f \colon X \to X'$ such that $\Alb^s(f) \colon \Alb^s(X) \to \Alb^s(X')$ is a biholomorphism and $\alb_{X'}^s \colon X' \to \Alb^s(X')$ is surjective;
\item[(iii)] there exists a normal quasi-projective variety  $X'$ and a morphism of algebraic varieties $f\colon X \to X'$, such that $\Alb^s(f) \colon \Alb^s(X) \to \Alb^s(X')$ is a biholomorphism and $\alb_{X'}^s \colon X' \to \Alb^s(X')$ is dominant;
\item[(iv)] $\Alb^s(X)$ is definably biholomorphic to $(\mathfrak{A}^s)^{\odef}$ for some normal quasi-projective variety $\mathfrak{A}^s$ and the truncated higher Albanese tower
\[
\Alb^s(X) \xrightarrow{p^s} \Alb^{s-1}(X) \to \ldots \xrightarrow{p^2} \Alb^1(X)
\]
admits an algebraisation. Moreover the actions $C^j \times \Alb^j(X) \to \Alb^j(X)$ admit algebraisation for each $j \le s$ (namely, $C^j=(\mathfrak{C}^j)^{\odef}$ for an algebraic group $\mathfrak{C}^j$ and the action is algebraic).
\end{itemize}
\end{cor}
\begin{proof}
First, \textit{(i)} $\implies$ \textit{(ii)}. To see that, set $X':= (\mathfrak{A}^s)^{\odef}$ and check that $\alb^s \colon X' \to \Alb^s(X')$ is a biholomorphism.

More precisely, suppose that there exists a definable biholomorphism $\phi \colon \Alb^s(X) \to (\mathfrak{A}^s)^{\odef}$ for some normal quasi-projective variety $\mathfrak{A}^s$. Denote $X':= \mathfrak{A}^s$.

By Proposition \ref{Brunebarbe extension prop} we may assume that $\alb^s_X$ is proper. The composition $ (X)^{\odef} \xrightarrow{\alb^s_X} \Alb^s(X)\xrightarrow{\phi}(X')^{\odef}$ is definable (Therorem \ref{definabilisation}), hence it is a definabilisation of an algebraic morphism of quasi-projective varieties $\alpha \colon X \to X'$. It induces a map between higher Albanese manifolds $\Alb^s(\alpha) \colon \Alb^s(X) \to \Alb^s(X')$ and we obtain a diagram
\[
\xymatrix{X \ar[r]^{\alpha} \ar[d]_{\alb^s_X} & X' \ar[d]^{\alb^s_{X'}} \\
\Alb^s(X) \ar[ur]^{\phi} \ar[r]_{\Alb^s(\alpha)} & \Alb^s(X')}
\]
It is enough to check that $\Alb^s(\alpha)$ is a biholomorphism to ensure that $\alb^s_{X'}=\phi^{-1} \circ \Alb^s(\alpha)$ is.

The map $\alpha$ induces a morphism of algebraic groups $\G^s(\alpha) \colon \G^s(X) \to \G^s(X')$ that is a morphism of Hodge structures on Lie algebras $\g^s(\alpha) \colon \g^s(X) \to \g^s(X')$ and sends $\G^s_{\Z}(X)$ to $\G^s_{\Z}(X')$.  Since $\pi_1(X')=\pi_1(\Alb^s(X))$ is torsion free nilpotent of nilpotency class $s$, the map $\mu^s_{X'} \colon \pi_1(X') \to \G^s_{\Z}(X')$ is an isomorphism and $\G^s_{\Z}(X) \to \G^s_{\Z}(X')$ is also an isomorphism. 

This shows that $\G^s(\alpha) \colon \G^s(X) \to \G^s(X')$ is a morphism between connected simply connected nilpotent Lie groups over $\Q$ that sends a Zariski dense lattice isomorphically to a Zariski dense lattice. Such a morphism is necessary and isomorphism. It moreover sends $F^0\G^s(X)$ to $F^0\G^s(X')$ and descends to a biholomorphism $\Alb^s(\alpha)$.
\\

Clearly, \textit{(ii)} $\implies$ \textit{(iii)}. The implication $\textit{(iii)} \implies \textit{(ii)}$ is a special case of Proposition \ref{Brunebarbe extension prop}. If $\alb^s_X \colon X \to \Alb^s(X)$ is dominant, take $X \hookrightarrow X'$ as in Proposition \ref{Brunebarbe extension prop} and observe the image of $\alb^s_{X'} \colon X' \to \Alb^s(X')=\Alb^s(X)$ is closed and contains $\alb^s_X(X)$, hence coincides with the whole $\Alb^s(X')$. 
\\

 Notice that \textit{(ii)} $\implies$ \textit{(iv)}. Since $\alb^{j-1}=\alb^j \circ p^j$ and $p^j$ are surjective, the maps $\alb^j \colon X \to \Alb^j(X)$ are surjective for every $j \le s$.  By \textit{(vi)} of Theorem \ref{definabilisation} each $\Alb^j(X)$ is the definable analytification of a quasi-projective variety $\mathfrak{A}^j$. The morphisms $p^j$  and the actions of $C_{\mathbf{W}}^j$ are algebraisable by Corollary \ref{algebraisation basic}.
 \\

 Finally, \textit{(iv)} $\implies$ \textit{(i)}.
\end{proof}

The last section of this paper (Section \ref{algebraic sec}) deals with the situation when one of the equivalent coniditions of the Corollary \ref{when tower is algebraic} is satisfied. It turns out, that this makes the situation very restrictive. Essentially, we show that this might happen either if $s \le 2$ or if the higher Albanese tower stabilises, i.e. $p^j \colon \Alb^j(X) \to \Alb^{j-1}(X)$ are isomorphisms for $j \ge 3$.

\subsection{Application: partial higher Albanese manifolds.}\label{shafarevich subsection}

Let $X$ be a normal quasi-projective variety and $\theta \colon \pi_1(X) \to \C$ and additive $\C$-valued character. Then it can be written uniquely as $\alb^*\theta_0$, where $\theta_0 \in H^1(\Alb(X), \C)$. The \emph{$\theta$-Albanese manifold} $\Alb_{\theta}(X)$ is defined as the quotient of $\Alb(X)$ by the maximal connected algebraic subgroup $B_{\theta} \subseteq \Alb(X)$ for which $\theta_0|_{B_{\theta}} = 0 \in H^1(B_{\theta}, \C)$. 

In this subsection, we construct higher analogues of the same construction. Let $X$ be a normal quasi-projective variety, $\mathbf{U}$ a connected simply connected unipotent group over a field $\Bbbk$ of characteristic zero and $\rho \colon \pi_1(X) \to \mathbf{U}(\Bbbk)$ a representation.

\begin{thrm}\label{partial higher albanese exists}
There exists a nil-Jacobian $\Alb^s_{\rho}(X)$, a representation $\rho_0 \colon \pi_1(\Alb^s_{\rho}(X)) \to \mathbf{U}(\Bbbk)$ and a definable map $\alb^s_{\rho} \colon X \to \Alb^s_{\rho}(X)$, such that
\begin{itemize}
\item[(i)] $\rho=\rho_0 \circ (\alb^s_{\rho})_*$;
\item[(ii)] for any sub-nil-Jacobian $N \subseteq \Alb^s_{\rho}$ of positive dimension the restriction $\rho_0|_{\pi_1(N)}$ is non-trivial;
\item[(iii)] If  $X \xrightarrow{\phi} S \xrightarrow{\psi} \Alb^s_{\rho}(X)$ is the Stein factorisation of $\alb^s_{\rho}$, then $(S, \phi)$ is the Shafarevich reduction of $\rho$.
\end{itemize}
\end{thrm}

Recall that if $X$ is a normal quasi-projective variety, $G$ a group and $\rho \colon \pi_1(X) \to G$ the \emph{Shafarevich reduction} of $\rho$ is a pair $(\Sh_{\rho}, \sh_{\rho})$, where $\Sh_{\rho}$ is a normal quasi-projective variety and $\sh_{\rho} \colon X \to \Sh_{\rho}$ is a dominant morphism with a connected general fibre satisfying the following universal property: for any normal connected algebraic variety $Y$ and a morphism $f \colon Y \to X$ the composition $Y \xrightarrow{f} X \xrightarrow{\sh_{\rho}} \Sh_{\rho}$ is constant if and only if $\pi_1(Y) \xrightarrow{f_*} \pi_1(X) \xrightarrow{\rho} G$ has finite image.

Shafarevich reductions play crucial role in the modern approach to Shafarevich Conjecture on holomorphic convexity of universal covers of algebraic varieties (see \cite{Eys}, \cite{EKPR}, \cite{BBT24}, \cite{DY+K}), see also \cite{BM} and \cite{CDY} for applications of Shafarevich reductions in over topics. 

The existence and essential uniqueness of Shafarevich reductions in the case where $G$ is an algebraic group over a field $\Bbbk$ of characteristic zero was proven by Bakker, Brunebarbe and Tsimerman in \cite{BBT24} (the work \cite{BBT24} is based on earlier  results of \cite{Eys}, \cite{Brun}, \cite{DY+K}; see also \cite{DY} for a similar statement when $\operatorname{char} \Bbbk>0$.). The construction in \cite{BBT24} is not very explicit, as it uses the $\C^{\times}$-action on the moduli stack of Higgs bundles and abstract existence theorems for complex variations of Hodge structures. Therefore, we find Theorem \ref{partial higher albanese exists} useful, as its item \textit{(iii)} gives explicit description of the Shafarevich reduction in the nilpotent case.

\begin{proof}[Proof of Theorem \ref{partial higher albanese exists}]

 Without loss of generality we may assume that $\mathbf{U}$ is defined over a subfield $\Bbbk' \subseteq \Bbbk$ of at most countable transcendence degree. Choosing an embedding $\Bbbk' \hookrightarrow \C$ we reduce everything to the case $\Bbbk=\C$.

 Let $s$ be the nilpotency of $\mathbf{U}$. By the universal property of Malcev completions (Theorem \ref{malcev exists}) there exists a factorisation
\[
\xymatrix{
\pi_1(X) \ar[rd]_{\mu^s} \ar[rr]^{\rho} && \mathbf{U}(\C)\\
&\G^s_{\C}(X) \ar[ur]_{\nu}&
}
\]
The homomorphis $\nu$ induces a morphism of Lie aglebras $\operatorname{Lie}(\nu) \colon \g^s_{\C}(X) \to \mathfrak{u}$, where $\mathfrak{u}$ is the Lie algebra of $\mathbf{U}$.

Let $\mathfrak{p} \subset \g^s_{\Q}(X)$ be the maximal Lie subalgebra over $\Q$ such that the following holds:
\begin{itemize}
\item $\mathfrak{p}_{\C} \subseteq \ker \operatorname{Lie}(\nu)$;
\item $\mathfrak{p}$ is a Hodge substructure.
\end{itemize}
We claim that $\mathfrak{p}$ is a Lie ideal. Indeed, a span of a collection of Hodge substructures is a Hodge substructure, therefore
\[
\mathfrak{p}_1:=\operatorname{Span} \left ( \bigcup_{x \in \g^s_{\Q}(X)} [x, \mathfrak{p}] \right)
\]
again satisfies the two properties above. By maximality, $\mathfrak{p}_1=\mathfrak{p}$.

Let $\mathbf{P}:=\exp \mathfrak{p}$ be the corresponding subgroup of $\G^s_{\Q}(X)$. This is a closed normal subgroup and $\Gamma_P:=\G^s_{\Z}(X) \cap \mathbf{P}$ is a normal subgroup of $\G^s_{\Z}(X)$ that is discrete and Zariski dense in $\mathbf{P}$.

Let $\mathbf{Q}:=\mathbf{P} \backslash \G^s_{\Q}(X)$. Then $\Gamma_Q:=\Gamma_P \backslash \G^s_{\Z}(X)$ is a discrete Zariski dense subgroup in $\mathbf{Q}$. Moreover, the Lie aglebra $\mathfrak{q}$ of $\mathbf{Q}$ inherits a mixed Hodge structure with negative weights and we can consider the nil-Jacobian
\[
\Alb^s_{\rho}(X):= \Gamma_Q \backslash \mathbf{Q}(\C)/F^0\mathbf{Q}.
\]
Clearly, the projection $\G^s_{\Q}(X) \to \mathbf{Q}$ descends to a surjective morphism of nil-Jacobians $\alpha \colon \Alb^s(X) \to \Alb^s_{\rho}(X)$. The composition with $\alb^s$ produces a definable map $\alb^s_{\rho} \colon X \to \Alb^s_{\rho}(X)$.

Since $\Gamma_P=\ker[\pi_1(\Alb^s(X)) \to \pi_1(\Alb^s_{\rho}(X))]$ is contained in the kernel of $\nu \colon \G^s_{\C}(X) \to \mathbf{U}(\C)$, the representation $\rho$ factorises as 
\[
\pi_1(X) \xrightarrow{(\alb^s_{\rho})_*} \Gamma_Q=\pi_1(\Alb^s_{\rho}(X)) \xrightarrow{\rho_0} \mathbf{U}(\C)
\]
for some representation $\rho_0$. For every nil-Jacobian $N \subseteq \Alb^s_{\rho}(X)$ either $N$ is a point, or $\rho_0|_{\pi_1(N)}$ is non-trivial. Indeed, if it is trivial, the preimage of $N$ in $\Alb^s(X)$ is contained in a fibre of $\alpha$ which is $\Gamma_P \backslash \mathbf{P}(\C)/F^0\mathbf{P}$.

Let us check that the Stein factorisation $X \xrightarrow{\phi} S \xrightarrow{\psi} \Alb^s_{\rho}(X)$ is indeed the Shafarevich reduction.

Let $f \colon Y \to X$ be a morphism from a normal connected algebraic variety $Y$ such that $f^*\rho \colon \pi_1(Y) \to \mathbf{U}(\C)$ has finite image. Consider the induced map $\alb^s(f) \colon \Alb^s(Y) \to \Alb^s(X)$. We obtain a diagram
\[
\xymatrix{
Y \ar[rr]^{\alb^s_Y} \ar[d]_{f} & &\Alb^s(Y) \ar[d]^{\alb^s(f)}\\
X \ar[rr] ^{\alb^s_X} \ar[d]_{\phi} \ar[rrd]_{\alb^s_{\rho}} && \Alb^s(X) \ar[d]^{\alpha}\\
S \ar[rr]_{\psi} && \Alb^s_{\rho}(X)
}
\]
Since the image of $\pi_1(Y)$ in $\mathbf{U}(\C)$ is finite, it is trivial (a connected unipotent group over $\C$ contains no torsion elements). Therefore the image of $f_* \colon \G^s_{\Q}(Y) \to \G^s_{\Q}(X)$ is a closed connected $\Q$-subgroup of $\G^s_{\Q}(X)$ whose Lie algebra is a Hodge substructure and whose complexification is contained in $\ker\nu$. Therefore, $f_*(\G^s_{\Q}(Y)) \subseteq \mathbf{Q}$ and $\im \alb^s(f)$ lies in a fibre of $\alpha$. It follows from the diagram above that $\alb^s_{\rho} \circ f \colon Y \to \Alb^s_{\rho}(X)$ is constant. Since $Y$ is connected, $f(Y)$ is contained in a fibre of $\phi$.

Vice versa, if $\phi \circ f \colon Y \to S$ is constant, then the composition $\pi_1(Y) \xrightarrow{f_*} \pi_1(X) \xrightarrow{(\alb^s_{\rho})_*} \pi_1(\Alb^s_{\rho}(X))$ is trivial. Since $\rho$ factorises through $\pi_1(\Alb^s_{\rho}(X))$, we deduce that $\rho \circ f_*$ is trivial as well.

\end{proof}

\section{Algebraic geometry of higher Albanese manifolds}\label{algebraic sec}

In this section, we prove the second Main Theorem mentioned in he introduction (Theorem \ref{no algebraic main}) and discuss some of its consequences.

\subsection{Commutative algebraic groups}\label{commutative groups subsec}
We recall a well-known structure result on commutative complex algebraic groups, which is a special case of a more general Chevalley - Barsotti - Rosenlicht structure theorem (see e.g. \cite{Con}).
\begin{thrm}\label{commutative groups}
Let $\mathfrak{C}$ be a connected commutative algebraic group over $\C$. Then there exists a short exact sequence of complex algebraic groups 
\begin{equation}\label{algebraic presentation}
0 \to \mathfrak{T} \times \mathfrak{U} \to \mathfrak{C} \to \mathfrak{A} \to 0,
\end{equation}
where $\mathfrak{A}$ is an abelian variety, $\mathfrak{T}=\mathbb{G}_m^r$ an algebraic torus and $\mathfrak{U}=\mathbb{G}_a^k$.

Moreover, such a decomposition is essentially unique.
\end{thrm}

We refer to $\mathfrak{A}$ as the \emph{maximal compact quotient} of $\mathfrak{C}$ and to $\mathfrak{T}$ and $\mathfrak{U}$ as \emph{multiplicative} and \emph{unipotent parts} of $\mathfrak{C}$ respectively.

Let $\mathfrak{C}$ be a connected commutative algebraic group and $C=\mathfrak{C}^{\an}$. Then $C$ is a connected commutative complex Lie group that can be uniquely written as $C=V/\Lambda$, where $V$ is a finite-dimensional complex vector space and $\Lambda=\pi_1(C) \subset V$ is a finitely generated discrete subgroup.  In particular, $C$ is a $K(\Lambda, 1)$-space for $\Lambda \simeq \Z^r$.

\begin{prop}\label{no unipotent}
Let $\mathfrak{C}$ be a connected commutative algebraic group over $\C$ and $C=\mathfrak{C}^{\an}$. Suppose that
\begin{itemize}
\item[(i)] the maximal compact quotient of $\mathfrak{C}$ is trivial;
\item[(ii)] $C \simeq J^0H$ for some mixed Hodge structure $H$ with $W_{-1}H=H$.
\end{itemize}
Then $\mathfrak{C}$ is an algebraic torus.
\end{prop}
\begin{proof}
It is sufficient to prove that the unipotent part of $\mathfrak{C}$ is trivial. The decomposition $\mathfrak{C}=\mathfrak{T} \times \mathfrak{U}$ yields a decomposition of complex Lie groups $C=T \times U$. Let $V=H_{\C}/F^0H$ be the universal cover of $C$. This is a complex vector space which splits $V=\widetilde{T} \oplus \widetilde{U}$, where $\widetilde{T}$ is the universal cover of $T$ and $\widetilde{U}$ is the universal cover of $U$. The projection $\widetilde{U} \to U$ is an isomorphism and $\Lambda=\pi_1(C)$ is contained inside $\widetilde{T}$ (as a subgroup of $V$). At the same time, $\Lambda$ is the image of $H_{\Z}$ under the projection $H_{\C} \to H_{\C}/F^0H=V$. The lattice $H_{\Z} \subseteq H_{\C}$ is complex Zariski dense. Therefore $\Lambda$ is Zariski dense in $V$, which leads to a contradiction.
\end{proof}

\subsection{Topology of commutative principal bundles}\label{chern hoefer subsec}

We fix a connected  commutative complex algebraic group $\mathfrak{C}$ and denote $C=\mathfrak{C}^{\an}$.

Let $S$ be a complex manifold. Denote by $\O(S, C)$ the sheaf of holomorphic $C$-valued functions on $S$. This is a sheaf of groups  on $S$ and the holomorphic principal $C$-bundles over $S$ are classified by its \v{C}ech cohomology group $H^1(S, \O(S, C))$.  In the case $\mathfrak{C}=\mathbb{G}_m$, this group is nothing but $H^1(S, \O^{\times}_S)=\Pic(S)$. If $\mathfrak{C}=\mathbb{G}_m^{r}$, there are canonical isomorphisms $H^1(S, \O(S, C)) \simeq H^1(S, \O^{\times}_S)^{r} \simeq \Pic^{\times r}(S)$.

Write $C = V/\Lambda$ for a vector space $V$ and a discrete group $\Lambda \subset V$. By $\Lambda_S$ we denote the constant local system of abelian groups $\Lambda \o \underline{\Z}_{S}$. 

There is an analogue of the exponential short exact sequence:
\begin{equation}\label{short exponential}
0 \to \Lambda_S \to \O_S \o V \to \O(S, C) \to 0
\end{equation}
that induces:
\begin{equation}\label{long exponential}
\ldots \to H^1(S, \Lambda) \to H^1(S, \O_S\o V) \to H^1(S, \O(S, C)) \xrightarrow{\mathbf{c}} H^2(S, \Lambda) \to \ldots
\end{equation}
We refer to the map $\mathbf{c} \colon H^1(S, \O(S, C)) \to H^2(S, \Lambda)$ as the \emph{Chern-H\"ofer class}. It coincides with the first Chern class in the case where $C=\C^{\times}$ and was studied by H\"ofer  in the case of  $C$ compact (\cite{Hoef}).

\begin{prop}\label{topologically trivial}
Let $p \colon X \to S$ be a holomorphic principal $C$-bundle. Assume that $\mathbf{c}(p)=0$. Then $p$ is smoothly trivial, that is, $X$ is diffeomorphic to $S \times C$.
\end{prop}
\begin{proof}
The short exact sequence (\ref{short exponential}) has a $\mathcal{C}^{\infty}$-version:
\begin{equation}\label{short exponential smooth}
0 \to  \Lambda_S \to \mathcal{C}^{\infty}(S, \C) \o V \to \mathcal{C}^{\infty}(S, C) \to 0,
\end{equation}
where $\mathcal{C}^{\infty}(S, C)$ is the sheaf of smooth $C$-valued functions on $S$. The isomorphism classes of smooth principal $C$-bundles over $S$ are parametrised by $H^1(S, \mathcal{C}^{\infty}(S, C))$. The sheaf $\mathcal{C}^{\infty}(S, \C) \o V$ is acyclic and from the $\mathcal{C}^{\infty}$-version of the long exact sequence (\ref{long exponential}) one gets an isomorphism $\mathbf{c} \colon H^1(S, \mathcal{C}^{\infty}(S, C)) \xrightarrow{\sim} H^2(S, \Lambda)$. This isomorphism is, of course, the same Chern-H\"ofer class, up to the forgetful map $H^1(S, \mathcal{O}(S, C)) \to H^1(S, \mathcal{C}^{\infty}(S, C))$. Thus, $\mathbf{c}(p)=0$ if and only if it is trivial as a $\mathcal{C}^{\infty}$-principal bundle.
\end{proof}

\begin{rmk}\label{characteristic classes}
Another equivalent definition of the ($\mathcal{C}^{\infty}$-) Chern-H\"ofer class is the following. The universal cover of $C$ is contractible, thus $C$ has the homotopy type of the classifying space $\operatorname{B}\Lambda$, where $\Lambda$ is viewed as a commutative topological group with discrete topology. Principal $C$-bundles over a manifold $S$ are thus parametrised by the homotopy classes of maps $S \to \operatorname{B}C=\operatorname{B}(\operatorname{B}\Lambda)$. There are functorial isomorphisms $[S, \operatorname{B}(\operatorname{B}\Lambda)] \xrightarrow{\sim} [S, K(\Lambda, 2)] \xrightarrow{\sim}H^2(S, \Lambda)$.
\end{rmk}

The following lemma is classical in algebraic topology.

\begin{lemma}\label{topology}
Let $p \colon X \to S$ be a principal $C$-bundle. Assume that $S$ is aspherical. Then
\begin{itemize}
\item[(i)] There is an exact sequence
\begin{equation}\label{homotopy short}
1 \to \pi_1(C) = \Lambda \to \pi_1(X) \to \pi_1(S) \to 1;
\end{equation}
\item[(ii)] (\ref{homotopy short}) is a central extension;
\item[(iii)] the corresponding class  of the central extension 
\[
[\pi_1(X)] \in \operatorname{Ext}^1(\pi_1(S), \Lambda)=H^2(\pi_1(S), \Lambda)
\]
is mapped to the Chern-H\"ofer class $\mathbf{c}(p)$ under the  isomorphism  $H^2(\pi_1(S), \Lambda) \xrightarrow{\sim} H^2(S, \Lambda)$.
\end{itemize}
\end{lemma}
\begin{proof}[Proof]
\textit{(i)} is the long homotopy sequence of the fibration $C \to X \to S$.
\\

\textit{(ii)}. It is sufficient to show that the action of $\pi_1(S)$ on the fundamental group of the fibre is trivial. Since $\pi_1(C)$ is abelian, this would follow from the triviality of the local system $R^1p_*\Z_X$ on $S$.

Let $\operatorname{B}C$ be the classifying space of the group $C$ and $P \colon \operatorname{E}C \to \operatorname{B}C$ the universal principal $C$-bundle.  Since $C$ is connected, $\operatorname{B}C$ is simply connected and $R^1P_{*}\Z_{\operatorname{E}C}$ is trivial. At the same time, $R^1p_{*}\Z_X=\phi^*R^1P_{*}\Z_{\operatorname{E}C}$ for the classifying map $\phi \colon S \to \operatorname{B}C$.
\\

\textit{(iii)}. Since both $S$ and $C$ are aspherical, $X$ is also aspherical. Every central extension sequence
\[
1 \to \Lambda \to \pi_1(X) \to \pi_1(S) \to 1
\]
yields a homotopy fibration
\[
\operatorname{B}\Lambda \to \operatorname{B}\pi_1(X) \to \operatorname{B}\pi_1(S),
\]
and thus, a classifying map $\operatorname{B}\pi_1(S) =S \to \operatorname{B}(\operatorname{B}\Lambda)=\operatorname{B}C$.

This is precisely the homotopy definition of the Chern-H\"ofer class (see Remark \ref{characteristic classes} above).
\end{proof}

\subsection{Blanchard’s theorem}\label{blanchard subsec}
We recall the theorem of Blanchard on holomorphic prinicipal torus bundles with K\"ahler  total space (see Theorem \ref{blanchard} below).

Let $S$ be a complex manifold and $A=V/\Lambda$ a compact complex torus. As before, isomorphism classes of holomorphic principal $A$-bundles over $S$ correspond to the elements of  $H^1(S, \O(S, A))$ and are topologically classified by the Chern-H\"ofer class $\mathbf{c} \colon H^1(S, \O(S, A)) \to H^2(S, \Lambda)$.

The following result is essentially due to Blanchard (\cite{Blanch})
\begin{thrm}[Blanchard]\label{blanchard}
Let $S$ be a complex manifold and $p \colon X \to S$ is a holomorphic principal $A$-bundle. Suppose that $X$ is K\"ahler and  $H_1(S, \Z)$ is torsion-free. Then $\mathbf{c}(p)=0$. In particular, $\pi_1(X)=\pi_1(S) \times \pi_1(A)$.
\end{thrm}
\begin{proof}[Sketch of a proof]
Consider the pullback of a principal bundle to its total space
\[
p' \colon X'=X \times_S X \to X.
\]
This is again a holomorphic principal $A$-bundle. It admits a holomorphic section (namely, the diagonal $X \to X \times_S X$), hence trivial. The class $[p']$ is the image of $[p]$ under the natural map $H^1(S, \O(S,C)) \to H^1(X, p^*\O(S,C)) \to H^1(X, \O(X, C))$, so
\[
p^*\mathbf{c}(p)= \mathbf{c}(p')=0.
\]
Therefore, $\mathbf{c}(p)$ lies in the kernel of the map $H^2(S, \Lambda) \to H^2(X, \Lambda)$.

The Leray spectral sequence
\[
H^n(S, R^kp_*\Q_X) \implies H^{n+k}(X,\Q)
\]
degenerates on the second step by the Deligne-Blanchard Degeneration Theorem (\cite{Del71}). This implies that the map $p^* \colon H^2(S, \Q) \to H^2(X, \Q)$ is injective. Therefore, $H^2(S, \Lambda) \o \Q \to H^2(X, \Lambda) \o \Q$ is also injective and $\mathbf{c}(p)$ is a torsion class.

At the same time,
\[
\operatorname{Tors}(H^2(S, \Lambda))=\operatorname{Tors}(H^2(S, \Z)) \o \Lambda = \operatorname{Tors}(H_1(S, \Z)) \o \Lambda=0.
\]
We conclude that $\mathbf{c}(p)=0$. By Proposition \ref{topologically trivial}, $X$ is diffeomorphic to $S \times A$.
\end{proof}

\subsection{Toric bundles}\label{toric subsec}

First, we prove the following algebraisation result.

\begin{prop}\label{algebraisation of toric actions}
Let $\mathfrak{X}$ be a complex algebraic variety and $X=\mathfrak{X}^{\an}$. Let $p \colon Y \to X$ be a holomorphic principal $T$-bundle, where $T=(\C^{\times})^k$. Suppose that this bundle is algebraic in the following sense: $Y=\mathfrak{Y}^{\an}$ for some algebraic variety $\mathfrak{Y}$, the map $p$ is the analytification of an algebraic morphism $\mathfrak{p}$ and the action $T \times Y \to Y$ is the analytification of an algebraic action $\mathfrak{T} \times \mathfrak{Y} \to \mathfrak{Y}$, where $\mathfrak{T}=\mathbb{G}_m^{k}$. Then $\mathfrak{p} \colon \mathfrak{Y} \to \mathfrak{X}$ is a Zariski locally trivial principal $\mathfrak{T}$-bundle.
\end{prop}
\begin{proof}
The map $p \colon Y \to X$ is a holomorphic principal $T$-bundle which is locally trivial in the analytic topology. First, we claim that $p$ is locally trivial in the \'etale topology. Indeed, let $x \in X$. Taking a generic iterated hyperplane section of $Y$ transverse to the fibre $p^{-1}(x)$, we obtain a rational multisection of $p$, i.e. a subvariety $Z \subset Y$, such that $p|_Z \colon Z \to X$ is dominant and \'etale on a dense open subset $Z^{\circ} \subset Z$ with $x \in p(Z^{\circ})$. Thus, $p|_{Z^{\circ}} \colon Z^{\circ} \to X$ is an \'etale neighbourhood of $x$ and the restriction $Y \times_{X} Z^{\circ} \to Z^{\circ}$ is a trivial $T$-bundle.

Recall that an algebraic group $\mathfrak{G}$ is said to be \emph{special in the sense of Serre} if every \'etale locally trivial $\mathfrak{G}$-torsor is Zariski locally trivial. The group $\mathbb{G}_m$ is special  (\cite{Serre}) and a product of special groups is special. Hence the claim.
\end{proof}

\begin{lemma}\label{toric tower topology}
Let $Y_1$ be the analytification of a smooth quasi-projective variety. Let $p_2 \colon Y_2 \to Y_1$ be the total space of an algebraic Zariski locally trivial principal $(\C^{\times})^{r_2}$-bundle. Let $p_3 \colon Y_3 \to Y_2$ be the total space of an algebraic Zariski locally trivial principal $(\C^{\times})^{r_3}$-bundle.
\[
\xymatrix{
Y_3 \ar[d]_{p_3}^{\acts (\C^{\times})^{r_3}}\\
Y_2 \ar[d]_{p_2}^{\acts (\C^{\times})^{r_2}}\\
Y_1
}
\]
Suppose also that $Y_1$ is aspherical and $\pi_1(Y_1)$ is abelian and torsion-free. Then $\pi_1(Y_3)$ is nilpotent and $\operatorname{nilp}(\pi_1(Y_3)) \le 2$.
\end{lemma}

\begin{proof}[Proof of Lemma \ref{toric tower topology}]

By the homotopy exact sequence of a fibration, $Y_2$ and $Y_3$ are also aspherical. From Lemma \ref{topology} and Proposition \ref{nilpotency of extensions} the group  $\pi_1(Y_2)$ is nilpotent of nilpotency class at most $2$ and $\pi_1(Y_3)$ is nilpotent of nilpotency class at most $3$. The only non-trivial assertion is that $\nilp(\pi_1(Y_3)) \le 2$. 

By Proposition \ref{nilpotency of extensions}, it is enough to show that the class of the central extension
\[
1 \to \pi_1((\C^{\times})^{r_3})=\Z^{r_3} \to \pi_1(Y_3) \to \pi_1(Y_2) \to 1
\]
lies in the image of a map $H^2(\Gamma, \Z^{r_3}) \to H^2(\pi_1(Y_2), \Z^{r_3})$ induced by some epimorphism $\pi_1(Y_2) \to \Gamma$ onto an abelian group $\Gamma$. We will show that this is precisely the case for $\Gamma = \pi_1(Y_1)$ and the group homomorphism induced by $p_2$.

The fibration $p_2$  gives a class 
\[
[p_2] \in H^1(Y_1, \O(Y_1,(\C^{\times})^{r_2})).
\]
Recall, that
\[
H^1(Y_1, \O(Y_1, (\C^{\times})^{r_2}))=H^1(Y_1, \O_{Y_1}^{\times})^{\times r_2}=\Pic^{\times r_2}(Y).
\]
Similarly, $[p_3] \in \Pic^{\times r_3}(Y_2)$. Their Chern-H\"offer classes are $\mathbf{c}(p_2) \in H^2(Y_1, \Z^{r_2})$ and $\mathbf{c}(p_3) \in H^2(Y_2, \Z^{r_3})$ respectively. By Lemma \ref{topology} it is enough to check that the class $\mathbf{c}(p_3)$ lies in the image of
\[
p_1^{*} \colon H^2(Y_1, \Z^{r_3}) \to H^2(Y_2, \Z^{r_3}).
\]

Let us prove a stronger statement, namely, that the principal bundle $[p_3] \in \Pic^{\times r_3}(Y_2)$ is a pull-back of a principal $(\C^{\times})^{r_3}$-bundle on $Y_1$.
For an algebraic variety $\mathfrak{Y}$  with $Y=\mathfrak{Y}^{\an}$ we denote by $\Pic_{\alg}(Y)$  the image of the analytification map 
\[
\Ppic(\mathfrak{Y}) \xrightarrow{(-)^{\an}} \Pic(Y),
\]
where $\Ppic(\mathfrak{Y})$ is the algebraic Picard group, i.e. the group of Zariski locally trivial principal $\mathbb{G}_m$-bundles on $\mathfrak{Y}$.

Denote also $\Pic^{\times r}_{\alg}(Y):=\operatorname{im}[\Ppic^{\times r}(\mathfrak{Y)} \xrightarrow{(-)^{\an}}\Pic^{\times r}(Y)]=(\Pic_{\alg}(Y))^{\times r}$. Proposition \ref{algebraisation of toric actions} implies that $[p_2]$ is contained in $\Pic_{\alg}^{\times r_2}(Y_1)$ and, similarly, $[p_3] \in \Pic_{\alg}^{\times r_3}(Y_2)$.

 By   \cite[Proposition 3.1]{FI}, the sequence
 \[
 \Pic_{\alg}(Y_1) \xrightarrow{p_2^*} \Pic_{\alg}(Y_2) \to \Pic_{\alg}((\C^{\times})^{r_2}) \to 0
 \]
 is exact. The same is true for
 \[
 \Pic^{\times r_3}_{\alg}(Y_1) \to \Pic^{\times r_3}_{\alg}(Y_2) \to \Pic_{\alg}^{\times r_3}((\C^{\times})^{r_2}) \to 0.
 \]
The group $\Pic_{\alg}^{r_3}((\C^{\times})^{r_2})$ is trivial, thus we conclude that $[p_3]=p_2^*[q]$ for some $[q] \in \Pic^{r_3}_{\alg}(Y_1)$. 
 
\end{proof}

\subsection{Proof of Theorem \ref{no algebraic main}}\label{algebraic proof subsec}

Now we are ready to prove Theorem \ref{no algebraic main}. First, we reduce the situation to the case $s=3$.

\begin{prop}\label{3 implies r}
Let $X$ be a normal quasi-projective variety. Suppose that the projection $p^s \colon \Alb^s(X) \to \Alb^{s-1}(X)$ is a biholomorphism. Then $p^r \colon \Alb^r(X) \to \Alb^{r-1}(X)$ is a biholomorphism for every $r \ge s$.
\end{prop}
\begin{proof}
We argue by induction on $r$. Let $r_0$ be the minimal integer such that $r_0 > s$ and $p^{r_0}$ is not a biholomorphism.  The higher Albanese tower looks like
\[
\ldots \to \Alb^{r_0}(X)  \to \Alb^{r_0-1}(X) \xrightarrow{\sim}\Alb^{r_0-2}(X) \xrightarrow{\sim} \ldots \xrightarrow{\sim} \Alb^{s}(X) \to \ldots
\]

Recall that $\pi_1(\Alb^r(X))=\G^r_{\Z}(X)$ and $(\alb^r)_* \colon \pi_1(X) \to \G^r_{\Z}(X)$ is the universal $r$-step nilpotent torsion-free quotient of $\pi_1(X)$.

In our case, $\G^{r_0-1}_{\Z}(X)=\ldots = \G^{s-1}_{\Z}(X) = \G^{s}_{\Z}(X)$ and $\G^{r_0}_{\Z}(X)$ is a central extension of $\G^{r_0-1}_{\Z}(X)$. By Proposition \ref{nilpotency of extensions}, $s \le \nilp(\G^{r_0}_{\Z}(X)) \le s+1$. 

Suppose $\nilp(\G^{r_0}_{\Z}(X))=s+1$ (the case $\nilp(\G^{r_0}_{\Z}(X))=s$ is analogous). The homomorphism $\alb^{r_0}_* \colon \pi_1(X) \to \G^{r_0}_{\Z}(X)$ factorises through $\G^{s+1}_{\Z}(X)$, which gives us an inverse to the map $\G^{r_0}_{\Z}(X) \to \G^{r_0-1}_{\Z}(X)=\G^{s+1}_{\Z}(X)$. Therefore, $\G^{r_0}_{\Q}(X) \to \G^{r_0-1}_{\Q}(X)$ is an isomorphism, and $p^{r_0}$ is a biholomorphism.
\end{proof}

\begin{thrm}\label{no algebraic}
Let $X$ be a normal quasi-projective variety and $s>2$ a natural number. Suppose that one of the following holds:
\begin{itemize}
\item[(i)] $\alb^s \colon X \to \Alb^s(X)$ is dominant;
\item[(ii)] $\Alb^s(X)$ is definably biholomorphic to the definable analytification of a quasi-projective variety.
\end{itemize}
Then the map $p^r \colon \Alb^r(X) \to \Alb^{r-1}(X)$ is a principal $(\C^{\times})^k$-bundle if $r=2$ and is an isomorphism for $r>2$.
\end{thrm}

\begin{proof}

 The conditions \textit{(i)} and \textit{(ii)} are equivalent by Corollary \ref{when tower is algebraic}. Moreover, it follows from Corollary \ref{when tower is algebraic} that $\Alb^3(X)$ is algebraic and the diagram 
\[
\Alb^3(X) \xrightarrow{p^3} \Alb^2(X) \xrightarrow{p^2} \Alb^1(X)
\]
admits an algebraisation. Recall, that this means that there exist algebraic spaces (in our case, quasi-projective varieties) $\mathfrak{Y}^j, \ j=1, \ 2, \ 3$ and morphisms $\mathfrak{p}^j \colon \mathfrak{Y}^j \to \mathfrak{Y}^{j-1}$ such that $\Alb^j(X) =(\mathfrak{Y}^{j})^{\odef}$ and $\mathfrak{p}^j=(p^j)^{\odef}$. Moreover, $\mathfrak{p}^j$ are algebraic principal $\mathfrak{C}^j$-bundles for commutative connected algebraic groups $\mathfrak{C}^j, \ j=2, \ 3$. Definable complex Lie groups $C^j=(\mathfrak{C}^j)^{\odef}$ are Jacobians of some mixed Hodge structures.

We claim that the groups $\mathfrak{C}^2$ and $\mathfrak{C}^3$ are algebraic tori. By Proposition \ref{no unipotent} it is enough to show that their maximal abelian quotients are trivial.

Let us proof this claim for $j=3$ (the argument for $j=2$ is the same).

Let $\mathfrak{C}^3 \to \mathfrak{A}$ be the maximal abelian quotient and $\mathfrak{B}$ its kernel (see Theorem \ref{commutative groups}). Set $A:=\mathfrak{A}^{\odef}$ and $B:=\mathfrak{B}^{\odef}$. We get an exact sequence of definable commutative Lie groups
\[
0 \to B \to C^3 \to A \to 0.
\]

The action of $\mathfrak{C}^3$ on $\mathfrak{Y}^3$ restricts to a free algebraic action of $\mathfrak{B}$. There exists a quasi-projective quotient $\mathfrak{M}=\mathfrak{Y}^3/\mathfrak{B}$ and $M=\mathfrak{M}^{\odef}$ is the definable holomorphic quotient $\Alb^3(X)/B$. The map $p^3 \colon \Alb^3(X) \to \Alb^2(X)$ factorises as
\[
\xymatrix{
\Alb^3(X)\ar[dd]_{p^3} \ar[rd]^{u} & \\
& M \ar[ld]^{v}\\
\Alb^2(X) &\\
}
\]
 
 The manifold $M$ is the analytification of a smooth quasi-projective variety, hence it admis a K\"ahler metric. Therefore, $M \xrightarrow{v} \Alb^2(X)$ is a holomorphic principal $A$-bundle with K\"ahler total space. Blanchard's Theorem (Theorem \ref{blanchard}) implies that this bundle is topologically trivial and $\pi_1(M)=\pi_1(\Alb^2(X)) \times \pi_1(A)$. In particular, $\pi_1(M)$ is torsion-free and $2$-step nilpotent. The surjective homomorphism 
 \[
 \pi_1(X) \xrightarrow{\alb^3_{*}} \pi_1(\Alb^3(X)) \xrightarrow{u_*} \pi_1(M)
 \]
 factorises through $\pi_1(X) \xrightarrow{\alb^2_*} \pi_1(\Alb^2(X))=\G^2_{\Z}(X)$. We obtain a surjective homomorphism $\pi_1(\Alb^2(X)) \to \pi_1(M)=\pi_1(\Alb^2(X))\times \pi_1(A)$ which is left inverse to $v_* \colon \pi_1(M) \to \pi_1(\Alb^2(X))$. Since the latter map is surjective, we conclude that it is an isomorphism. Thus, $\pi_1(A) = 0$ and $A$ is trivial.

 Now, we are in the situation of Lemma \ref{toric tower topology}. We deduce that $\nilp(\Alb^3(X)) \le 2$ and the statement of the Theorem follows from Proposition \ref{3 implies r}.

\end{proof}

\begin{cor}\label{nilpotent cor}
Let $X$ be a normal quasi-projective variety. Suppose that  $\alb^s \colon X^{\an} \to \Alb^s(X)$ is dominant for some $s \ge 3$. Then the pro-unipotent completion of $\pi_1(X)$ is $2$-step nilpotent. In particular, if $\pi_1(X)$ is nilpotent, then $\nilp(X) \le 2$.
\end{cor}

\section{Conclusion}\label{conclusion}

We finish with some open questions motivated by our results.
\\

\textbf{1. Hodge structures and definable Lie groups.}
Theorem \ref{nil-jacobians o-min} implies that if $H$ is a graded polarisbale mixed $\Z$-Hodge structure with negative weihgts ($W_{-1}H=H$), then $J^0H$ carries a canonical structure of a definable commutative complex Lie group.

The restriction $W_{-1}H=H$ is always satisfies after replacing $H$ with an appropriate Tate twist. Indeed, $F^pH(n)=F^{p+n}H$, therefore, if $q$ is the maximal integer such that $F^qH \neq 0$, then all the non-zero parts of the Hodge filtration of $H(q+1)$ are in negative rank and $W_{-1}H=H$. 

The operation $H \mapsto J^0(H(q+1))$ defines a functor
\[
\mathcal{J} \colon \{\text{graded polarised } \Z-\text{MHS} \} \to \{ \text{ commutative definable complex Lie groups } \}
\]

\begin{q}\label{jacobian question}
How far is the functor $\mathcal{J}$ from being fully faithful? Is it true that $\mathcal{J}H \simeq \mathcal{J}H'$ if and only if $H$ is isomorphic to $H'$ up to a shift of gradings?
\end{q}

Shift of gradings might be still necessary, as can be seen in the following example. Let $H$ be a pure polarised $\Z$-Hodge structure of weight $-1$. Then $J^0H$ is an abelian variety. Consider a weight $-2$ Hodge structure $H'$ whith $H'_{\Z}=H_{\Z}$ and the pieces of the Hodge decomposition
\[
\begin{cases}
(H')^{-2,0}=H^{-1,0};\\
(H')^{-1,-1}=0;\\
(H')^{0,-2}=H^{0,-1}
\end{cases}
\]

Then $J^0H'$ is the same abelian variety and two biholomorphic abelian varieties are isomorphic algebraically (and hence, definably).

Observe that the answer to Question \ref{jacobian question} is automatically positive if $\mathcal{J}H$ and $\mathcal{J}H'$ are isomorphic as nil-Jacobians. This motivates the following question:

\begin{q}
Let $f \colon N_{\mathbf{W}} \to N_{\mathbf{W}'}$ be a a holomorphic definable map of nil-Jacobians. Is it true, that $f$ is a morphism of nil-Jacobians?
\end{q}

This leads to a more general philosophical question: to what extend does o-minimal geometry preserves the Hodge-theoretic information?
\\

\textbf{2. $\alb^2$ and the Malcev completion of $\pi_1$}.
As we mentioned in the introduction, Corollary \ref{nilpotent cor} can be viewed as a non-proper version of a result of Aguillar Aguillar and Campana \cite{AC} that says that if $\alb \colon X \to \Alb(X)$ is surjective and proper, then the Malcev completion of $\pi_1(X)$ is abelian. This result follows from a more general theorem due to the same authors that says that if $\alb \colon X \to \Alb(X)$ is proper, then the map to the normalisation of the Albanese image $ X \to \alb(X)^{\nu}$ induces isomorphism on Malcev completions of $\pi_1$'s.

\begin{q}
Let $X$ be a normal quasi-projective variety and $\alb^2(X)^{\nu}$ the normalisation of the image of the second Albanese map. Is it true that $X \to \alb^2(X)^{\nu}$ induces an isomorphism of Malcev completions $\G(\pi_1(X)) \xrightarrow{\sim} \G(\pi_1(\alb^2(X)^{\nu}))$?
\end{q}

We do not know any example, where the answer is negative, but perhaps this is only because of our lack of explicit understanding of higher Albanese maps.
\\

\textbf{3. Varieties with surjective higher Albanese map.} Finally, it would be interesting to find a geometric criterion for surjectivity of higher Albanese maps.

It is known, that if $X$ is a weakly special or $h$-special quasi-projective variety, then the classical Albanese map $\alb \colon X \to \Alb(X)$ is dominant (\cite[Lemma 11.5]{CDY}; see \textit{ibid.} for the definition of weakly special and $h$-special varieties). At the same time, it is conjectured that the fundamental group of such variety is virtually nilpotent (\cite[Conjecture 11.4]{CDY}). This Conjecture is known to hold if $\pi_1(X)$ is linear (\cite{CDY}) or if it is virtually solvable (\cite[Corollary 7.2]{Rog}).

We propose the following two conjectures.

\begin{conj}\label{special dominant}
Let $X$ be a normal quasi-projective variety which is either weakly special or $h$-special. Then $\alb^s \colon X \to \Alb^s(X)$ is dominant for every $s$.
\end{conj}

\begin{conj}\label{special nilpotent}
Let $X$ be a normal quasi-projective variety which is either weakly special or $h$-special. Then $\pi_1(X)$ is virtually at most two step nilpotent.
\end{conj}

Conjecture \ref{special dominant} implies Conjecture \ref{special nilpotent} by Corollary \ref{nilpotent cor}. At the same time, Conjecture \ref{special nilpotent} follows from \cite[Conjecture 11.4]{CDY} and Conjecture \ref{campana conjecture}.

\bibliography{references_ha_v2.bib}{}
\bibliographystyle{alpha}

\end{document}